\documentclass[11pt,twoside]{article}
\usepackage{graphicx}
\usepackage{amssymb,amsbsy,amsmath,amsfonts,amssymb,amscd}
\usepackage{amsthm}
\usepackage{graphicx,color}

\usepackage{xcolor}
\usepackage{hyperref}
\usepackage{bigints}
\hypersetup{
  colorlinks,
  citecolor=blue,
  linkcolor=red,
  urlcolor=blue}
\usepackage{epsfig,color}

\usepackage{dsfont}

\newcommand{\commentout}[1]{}

\newcommand {\e}  {\varepsilon}

\newcommand {\dt}   {\Delta t}
\newcommand {\dx}   {\Delta x}

\newcommand {\bv} { {\mathbf{v}} }
\newcommand {\bx} { {\mathbf{x}} }

\newcommand {\p}   {\partial}

\newcommand{\beq}{\begin{equation}}
\newcommand{\eeq}{\end{equation}}
\newcommand{\bea} {\begin{array}{rl}}
\newcommand{\eea} {\end{array}}
\newcommand{\bepa}{\left\{ \begin{array}{l}}
\newcommand{\eepa} {\end{array}\right.}
\newcommand*\diff{\mathop{}\!\mathrm{d}}
\newcommand{\brho}{\boldsymbol{\rho}}
\newcommand{\br}{\boldsymbol{r}}

\newtheorem{theorem}{Theorem}

\newtheorem{remark}[theorem]{Remark}
\newtheorem{proposition}[theorem]{Proposition}

\begin{document}

\title{Asymptotic preserving schemes for nonlinear kinetic equations leading to volume-exclusion chemotaxis in the diffusive limit}

\author{Gissell Estrada-Rodriguez\thanks{Mathematical Institute, University of Oxford, Oxford OX2 6GG, U.K. email: {estradarodri@maths.ox.ac.uk}}
\and Diane Peurichard\thanks{Sorbonne Universit\'e, Universit\'e de Paris, CNRS, Inria, Laboratoire Jacques-Louis Lions, F-75005 Paris, France. email: {}
}\and Xinran Ruan\thanks{School of Mathematical Sciences, Capital Normal University,  
100048 Beijing, China. email: {xinran.ruan@cnu.edu.cn}}}

\date{}

\maketitle

\begin{abstract}
In this work we first prove, by formal arguments, that the diffusion limit of nonlinear kinetic equations, where both the \textit{transport term and the turning operator are density-dependent}, leads to volume-exclusion chemotactic equations. We generalise an asymptotic preserving scheme for such nonlinear kinetic equations based on a micro-macro decomposition. By properly discretizing the nonlinear term implicitly-explicitly in an upwind manner, the scheme produces accurate approximations also in the case of strong chemosensitivity. We show, via detailed calculations, that the scheme presents the following properties: asymptotic preserving, positivity preserving and energy dissipation, which are essential for practical applications. We extend this scheme to two dimensional kinetic models and we  validate its efficiency by means of 1D and 2D numerical experiments of pattern formation in biological systems.
\end{abstract}

\section{Introduction}
Chemotaxis is the mechanism by which cells and organisms adapt their movement in response to a chemical stimulus present in their environment. This phenomenon has been observed in many biological systems \cite{Alt1980,Kaupp2012,Jin2008,Roussos2011}. 

The mathematical study of chemotaxis started from the seminal contributions of Patlak \cite{patlak1953random} and Keller and Segel \cite{keller1970initiation,keller1971model}, where the authors introduced the celebrated Patlak-Keller-Segel (PKS) model. This model was originally proposed for pattern formation in bacterial populations through an advection-diffusion system of two coupled parabolic equations describing the evolution of the cell density and the chemoattractant (see \cite{Hillen2009} for a review about Keller-Segel type models). 
The PKS original model has been modified by various authors, with the aim of improving its consistency with biological systems. One example is the volume-exclusion chemotactic system introduced by Hillen and Painter \cite{painter2002volume} to take into account the finite size of the cells and volume limitations. In such models the chemotactic sensitivity (i.e. the term leading to cell aggregation) depends on both the chemical concentration in the medium and the local cell density, thus, the population density directly modulates its own sensitivity response. The coupled system reads, in it's parabolic-elliptic form as
\begin{equation}
\begin{aligned}
    \partial_t\rho-\nabla\cdot\left(D_\rho(q(\rho)-\rho q'(\rho))\nabla\rho-\chi_0q(\rho)\rho\nabla c\right)&=h(\rho,c)\ , \ \qquad t\geq 0\ , \mathbf{x}\in \Omega\subset\mathbb{R}^n\ ,\\
    D_c\Delta c+g(\rho,c)& = 0\ .
    \label{eq: starting model equation}
\end{aligned}
\end{equation}
Here, $\rho(t,\mathbf{x})$ is the cell density, $c(t,\mathbf{x})$ is the chemoattractant concentration and  $q(\rho)$ is a function that describes the packing capacity of the cell aggregates.  The diffusion coefficients for the cells and chemoattractant are $D_\rho$ and $D_c$, respectively, and $\chi_0$ is the chemotactic sensitivity. The proliferation (or death) of the cells is described by $h(\rho,c)$ and the production and consumption of the chemoattractant is given by $g(\rho,c)$. Note that the classical PKS model is recovered by taking $q(\rho) = 1$ in \eqref{eq: starting model equation}.  It has been shown \cite{wrzosek2010volume,calvez2006volume,zheng2015boundedness} that such volume-exclusion effect prevents blow-ups in finite time compared to the model without density effects (with $q(\rho) = 1$). The volume-exclusion chemotactic equations have been widely studied in
the literature, from a modelling \cite{painter2002volume,wang2010chemotaxis}, analytic \cite{wrzosek2010volume,han2017pattern,ma2012stationary,dolak2005keller} and numerical perspectives \cite{ibrahim2014efficacy}, and
they have proven to be successful at describing aggregation phenomena \cite{almeida2020treatment,bubba2019chemotaxis}. 

A natural question that arises is whether the volume-exclusion Keller-Segel equation \eqref{eq: starting model equation} proposed in  \cite{painter2002volume} could be obtained as the diffusion limit of a kinetic `velocity jump' model \cite{othmer1988models}, giving insights into how the individual mechanisms at the cell level can lead to volume-exclusion effects at the population level. In this paper, inspired by the approach followed in \cite{othmer2002diffusion,othmer2000diffusion}, we show, by formal arguments, that the system \eqref{eq: starting model equation} can be obtained in the diffusion limit of a nonlinear kinetic equation, provided that both the transport term and the turning operator are density-dependent { (Theorem \ref{theorem1} in Section \ref{sec:macroscopic derivation})}. The corresponding kinetic `velocity jump' model we propose reads 
\begin{equation}
    \partial_t f+\mathbf{v}\cdot \nabla(F[\rho](t,\mathbf{x},\mathbf{v}) f)= \lambda q(\rho)(-f+\rho T(\mathbf{v},\rho,\nabla c))\ ,\label{eq: inital model without scaling}
\end{equation} where $f(t,\mathbf{x},\mathbf{v}) \geq 0$ is the phase space cell density,  $\mathbf{x}\in \Omega \subset \mathbb{R}^n$ denotes the position, $\mathbf{v} \in V \subset \mathbb{R}^n$ is the velocity where $V$ denotes the unit sphere  $V=\{\mathbf{x}\in\mathbb{R}^n:\ |\mathbf{x}|=1 \}$, and  $t \in \mathbb{R}^+$ the time.
Here, $\lambda$ is the constant turning rate, with $1/\lambda$ giving a measure of the mean run length between velocity jumps, $T(\mathbf{v},\rho,\nabla c)$ gives the probability of a velocity jump to velocity $\mathbf{v}$, which depends on the chemical concentration $c$ and the term $F[\rho](t,\mathbf{x},\mathbf{v})$ describes the anisotropic transport due to the density limited motion.

We then numerically investigate whether the relevant macroscopic volume-exclusion equation corresponds to the underlying physical system described by the kinetic equation { in the diffusion limit}. These methods are widely known as asymptotic preserving schemes (AP) \cite{filbet2010class,jin1999efficient} since they mimic the asymptotic behaviour of the kinetic equation when the scaling parameter approaches to zero and the mesh size and time steps are fixed. In this paper, we use a micro-macro decomposition of the unknown in the sense of  \cite{LemouMieussens} as detailed in Section \ref{sec:micro-macro}. {The finite difference discretization is explained in Section \ref{sec: implicit-explicit scheme}, where a proper implicit-explicit discretization of the nonlinear terms including the chemosensitivity term is applied to improve the efficiency and stability of our numerical scheme.} This decomposition of the solution of the kinetic equation is analogous to a Chapman-Enskog expansion in the case of the classical Boltzmann equation. It uses the properties of the ``collisional operator'', which in our formulation describes the run and tumble movement of the individual. For the kinetic counterpart of the classical PKS equations we refer to \cite{carrillo2013asymptotic}, where the authors used an odd-even splitting at the kinetic level and studied the behaviour of solutions (blow up). Other related works can be found in \cite{emako2016well,wang2019asymptotic,jin2011class}. 

\paragraph{The volume-exclusion PKS for modelling tumor growth} {In real-world applications,} mathematical models provide useful tools towards identifying links between phenomena observed at the macroscopic level and the underlying microscopic properties. Chemotactic models have been extensively used to describe glioblastoma (GBM) aggregates  \cite{Vital2011} and in particular, the volume-exclusion system \eqref{eq: starting model equation} was used in \cite{almeida2020treatment} to explain the mechanical changes at the cell level in these systems due to the presence of a chemical treatment. In this work, the cell's elasticity is modelled through the term $q(\rho)$ which incorporates the cell-cell interactions \cite{wang2007classical}. This function $q(\rho)$ can be explicitly written as $q(\rho)=1-(\rho/\bar{\rho})^\gamma$ where $\gamma$ is a parameter that depends on the concentration of the treatment and $\bar{\rho}$ is the maximum cell density in each aggregate. For $\gamma=1$ the cells are considered as solid particles while {$\gamma>1$  corresponds to semi-elastic particles} that can squeeze into empty spaces.

A drawback of this approach is that it is formulated directly at the continuous (PDE) level, therefore the interactions comprised in such models are mostly based on phenomenological considerations at the population level. However, neglecting the microstructural features causes these models to fail to predict the micromechanical behaviors, which makes the validation of these approaches {difficult} due to the lack of physical mechanisms to interpret the population macroscopic behavior. Establishing the link between macroscopic models and their microscopic counterparts, the central task of kinetic theory, can benefit macroscopic models gain in predictive character and become invaluable aids for experimental data analysis. {This paper aims to take a step in this direction, by providing an interpretation of the volume-filling chemotactic system as the diffusion limit of a kinetic 'velocity jump' model which gives a more direct interpretation of the PDE operators in terms of more fundamental characteristics of the motion.} {We then validate this limit using an AP numerical scheme.}

\paragraph{Outline of the paper} This paper is organized as follows. In Section \ref{sec:model} we {introduce the new version of the kinetic `velocity jump' model given by \eqref{eq: inital model without scaling}} and consider its diffusion limit using a Hilbert expansion method as in \cite{othmer2002diffusion} under appropriate scaling assumptions for the turning operator. We show that the limiting system has a dissipating free energy and derive an energy estimate used to validate the numerical method. {In Section \ref{sec:micro-macro}, we present the micro-macro decomposition of a slightly modified version of the kinetic model taking into account a proliferation term. Section \ref{sec: implicit-explicit scheme} is devoted to the design of an AP scheme and the study of its properties (positivity preserving, AP etc). Finally, we present the numerical experiments and draw conclusions in Section \ref{numerical results}}.

\section{Volume-exclusion kinetic equation: Macroscopic limit}
\label{sec:model}
In this section we analyse the diffusion limit of the following transport equation
\begin{equation*}
    \partial_t f+\mathbf{v}\cdot \nabla(F[\rho](t,\mathbf{x},\mathbf{v}) f)= \lambda q(\rho)(-f+\rho T(\mathbf{v},\rho,\nabla c))\ .
\end{equation*}
The turning kernel is supposed to be independent on the previous velocity of the jumping particle, and only dependent on the new velocity $\mathbf{v}$, the chemical concentration $c(t,\mathbf{x})$ and the density of particles $\rho(t,\mathbf{x})=\int_V f(t,\textbf{x},\mathbf{v})\diff \mathbf{v}$ near position $\mathbf{x}$.
\noindent The function $q(\rho)$ is the probability {for a cell to find} space at its neighbouring locations, and we assume that only a finite number of cells, $\bar{\rho}$, can be accommodated at any site. We will therefore consider functions $q(\rho)$ such that
$$
q(\bar{\rho}) = 0 \quad \text{and} \quad q(\rho) \geq 0 \quad \text{for all } \; 0\leq \rho \leq \bar{\rho}\ .
$$

We {suppose that the turning kernel $T$ integrates to 1 in the velocity variable},
$$
\int_V T(\textbf{v},\rho,\nabla c)\diff \mathbf{v}=1\ .
$$
In the volume-exclusion approach, following the lines of \cite{Hillen2009}, we assume that the probability of making a jump depends upon the availability of space into which it can move. To this aim, we suppose that cells can only make a turn in directions where space is available, and we choose the turning operator $T$ to be
\begin{equation*}
    T(\mathbf{v},\rho,\nabla c)=\tilde{c}(t,\mathbf{x})\psi(\mathbf{v},\nabla c) q(\rho(t,\mathbf{x}+ \mathbf{v}))\ ,
\end{equation*}
\noindent where $\tilde{c}(t,\mathbf{x})$ is a normalisation factor given by
\[
\tilde{c}(t,\mathbf{x})=\frac{1}{\displaystyle\int_{V} \psi(\mathbf{v},\nabla c) q(\rho(t,\mathbf{x}+ \mathbf{v}))\diff \mathbf{v} }\ .
\]
\noindent Note that under these assumptions, particles will only make a turn (i) if they are not already trapped in a high density region (where they stop) and (ii) only in directions where the density of cells is not already too large. 

In order to take into account density limited motion, we will also consider that cells are only transported to non-overcrowded regions, and choose for the transport term
\begin{equation*}
F[\rho](t,\mathbf{x},\mathbf{v}) = q\big(\rho(t,\mathbf{x}+\mathbf{v})\big)\ .
\end{equation*}
\subsection{Diffusion scaling} Following the lines of \cite{Hillen2009}, we aim to obtain a macroscopic limit by choosing space and time scales on which {there are many velocity jumps in one order of time, but small net displacements on this time scale}. To this aim, we define the dimensionless velocity, space and time variables as
$$
\mathbf{u} = \frac{\mathbf{v}}{s}\ , \quad \mathbf{\xi} = \frac{\mathbf{x}}{L}\ , \quad \tau = \frac{t}{\sigma}\ ,
$$
\noindent where $s$ is the characteristic speed, $L$ the characteristic length scale and $\sigma$ yet to be determined. Equation \eqref{eq: inital model without scaling} now writes,
\begin{equation}
    \frac{1}{\tau} \partial_\tau \tilde{f}+ \frac{s}{L} \mathbf{u} \cdot \nabla_\xi (F[\tilde{\rho}](\tau,\mathbf{\xi},\mathbf{u}) \tilde{f})= \lambda q(\tilde{\rho})(-\tilde{f}+\tilde{\rho} T(\mathbf{u},\tilde{\rho},\nabla c))\ ,\label{eq: scaled model}
\end{equation}
\noindent where $\tilde{f}(\mathbf{\xi},\mathbf{u},\tau) = f(\sigma \tau,L \mathbf{\xi} ,s \mathbf{u},t)$, $\tilde{\rho}(\tau,\mathbf{\xi}) = \rho(\sigma t,L \mathbf{\xi})$ and therefore 
$$
F[\tilde{\rho}](\tau,\mathbf{\xi},\mathbf{u}) = q(\tilde{\rho}(\tau, \mathbf{\xi} + \frac{s}{L} \mathbf{u})\ .
$$
We estimate the diffusion coefficient as the product of the characteristic speed times the distance traveled between velocity jumps, giving $D \approx \mathcal{O}(\frac{s^2}{\lambda})$, and we deduce the characteristic diffusion time on the length scale $L$ by $\tau_{diff} \approx \frac{L^2 \lambda}{s^2}$. The characteristic drift time is defined by $\tau_{drift} = \frac{L}{s}$ and we assume that the space scale is such that $\tau_{run} = \frac{1}{\lambda} \ll \tau_{drift} \ll \tau_{diff}$. We therefore introduce a small parameter $\varepsilon \ll 1$ and ensure that $\tau_{run} = \mathcal{O}(1)$, $\tau_{drift} = \mathcal{O}(\frac{1}{\varepsilon})$ and $\tau_{diff} = \mathcal{O}(\frac{1}{\varepsilon^2})$ by choosing the time and space scales to be $L \approx \mathcal{O}(\frac{s}{\varepsilon})$ and $\sigma = \tau_{diff}$. Without loss of generality, we set $\lambda = 1$ and now equation \eqref{eq: scaled model} becomes (dropping the tildes {and, with a slight abuse of notations, going back to $\mathbf{x},\mathbf{v}$ and $t$ for the dimensionless quantities $\xi, \mathbf{u}$ and $\tau$),
\begin{equation}
    \varepsilon^2 \partial_t f^\varepsilon+\varepsilon\mathbf{v}\cdot \nabla(F_\varepsilon[\rho](t,\mathbf{x},\mathbf{v}) f^\varepsilon) = q(\rho)(-f^\varepsilon+\rho T_\varepsilon(\mathbf{v},\rho,\nabla c))\ , \label{eq: inital model with scaling}
\end{equation}
\noindent where
\begin{equation}\label{TepsFeps}
    T_\varepsilon(\mathbf{v},\rho,\nabla c)=\frac{\psi_\varepsilon(\mathbf{v},\nabla c) q(\rho(t,\mathbf{x}+\varepsilon \mathbf{v}))}{\displaystyle\int_{V}\psi_\varepsilon(\mathbf{v},\nabla c) q(\rho(t,\mathbf{x}+\varepsilon \mathbf{v})) \diff \mathbf{v}}\ ,\qquad
    F_\varepsilon[\rho](t,\mathbf{x},\mathbf{v}) = q(\rho(t,\mathbf{x}+\varepsilon \mathbf{v}))\ .
\end{equation}
} 

{We will consider that the dependency of the turning operator on the chemical gradient $\nabla c$ happens as a perturbation of magnitude $\varepsilon$ in the following way
\begin{equation}\label{expansion_psi}
    \psi_\varepsilon(\mathbf{v},\nabla c)=\psi_0(\mathbf{v})+\varepsilon\psi_1(\mathbf{v},\nabla c)\ .
\end{equation}
Moreover, we will consider that $\langle\psi_\varepsilon \rangle = \int_V \psi_\varepsilon (\mathbf{v},\nabla c) \diff \mathbf{v} = 1$, {where $\psi_\varepsilon(\mathbf{v},\nabla c)$ is non-negative and decreasing in $\nabla c$, which means that cells are less likely to tumble when the chemical gradient increases.} In order to recover the volume-exclusion Keller-Segel equation, we will assume that $\psi_0$ is radially symmetric and that the perturbation $\psi_1(\mathbf{v},\nabla c)$ depends linearly on the chemical gradient $\nabla c$. These assumptions lead to the following hypotheses: 
\paragraph{Hypothesis 1}
\begin{equation}\tag{H1}
    \int_V\psi_0(\mathbf{v})\diff \mathbf{v}=1\ , \quad \int_V\psi_1(\mathbf{v},\nabla c)\diff \mathbf{v}=0\ ,\label{hypo1}
\end{equation}
\paragraph{Hypothesis 2}
\begin{equation}\tag{H2}
\int_V \mathbf{v}\psi_0(\mathbf{v})\diff \mathbf{v}=0\ , \quad \psi_1(\mathbf{v},\nabla c) = \phi(\mathbf{v}) \cdot \nabla c\ . \label{hypo2}
\end{equation}
}

\subsection{Macroscopic model}\label{sec:macroscopic derivation} 

In this section we prove the following theorem:
\begin{theorem}{}\label{theorem1}(formal)
The limit $\varepsilon \rightarrow 0$ of $f^\varepsilon$ solving \eqref{eq: inital model with scaling}-\eqref{expansion_psi} together with Hypotheses \ref{hypo1} and \ref{hypo2} is $f^0 = \rho(t,\mathbf{x})\psi_0(\mathbf{v})$, where $\rho$ solves
\begin{equation}
    \partial_t\rho-\nabla\cdot\left(D_0\big(q(\rho)-\rho q'(\rho) \big)\nabla \rho-\beta\rho q(\rho)\nabla c \right)=0\ ,\label{eq: macroscopic equation1}
\end{equation}
with the diffusion coefficient $D_0$ and the chemotactic sensitivity parameter $\beta$ given by
\begin{equation}D_0=\langle(\mathbf{v}\otimes \mathbf{v})\psi_0(\mathbf{v})\rangle\qquad \textnormal{and}\qquad \beta=\langle\mathbf{v}\otimes\phi(\mathbf{v})\rangle\ . \label{eq: macroscopic coefficients}
\end{equation}
\end{theorem}
\begin{proof}
We first expand the transport quantity $F_\varepsilon$ and the turning operator $T_\varepsilon$ given by \eqref{TepsFeps}. For $\varepsilon \ll 1$, we have
\begin{equation}
  F_\varepsilon[\rho] =  q(\rho(t,\mathbf{x}+\varepsilon \mathbf{v}))=q(\rho(t,\mathbf{x}))+\varepsilon q'(\rho(t,\mathbf{x})) \mathbf{v}\cdot \nabla\rho(t,\mathbf{x})+\mathcal{O}(\varepsilon^2)\ ,\label{eq: F scaled}
\end{equation}
where $q'(\rho)=\frac{\diff q}{\diff\rho}$. Introducing this expansion in the expression for $T_\varepsilon$, we write
\begin{equation*}
    \tilde{c}(t,\mathbf{x})=\frac{1}{q(\rho) \langle\psi_\varepsilon\rangle}\Bigl(1-\varepsilon\frac{q'(\rho) }{q(\rho) } \nabla\rho \cdot \frac{\langle \mathbf{v}\psi_\varepsilon\rangle}{\langle\psi_\varepsilon\rangle} \Bigr)\ ,
\end{equation*}
\noindent where $\langle\psi_\varepsilon\rangle =\int_V\psi_\varepsilon(\mathbf{v},\nabla c)\diff \mathbf{v}$ and $\langle \mathbf{v}\psi_\varepsilon\rangle = \int_V \mathbf{v} \psi_\varepsilon(\mathbf{v},\nabla c)\diff \mathbf{v}$ denote the first and second moments of $\psi_\varepsilon$. Finally, we obtain 
{\begin{equation*}
  T_\varepsilon (\mathbf{v},\rho,\nabla c)=\frac{\psi_\varepsilon(\mathbf{v},\nabla c)}{\langle\psi_\varepsilon\rangle} + \varepsilon\frac{q'(\rho)}{q(\rho)}\frac{\psi_\varepsilon(\mathbf{v},\nabla c)}{\langle\psi_\varepsilon\rangle} \nabla\rho \cdot \left(\mathbf{v} - \frac{\langle \mathbf{v}\psi_\varepsilon\rangle}{\langle\psi_\varepsilon\rangle}\right) +\mathcal{O}(\varepsilon^2)\ .
\end{equation*}}
Note that the error terms contained in $\mathcal{O}(\varepsilon^2)$ integrate to 0 in the velocity variable. Using \eqref{expansion_psi}, the turning kernel $T_\varepsilon$ writes
\begin{align}
    T_\varepsilon(\mathbf{v},\rho,\nabla c)=&\psi_0(\mathbf{v})+\varepsilon\Bigl(  \psi_1(\mathbf{v},\nabla c)+\frac{q'(\rho)}{q(\rho)}(\mathbf{v}\cdot\nabla\rho)\psi_0(\mathbf{v})\Bigr) + {\mathcal{O}(\varepsilon^2)} \ ,\label{eq: turning operator scaled}
\end{align}
{where the $\mathcal{O}(\varepsilon^2)$-term {is such that} $\langle\mathcal{O}(\varepsilon^2)\rangle=0$. We also note that using Hypotheses \ref{hypo1} and \ref{hypo2} we have $\langle T_\varepsilon\rangle=1$ which describes the conservation of individuals during the velocity reorientation.}
We now consider a second order regular expansion of $f^\e$ in $\varepsilon$, 
\begin{equation*}
    f^\e(t,\mathbf{x},\mathbf{v})=f^0(t,\mathbf{x},\mathbf{v})+\varepsilon f^1(t,\mathbf{x},\mathbf{v})+\varepsilon^2 f^2(t,\mathbf{x},\mathbf{v})+\mathcal{O}(\varepsilon^3)\ , 
\end{equation*}
where $\int_V f^\e(t,\mathbf{x},\mathbf{v})\diff \mathbf{v} = \int_V f^0(t,\mathbf{x},\mathbf{v})\diff \mathbf{v} = \rho(t,\mathbf{x})$, therefore $\int_V f^i(t,\mathbf{x},\mathbf{v})\diff \mathbf{v}=0$, $\forall i\geq 1$.
Introducing this ansatz in \eqref{eq: inital model with scaling}, we obtain
{\begin{align*}
    \varepsilon^2\partial_t &f^0+\varepsilon\mathbf{v}\cdot\nabla\big(q(\rho)\big(f^0+\varepsilon f^1\big)  + \varepsilon q'(\rho)(\mathbf{v}\cdot\nabla\rho)f^0  \big) \\
    =&q(\rho)\Bigl[-f^0+\rho \psi_0(\mathbf{v})+\varepsilon \Bigl((-f^1 + \rho \Bigl(  \psi_1(\mathbf{v},\nabla c)+\frac{q'(\rho)}{q(\rho)}(\mathbf{v}\cdot\nabla\rho)\psi_0(\mathbf{v})\Bigr)\Bigr){+\mathcal{O}(\varepsilon^2)}
      \Bigr]\ .
\end{align*}}
\noindent Identifying the different equations in powers of $\varepsilon$, we obtain
\begin{align}
    \varepsilon^0: & \ f^0(t,\mathbf{x},\mathbf{v})=\rho(t,\mathbf{x})\psi_0(\mathbf{v})\ ,\label{eq: epsilon0}\\
    \varepsilon^1: & \ \mathbf{v}\cdot\nabla(q(\rho)f^0)=q(\rho)\left(-f^1+\rho\left(\psi_1(\mathbf{v},\nabla c)+\frac{q'(\rho)}{q(\rho)}(\mathbf{v}\cdot\nabla\rho)\psi_0(\mathbf{v})\right)\right)\ ,\label{eq: epsilon1}\\
    \varepsilon^2: & \ \partial_tf^0+\mathbf{v}\cdot\nabla \bigg(q(\rho)f^1 + q'(\rho)(\mathbf{v}\cdot\nabla\rho)f^0 \bigg)={\mathcal{O}(\varepsilon^2)} \ \label{eq: epsilon2}.\
\end{align}
Integrating \eqref{eq: epsilon2} with respect to $\mathbf{v}\in V$ and noticing that the right hand terms integrate to zero using Hypothesis \ref{hypo1}, we get
\begin{equation}
    \partial_t\rho+\nabla\cdot \left(q(\rho)\langle\mathbf{v}f^1\rangle +\rho q'(\rho)\langle(\mathbf{v}\otimes \mathbf{v}) \psi_0\rangle \nabla\rho\right)=0\ .\label{eq: macroscopic equation}
\end{equation}
Next, after replacing $f^0$ by its expression (Eq. \eqref{eq: epsilon0}), we multiply \eqref{eq: epsilon1} by $\mathbf{v}$ and we integrate again with respect to $\mathbf{v}$ to obtain
\begin{align}
    q(\rho)\langle\mathbf{v}f^1\rangle&=- \nabla \cdot \big(\rho q(\rho) \langle(\mathbf{v}\otimes \mathbf{v})\psi_0\rangle \big)+\rho q(\rho)\langle \mathbf{v}\psi_1\rangle  +\rho q'(\rho) \langle(\mathbf{v}\otimes \mathbf{v})\psi_0\rangle  \nabla\rho\ .\label{eq: mean direction}
\end{align}
Substituting \eqref{eq: mean direction} into \eqref{eq: macroscopic equation} we get
\begin{align*}
    \partial_t\rho+\nabla\cdot\Bigl[-\langle(\mathbf{v}\otimes \mathbf{v}) \psi_0\rangle \nabla(\rho q(\rho))&+\rho q(\rho)\langle\mathbf{v} \psi_1\rangle \nonumber\\
    &+2 \rho q'(\rho)\langle(\mathbf{v}\otimes \mathbf{v})\psi_0\rangle\nabla\rho  \Bigr]=0\ . 
\end{align*}
Noting that $\nabla(q(\rho)\rho)=q'(\rho)\rho\nabla\rho+q(\rho)\nabla\rho$ and using Hypothesis \ref{hypo2} for the perturbation $\psi_1$, we finally arrive to the volume-exclusion Keller-Segel model \eqref{eq: macroscopic equation1} together with \eqref{eq: macroscopic coefficients}.
\end{proof}

This macroscopic equation describes the
volume-exclusion chemotactic motion associated with the so-called \emph{squeezing probability} $q(\rho)$. Depending on the choice of this function we can consider the cells either as solid blocks, for the case $q(\rho)=1-\frac{\rho}{\bar{\rho}}$, where $\bar{\rho}$ is the maximum cell density in each aggregate, or as semi-elastic entities for $q(\rho)=1-\left(\frac{\rho}{\bar{\rho}} \right)^\gamma$ (see \cite{wang2007classical}). {In Appendix \ref{subsec:energy}, we show that equation \eqref{eq: macroscopic equation1} admits an energy functional decreasing in time}.

\section{Micro-macro decomposition}\label{sec:micro-macro}
In this section, we will consider a more general volume-exclusion kinetic model by including a proliferation term with the appropriate scaling
\begin{equation*}
	\e^2 \p_t f + \e \bv \cdot \nabla(F_\varepsilon(\rho)f) = q(\rho) \left(-f + \rho T_\varepsilon(\bv, \rho, \nabla c)\right) + \e^2 r_0f\left(1-\rho/\rho_{\textnormal{max}}\right)_+\ ,
\end{equation*}
where  $\rho_{\textnormal{max}}$ is the carrying capacity, and the transport quantity $F_\varepsilon(\rho)$ and the turning operator $T_\varepsilon$ are defined in \eqref{eq: F scaled} and \eqref{eq: turning operator scaled}, respectively. 
{As we are interested in the limit of small $\varepsilon$, we will consider from now on a slightly modified version of the kinetic model by truncating $F_\varepsilon$ and $T_\varepsilon$ to the first order and solving the approximate equation 
\beq
\label{eq:kinetic_approx}
	\e^2 \p_t f + \e \bv \cdot \nabla(\tilde{F}_\varepsilon(\rho)f) = q(\rho) \left(-f + \rho \tilde{T}_\varepsilon(\bv, \rho, \nabla c)\right) + \e^2 r_0f\left(1-\rho/\rho_{\textnormal{max}}\right)_+\ ,
\eeq
where
\begin{align*}
    &\tilde{F}_\varepsilon = q(\rho(t,\bx)) + \varepsilon q'(\rho(t,\bx))\bv\cdot\nabla\rho(t,\bx)\ ,\\
    &\tilde{T}_\varepsilon = \psi_0(\bv) + \varepsilon\left(\psi_1(\bv,\nabla c) + \frac{q'(\rho)}{q(\rho)}(\bv\cdot\nabla\rho)\psi_0(\bv)\right)\ .
\end{align*}}
With a similar argument as in the proof of Theorem \ref{theorem1}, we can show that the approximate generalized kinetic model \eqref{eq:kinetic_approx} converges to the following macroscopic limit as $\varepsilon\to 0$
\beq
\label{eq:limit}
	\p_t \rho - \nabla \cdot \left[D_0\left(q(\rho)-\rho q'(\rho) \right)\nabla \rho-\langle\bv\psi_1\rangle\rho q(\rho)\right] = r_0 \rho\left(1 - \rho/\rho_{\rm max}\right)_+\ ,
\eeq
where $D_0 = \langle(\bv\otimes\bv)\psi_0\rangle$ and $\langle \mathbf{v}\psi_1 \rangle=\langle \mathbf{v}\phi(\mathbf{v})\rangle\nabla c=\beta\nabla c$.

To design an asymptotic preserving scheme which automatically preserves the macroscopic limit, a micro-macro formulation needs to be derived. 
We decompose the solution $f(t,\bx,\bv)$ as 
\begin{equation}\label{def:f_decompose}
    f(t,\bx,\bv) = \rho(t,\bx) \psi_0(\bv) + \varepsilon g(t,\bx,\bv)\ .
\end{equation}
We note that $\langle g\rangle = 0$ and the transport term is given by
\begin{align*}
    \bv \cdot \nabla(\tilde{F}_\varepsilon(\rho)f) & = (\bv\psi_0) \cdot \nabla(\rho q(\rho)) + \varepsilon \bv \cdot \nabla \left[q(\rho)g + (\bv\cdot\nabla\rho)\rho q'(\rho)\psi_0\right]\\
    & = (\bv\psi_0) \cdot \nabla(\rho q(\rho)) + \varepsilon \bv \cdot \nabla (q(\rho)g) + \varepsilon \nabla \cdot (\rho q'(\rho) (\bv\otimes\bv)\psi_0\nabla\rho)\ .
\end{align*}
Substituting the micro-macro decomposition of $f(t, \bx, \bv)$ given by \eqref{def:f_decompose} into the generalized volume-exclusion kinetic model \eqref{eq:kinetic_approx}, integrating over $\bv$ and noticing the fact that $\langle\bv\psi_0\rangle = 0, \, \langle T_\varepsilon \rangle=1 \, \textnormal{and}\, \langle f\rangle = \rho$, we have the equation for the macroscopic quantity $\rho(t, \bx)$
\begin{equation*}
    \p_t \rho +  \langle\bv\cdot \nabla (q(\rho) g)\rangle + \nabla \cdot (q'(\rho)\rho D_0\nabla\rho) =  r_0 \rho\left(1 - \rho/\rho_{\rm max}\right)_+\ .
\end{equation*}
To get the equation for $g$, we use the projection technique. 
For simplicity of notations, we introduce the projection operator $\Pi$ defined as 
\begin{equation*}
    \Pi f(t,\mathbf{v},\mathbf{x}) = \langle f(t,\mathbf{v},\mathbf{x})\rangle \psi_0(\mathbf{v})\ .
\end{equation*}
It is easy to check that, {for $I$ the identity operator,} 
\begin{align*}
    & (I - \Pi) f = \varepsilon(I - \Pi) g = \varepsilon g\ ,\\
    & (I-\Pi) (\bv \cdot \nabla (\tilde{F}_\varepsilon f))= (\bv\psi_0) \cdot \nabla(q(\rho)\rho) + \varepsilon(I-\Pi)\left[\bv \cdot \nabla (q(\rho)g) + \nabla \cdot (\rho q'(\rho)(\bv\otimes\bv)\psi_0\nabla\rho)\right]\ ,\\
    & (I - \Pi) \left(q(\rho)(-f + \rho \tilde{T}_\varepsilon)\right) = \varepsilon\left(\rho q(\rho)\psi_1 + \rho q'(\rho)(\bv\cdot\nabla\rho)\psi_0 - q(\rho) g\right)\ ,\\
    & (I- \Pi) \left(r_0 f \left(1 - \rho/\rho_{\rm max}\right)\right) = \varepsilon r_0 {g} \left(1 - \rho/\rho_{\rm max}\right)_+ \ .
\end{align*}
Finally taking the operator $I-\Pi$ into equation \eqref{eq:kinetic_approx}, we get
\begin{align*}
    &\p_t g + \frac{1}{\varepsilon} (I-\Pi)\left[\bv \cdot \nabla (q(\rho)g) + \nabla \cdot (\rho q'(\rho) (\bv\otimes\bv)\nabla\rho\psi_0)\right] \\
    = &\frac{1}{\varepsilon^2}\left[- q(\rho)(\bv\cdot\nabla\rho)\psi_0+\rho q(\rho)\psi_1  - q(\rho) g\right] + r_0 g \left(1 - \rho/\rho_{\rm max}\right)_+\ .
\end{align*}

As a summary, by decomposing $f$ as \eqref{def:f_decompose}, the following micro-macro formulation of the system is derived
\begin{align}\label{eq:macro_micro}
\begin{cases}
	&\p_t \rho +  \langle\bv\cdot \nabla (q(\rho) g)\rangle + { D_0 \nabla \cdot (\rho  q'(\rho)\nabla \rho)} =  r_0 \rho\left(1 - \rho/\rho_{\rm max}\right)_+\ , \\
	&\p_t g + \frac{1}{\varepsilon} (I-\Pi) K_\varepsilon = \frac{1}{\varepsilon^2}S_\varepsilon + r_0 g \left(1 - \rho/\rho_{\rm max}\right)_+\ ,\\
	&  \Delta c + \rho - c=0\ ,
\end{cases}
\end{align}
where $D_0 = \langle(\bv\otimes\bv)\psi_0\rangle$, and 
\begin{align*}
    & K_\varepsilon = \bv \cdot \nabla (q(\rho)g) + \nabla \cdot (\rho \psi_0(\bv\otimes\bv)\nabla q(\rho))\ , \\
    & S_\varepsilon = - q(\rho)(\bv\cdot\nabla\rho)\psi_0+\rho q(\rho)\psi_1  - q(\rho) g\ .
\end{align*}
With a sufficiently large domain, we expect $f$ as well as $c$ will almost reach a steady state at the boundary. 

Here we formally show that the micro-macro formulation derived {recovers} the macroscopic limit as $\varepsilon\to 0$. 
In fact, the leading order term in the equation of $g$ shows that 
\begin{equation*}
    q(\rho)g = - q(\rho)(\bv\cdot\nabla\rho)\psi_0+\rho q(\rho)\psi_1
\end{equation*}
in the limit $\varepsilon\to 0$. 
Therefore, 
\begin{equation*}
    \langle\bv\cdot \nabla (q(\rho) g)\rangle = \nabla \cdot \left[-D_0{q(\rho)\nabla\rho }+\langle\bv \psi_1\rangle\rho q(\rho)\right]\ .
\end{equation*}
Substituting it into the equation of $\rho$ in \eqref{eq:macro_micro}, we get 
\begin{equation}
\label{eq:macro_proliferation}
\begin{cases}
    & \p_t \rho - \nabla \cdot \left[D_0({q(\rho)}-{\rho q'(\rho)})\nabla\rho - \langle\bv \psi_1\rangle\rho q(\rho)\right] = r_0 \rho\left(1 - \rho/\rho_{\rm max}\right)_+\ , \\
    &  \Delta c + \rho - c=0\ .
\end{cases}
\end{equation}

{In Appendix \ref{subsec:energy} (see Remark \ref{rem 6}), we show that equation \eqref{eq:macro_proliferation} admits an energy functional that decreases in time providing we have the condition $\rho_{\max} (1 + \frac{1}{\bar{\rho}}) \leq 1$}. 

\section{An asymptotic preserving finite difference scheme}\label{sec: implicit-explicit scheme}
By discretizing the system  \eqref{eq:macro_micro} via finite difference method, we will get an asymptotic preserving scheme, which will be formally proven later in the section. 
To describe the fully discretized scheme, we consider the 1D case for simplicity, i.e. {$x, \,v \in [x_{\min}, x_{\max}]\times [v_{\min}, v_{\max}]$ with periodic boundary conditions in the $x$-direction and zero boundary conditions in the $v$-direction.}
The generalization to
the multidimensional case with tensor product grids is straightforward and is included in  Appendix \ref{sec:appendix-FD-2D}. We use a uniformly distributed mesh with 
\begin{equation*}
    t_n = n \Delta t\ , \quad x_j = j \dx\ , \quad v_k = k \Delta v\  ,
\end{equation*}
where $n\ge0$, $j=0,1,\cdots, N_x-1$, $k=0,1,\cdots, N_v$, $N_x = (x_{\max} - x_{\min}) / \Delta x$ and $N_v = (v_{\max} - v_{\min}) / \Delta v$. 
For the unknown functions $\rho(t,x)$ and $g(t,x,v)$, we compute its approximations $\rho_j^n$ and $g^n_{j+\frac{1}{2},k}$ with
\begin{equation*}
  \rho_j^n\approx\rho(t_n,x_j)\ ,\quad\textnormal{and} \quad  
  g^n_{j+\frac{1}{2},k}\approx g(t_n,x_{j+\frac{1}{2}},v_k)\ . 
\end{equation*}
Note that for the convenience of numerical computation, the approximation of the density function $\rho(t,x)$ is computed on grid points $x_j$, while the perturbation function $g(t,x,v)$ is computed on half grid points $x_{j+\frac12}$. 
Approximations of the density function $\rho(t,x)$ at half grid points can then be  efficiently computed by interpolation. 
To be more precise,  $\rho(t_n,x_{j+\frac12}) \approx \bar{\rho}_{j+\frac12}^n:=(\rho_j^n + \rho_{j+1}^n)/2$. 

For simplicity of notations, we further introduce 
the standard finite difference operators $\delta_t^+$ and $\delta_x$, which are numerical approximations of $\partial_t$ and  $\partial_x$, respectively, and defined as 
\begin{equation*}
    \delta_t^+ \rho_j^n = \frac{\rho_j^{n+1} - \rho_j^n}{\Delta t}\ , \, 
    \delta_x \rho_{j+\frac12}^n = \frac{\rho_{j+1}^n - \rho_{j}^n}{\Delta x}\ , \,
    \delta_x g_{j,k}^n = \frac{g_{j+\frac12,k}^n - g_{j-\frac12,k}^n}{\Delta x}\ . 
\end{equation*}
The composite of two operators $\delta_x$, which is denoted as $\delta_x^2$, is then defined to be 
\begin{equation*}
    \delta_x^2 \rho_j^n = \frac{\delta_x \rho_{j+\frac12}^n - \delta_x \rho_{j-\frac12}^n}{\Delta x}= \frac{\rho_{j+1}^n - 2\rho_{j}^n + \rho_{j-1}^n}{(\Delta x)^2}\ , 
\end{equation*}
which is the numerical approximation of $\partial_x^2$.
The standard finite difference operators can be applied to a multiplication of two functions. 
As an example, we define
\begin{equation*}
    \delta_x (q(\bar{\rho}_*^n) g_{*,k}^n)_j = \frac{q(\bar{\rho}_{j+\frac12}^n) g_{j+\frac12,k}^n - q(\bar{\rho}_{j-\frac12}^n) g_{j-\frac12, k}^n}{\Delta x}\ ,
\end{equation*}
where we use * to denote the positions where the sub-index $j$ is substituted. 
Another important notation to be introduced is $\langle\cdot\rangle_h$, which is defined as  
\begin{equation*}
     \langle\eta^n_{j,k}\rangle_h := \Delta v \sum_k \eta_{j,k}^n\ ,
\end{equation*}
where $\eta_{j,k}^n\approx\eta(t_n, x_j, v_k)$ for some general function $\eta(t,x,v)$. 
Obviously, $\langle\eta^n_{j,k}\rangle_h$ is the finite difference approximation of $\langle\eta(t_n,x_j,v)\rangle := \int_V \, \eta(t_n,x_j,v) \, dv$.
Then $D_0:=\langle v^2\psi_0(v)\rangle$ can be approximated by 
\begin{equation}\label{def:Dh}
    D_h := \langle v_k^2\psi_{0}(v_k)\rangle_h\ .
\end{equation}

\noindent Finally, to better approximate $q(\rho)\rho$ at $x=x_{j+\frac12}$ and $t=t_n$, 
we introduce the notation $\Phi^{n_1,n}_{j+\frac12}$, which is defined as 
\beq \label{def:Phi}
	\Phi^{n_1,n}_{j+\frac12} = 
	\begin{cases}
		\rho_j^{n_1} q(\rho_{j+1}^n)\ , \quad & \text{ if } \delta_x c_{j+\frac12}^n \ge 0\ , \\
		\rho_{j+1}^{n_1} q(\rho_j^n)\ , \quad & \text{ if } \delta_x c_{j+\frac12}^n <0\ ,
	\end{cases}
\eeq
where $n_1 = n$ or $n+1$.
As shown in \cite{de2018energy}, $\Phi^{n_1,n}_{j+\frac12}$ approximates $q(\rho)\rho$ at $t=t_n$ and  $x=x_{j+\frac12}$  
in an upwind manner and thus helps improve the stability of the numerical scheme. 

With the notations defined, the system  \eqref{eq:macro_micro} can be discretized as 
\begin{align}
\label{eq:scheme}
    \begin{cases}
    & \delta_t^+ \rho_j^n +  
    \langle v_k\delta_x(q(\bar{\rho}_*^n) g_{*,k}^{n+1})_j\rangle_h
     + D_h \delta_x(\bar{\rho}_*^n q'(\bar{\rho}_*^n)\delta_x \rho_*^{n+1})_j 
    = r_0 \rho_j^{n}\left(1 - \frac{\rho_j^n}{\rho_{\rm max}}\right)_+\ ,\\
    &\delta_t^+ g_{j+\frac12,k}^n + \frac{1}{\varepsilon}(I-\Pi_h) K_{j+\frac12,k}^n 
    = \frac{1}{\varepsilon^2}S_{j+\frac12,k}^{n, n+1} + r_0 g_{j+\frac12,k}^n\left(1 - \frac{\rho_{j+\frac12}^n}{\rho_{\rm max}}\right)_+\ ,\\
    & \delta_x^2 c_j^{n+1} + \rho_j^{n+1} - c_j^{n+1}=0\ ,
\end{cases}
\end{align}
where  $\Pi_h$ is the discrete projection operator defined as 
$\Pi_h \eta_{j,k}^n = \langle\eta_{j,k}^n\rangle_h \psi_{0}(v_k)$ for some general function $\eta(t,x,v)$, and 
\begin{align*}
    & K_{j+\frac12,k}^n = v_k^+\delta_x(q(\bar{\rho}_*^n) g_{*,k}^n)_j - v_k^-\delta_x(q(\bar{\rho}_*^n) g_{*,k}^n)_{j+1}
    + v_k^2\psi_{0}(v_k)\delta_x(\bar{\rho}_*^n q'(\bar{\rho}_*^n)\delta_x\rho_*^n)_{j+\frac12}\ , \\
    & S_{j+\frac12,k}^{n, n+1} = -v_k\psi_{0}(v_k)q(\bar{\rho}_{j+\frac12}^n)  \delta_x \rho_{j+\frac12}^{n+1} + \psi_{1}(v_k, \delta_x c_{j+\frac12}^n)
    \Phi^{n+1,n}_{j+\frac12}
    -q(\bar{\rho}_{j+\frac12}^n) g_{j+\frac12,k}^{n+1}\ ,
\end{align*}
where $v^+ = \max\{v, 0 \}$ and $v^- = \max\{-v, 0\}$.

Following the idea in \cite{LemouMieussens}, the scheme \eqref{eq:scheme} can be solved efficiently. 
Instead of solving the system \eqref{eq:scheme} directly, where all densities $\rho_{j}^{n+1}$ and perturbations $g_{j+\frac{1}{2},k}^{n+1}$ are coupled so that a large linear system needs to be inverted, 
we introduce $\tilde{g}_{j+\frac{1}{2},k}^{n+1}$, which satisfies 
\begin{align}     \label{eq:tilde_g}
    &\frac{\tilde{g}_{j+\frac12,k}^{n+1} - g_{j+\frac12,k}^n}{\Delta t} + \frac{
	1}{\e}(I - \Pi)K_{j+\frac12,k}^n
    = \frac{
	1}{\e^2} \tilde{S}_{j+\frac12,k}^{n,n+1}
	+ r_0 g_{j+\frac12,k}^n\left(1 - \frac{\rho_{j+\frac12}^n}{\rho_{\rm max}}\right)_+,
\end{align}
where 
\begin{equation*}
    \tilde{S}_{j+\frac12,k}^{n,n+1} =  
	- v_k \psi_{0}(v_k)  
	q(\bar{\rho}_{j+\frac12}^n)\delta_x\rho_{j+\frac12}^{n}
	+ \psi_{1}(v_k,\delta_x c_{j+\frac12}^n)\Phi^{n,n}_{j+\frac12}
	-q(\bar{\rho}_{j+\frac12}^n)\tilde{g}_{j+\frac12,k}^{n+1}\ .
\end{equation*}
By reformulating \eqref{eq:tilde_g}, it is easy to see that 
\begin{align*}
  \left(\frac{1}{\Delta t} + \frac{q(\bar{\rho}_{j+\frac12}^n)}{\varepsilon^2}\right)\tilde{g}_{j+\frac12,k}^{n+1} & =  
    \frac{g_{j+\frac12,k}^{n}}{\Delta t} - \frac{1}{\varepsilon}(I-\Pi) K_{j+\frac12,k}^n 
    + r_0 g_{j+\frac12,k}^n\left(1 - \frac{\rho_{j+\frac12}^n}{\rho_{\rm max}}\right)_+ \\
    & + \frac{1}{\varepsilon^2}
    \left(
    - v_k \psi_{0}(v_k)  
	q(\bar{\rho}_{j+\frac12}^n)\delta_x\rho_{j+\frac12}^{n} + \psi_{1}(v_k,\delta_x c_{j+\frac12}^n)\Phi^{n,n}_{j+\frac12}
    \right)\ ,
\end{align*}
where all the unknowns $\tilde{g}_{j+\frac{1}{2},k}^{n+1}$ can be solved explicitly from \eqref{eq:tilde_g}. 
By comparing \eqref{eq:tilde_g} and the second equation in \eqref{eq:scheme}, it can be observed that 
\begin{align}\nonumber
    &  g_{j+\frac12,k}^{n+1} 
    + \frac{1}{\varepsilon^2}\left(\frac{1}{\Delta t} + \frac{q(\bar{\rho}_{j+\frac12}^n)}{\varepsilon^2}\right)^{-1}\left(
    v_k\psi_{0}(v_k)q(\bar{\rho}_{j+\frac12}^n)\delta_x\rho_{j+\frac12}^{n+1}
    -\psi_{1}(v_k,\delta_x c_{j+\frac12}^n)\Phi^{n+1,n}_{j+\frac12} \right)\\
    = &  \tilde{g}_{j+\frac12,k}^{n+1} 
    + \frac{1}{\varepsilon^2}\left(\frac{1}{\Delta t} + \frac{q(\bar{\rho}_{j+\frac12}^n)}{\varepsilon^2}\right)^{-1}\left(v_k\psi_{0}(v_k)q(\bar{\rho}_{j+\frac12}^n)\delta_x\rho_{j+\frac12}^{n}
    -\psi_{1}(v_k,\delta_x c_{j+\frac12}^n)\Phi^{n,n}_{j+\frac12}\right).
    \label{expression:g}
\end{align}
Then by substituting \eqref{expression:g} into the first equation in  \eqref{eq:scheme}, a system which contains only the unknowns for the densities is derived. 
Specifically, we have 
\begin{align} 
    &\frac{\rho_j^{n+1}}{\Delta t} 
    - \delta_x(a_*^n \delta_x \rho_*^{n+1})_j 
    + \delta_x(b_*^n\Phi^{n+1,n}_*)_j + D_h \delta_x(\rho_*^n q'(\rho_*^n)\delta_x \rho_*^{n+1})_j =  r_j^n\ ,
    \label{eq:rho_n}
\end{align}
where the coefficients $a_{j+\frac12}^n$, $b_{j+\frac12}^n$ and residuals $r_j^n$ can be explicitly computed via 
\begin{equation}
\begin{aligned}
    & a_{j+\frac12}^n 
    = \frac{q(\bar{\rho}_{j+\frac12}^n)\Delta t}{\e^2 + q(\bar{\rho}_{j+\frac12}^n)\Delta t} D_h q(\bar{\rho}_{j+\frac12}^n)\ , \\
    & b_{j+\frac12}^n = \frac{q(\bar{\rho}_{j+\frac12}^n)\Delta t}{\e^2 + q(\bar{\rho}_{j+\frac12}^n)\Delta t} \langle v_k\psi_{1}(v_k,\delta_x c_{j+\frac12}^n)\rangle_h\ , \\
    & r_j^n = \frac{\rho_j^n}{\Delta t} - \langle v_k\delta_x (q_*^n \tilde{g}_{*,k}^{n+1})_j\rangle_h + r_0 \rho_j^n\left(1 - \frac{\rho_j^n}{\rho_{\rm max}}\right)_+ - \delta_x(a_*^n \delta_x \rho_*^n)_j
    + \delta_x(b_*^n\Phi_*^{n,n})_j\ .\label{eq: system implicit discretisation}
\end{aligned}
\end{equation}
To solve all the unknowns $\rho_j^{n+1}$ from the system \eqref{eq:rho_n}, only a tridiagonal matrix needs to be inverted. 
And then the unknowns $g_{j+\frac12,k}^{n+1}$ can be solved explicitly via \eqref{expression:g}. 
In this way, we efficiently update the system \eqref{eq:scheme} from $t=t_n$ to $t=t_{n+1}$. 

\subsection{Asymptotic preserving property}
{Here, we formally check the asymptotic preserving property of the scheme by taking $\varepsilon \to 0$ in the system \eqref{eq:scheme} and we show that the scheme for the kinetic model \eqref{eq:kinetic_approx} converges to a scheme for solving the corresponding macroscopic model \eqref{eq:limit}}. 
By checking the order of $\varepsilon$ of each term in the equation for the perturbation function $g$, i.e. the second equation in \eqref{eq:scheme}, it is easy to see that, as $\varepsilon\to 0$, we should have  $S_{j+\frac12,k}^{n,n+1} = 0$, namely
\begin{align*}
	-v_k\psi_{0}(v_k)q(\bar{\rho}_{j+\frac12}^n)  \delta_x \rho_{j+\frac12}^{n+1} + \psi_{1}(v_k, \delta_x c_{j+\frac12}^n)
    \Phi^{n+1,n}_{j+\frac12}
    -q(\bar{\rho}_{j+\frac12}^n) g_{j+\frac12,k}^{n+1}
	 = 0\ ,
\end{align*}
from where a simple reformulation gives that 
\begin{equation} \label{qg_expression}
    q(\bar{\rho}_{j+\frac12}^n) g_{j+\frac12,k}^{n+1} = -v_k\psi_{0}(v_k)q(\bar{\rho}_{j+\frac12}^n)  \delta_x \rho_{j+\frac12}^{n+1} + \psi_{1}(v_k, \delta_x c_{j+\frac12}^n)
    \Phi^{n+1,n}_{j+\frac12}\ .    
\end{equation}
Combining \eqref{def:Dh} and \eqref{qg_expression}, a direct computation shows that
\begin{equation}\label{proof:qg_int}
    \langle v_k \delta_x (q(\bar{\rho}_*^n)g_{*,k}^{n+1})_j\rangle_h = 
    \delta_x(-D_h q(
    \bar{\rho}_*^n)\delta_x\rho_*^{n+1} + \langle v_k \psi_1(v_k, \delta_x c_*^n)\rangle_h \Phi_*^{n+1,n})_j\ .
\end{equation}
Finally, by substituting \eqref{proof:qg_int} into the first equation in \eqref{eq:scheme}, we get  
\begin{equation}\label{scheme:limit}
\begin{cases}
    &\delta_t^+ \rho_j^n 
     - \delta_x\left[D_h d(\bar{\rho}_*^n)\delta_x \rho_*^{n+1} - \langle v_k \psi_1(v_k, \delta_x c_*^n)\rangle_h \Phi_*^{n+1,n}\right]_j 
    = r_0 \rho_j^{n}\left(1 - \frac{\rho_j^n}{\rho_{\rm max}}\right)_+,\\
    &\delta_x^2 c_j^{n+1} + \rho_j^{n+1} - c_j^{n+1}=0\ ,
\end{cases}
\end{equation}
where $d(\bar{\rho}_*^n) = q(\bar{\rho}_*^n)-\bar{\rho}_*^n q'(\bar{\rho}_*^n)$, 
which is indeed a finite difference scheme for solving the corresponding macroscopic model \eqref{eq:limit}. 
In this way, we verified the asymptotic preserving property of our scheme \eqref{eq:scheme}.

\subsection{Positive preserving property}
Though the scheme \eqref{eq:rho_n} might not be positive preserving for a general fixed $\dt > 0$ and $\e>0$, 
the following proposition shows that its limit \eqref{scheme:limit} as $\e\to 0^+$ is positive preserving if $q(\rho) = 1 - (\rho / \bar{\rho})^\gamma$, where $\gamma\ge1$ and $\overline{\rho}\ge\rho_{\rm{max}}$.
 The above choice of the squeezing probability function is commonly used  for semi-elastic entities as described in the Introduction. 
 A direct computation shows that, with $q(\rho)= 1 - (\rho / \bar{\rho})^\gamma$, the following is always non-negative, 
\begin{equation*}
    d(\rho)= q(\rho) - \rho q'(\rho) = 1 + (\gamma-1)\left(\frac{\rho}{\overline{\rho}}\right)^{\gamma} \ge 0\ .
\end{equation*}

\begin{proposition}
    With a general non-negative function $d(\rho):=q(\rho) - \rho q'(\rho)$, if $\rho_0^n \ge 0$ for all $j$, then, for whatever $\dt>0$, we have 
    $\rho_j^n \ge 0$ for all $j$ and $n\ge1$ in \eqref{scheme:limit}.
\end{proposition}
\begin{proof}
We prove by induction. 
Assuming that $\rho_j^n \ge0$ for all $j$, we aim to show that $\rho_j^{n+1} \ge0$ holds true for all $j$.
For simplicity of notations, we denote $\eta_{j+\frac12}^n = \langle v_k \psi_1(v_k, \delta_x c_{j+\frac12}^n)\rangle_h$.
Noticing that $\eta_{j+\frac12}^n\Phi_{j+\frac12}^{n+1,n}=(\eta_{j+\frac12}^n)_+ q(\rho_{j+1}^n)\rho_j^{n+1} - (\eta_{j+\frac12}^n)_- q(\rho_{j}^n)\rho_{j+1}^{n+1}$ via \eqref{def:Phi}, the numerical scheme \eqref{scheme:limit} can be reformulated in the matrix form
\begin{equation}\label{scheme:macro_matrix}
    M^n \brho^{n+1} = \br^n\ ,
\end{equation}
where $M^n = (m_{i,j}^n)$ is a tri-diagonal matrix and $\br^n = (r_j^n)$ is a vector with
\begin{align}
    & m_{j,j}^n = 1 + \dt 
    \left[
    D_h\frac{
    d(\bar{\rho}_{j+\frac12}^n) + d(\bar{\rho}_{j-\frac12}^n)
    }{(\dx)^2} + 
    \frac{
    (\eta_{j+\frac12}^n)_+ q(\rho_{j+1}^n) + (\eta_{j-\frac12}^n)_- q(\rho_{j-1}^n)}{\dx}
    \right] \ge 0\ , \nonumber\\
    & m_{j,j+1}^n = -\frac{\dt}{(\dx)^2} D_h d(\bar{\rho}_{j+\frac12}^n) - \frac{\dt}{\dx} (\eta_{j+\frac12}^n)_- q(\rho_j^n) \le0\ , \nonumber\\
    & m_{j,j-1}^n = -\frac{\dt}{(\dx)^2} D_h d(\bar{\rho}_{j-\frac12}^n) - \frac{\dt}{\dx} (\eta_{j-\frac12}^n)_+ q(\rho_j^n)\le0\ ,\nonumber\\
    & r_j^n = \rho_j^n + \dt \, r_0 \rho_j^{n}\left(1 - \frac{\rho_j^n}{\rho_{\rm max}}\right)_+ \ge 0\ , \label{def:m_element}
\end{align}
where $(\eta)_+ = \max\{\eta, 0\}\ge0$ and $(\eta)_-=\max\{-\eta, 0\}\ge0$.
Noticing that 
\begin{equation*}
    m_{j,j}^n + m_{j-1,j}^n + m_{j+1,j}^n = 1\ ,
\end{equation*}
the matrix $M^n$ is strictly diagonal dominant in columns with all diagonal elements positive and off-diagonal elements non-positive. 
As a result, the matrix $M^n$ is an M-matrix and thus inverse positive, i.e. all elements of its inverse $(M^n)^{-1}$ are non-negative. 
As a result, we must have $\brho^{n+1} =(M^n)^{-1}\br^n\ge0$. 
\end{proof}

\begin{remark}
From the kinetic scheme \eqref{eq:rho_n}-\eqref{eq: system implicit discretisation} we can show that, when $\varepsilon \to 0$ we recover the formulation \eqref{scheme:macro_matrix}-\eqref{def:m_element}.
Moreover, for the simplified case when $\varepsilon \to 0$, $c(t,x)=0$ and $q(\rho)$ is a constant we obtain
\begin{align*}
    \rho_{j}^{n+1}\Bigl[1+\frac{\Delta ta_{j+\frac12}^n}{(\Delta x)^2}+\frac{\Delta ta_{j-\frac12}^n}{(\Delta x)^2}\Bigr] -\rho_{j+1}^{n+1}\frac{\Delta ta_{j+\frac12}^n}{(\Delta x)^2}&
    -\rho_{j-1}^{n+1}\frac{\Delta ta_{j-\frac12}^n}{(\Delta x)^2}\\
    &=\rho_j^n+\Delta tr_0\rho_j^n\Bigl(1-\frac{\rho_j^n}{\rho_{\textnormal{max}}} \Bigr)_+\ ,
\end{align*}
which is analogous to the positivity preserving property obtained in \cite{bailo2020fully}.
\end{remark}

\begin{proposition}
    If we replace $d(\bar{\rho}_*^n)$ by  $d(\bar{\rho}_*^{n+1})$ in \eqref{scheme:limit} and consider the following modified implicit scheme 
    \begin{equation}\label{scheme:limit_implicit}
        \delta_t^+ \rho_j^n 
        - \delta_x\left[D_h d(\bar{\rho}_*^{n+1})\delta_x \rho_*^{n+1} - \langle v_k \psi_1(v_k, \delta_x c_*^n)\rangle_h \Phi_*^{n+1,n}\right]_j 
        = r_0 \rho_j^{n}\left(1 - \frac{\rho_j^n}{\rho_{\rm max}}\right)_+,
    \end{equation}
{it can be further proved that if $\rho_j^0 < \overline{\rho}$ for all $j$ and $(1+r_0\dt)\rho_{\rm max}<\bar{\rho}$, then $\rho_j^n \le \overline{\rho}$ for all $j$, and $n\ge1$ .}  
\end{proposition}
\begin{proof}
    We prove by contradiction. 
    For simplicity, we denote $j_n$ to be the index at $t=t_n$ such that $\rho_{j_n}^n = \max_j \{\rho_j^n\}$,
    and consider the smallest $n$ such that $\rho_{j_n}^n\ge \overline{\rho}$.  
    Then $n\ge1$ and $\rho_{j_n}^{n-1}< \overline{\rho}$.
    On the other hand, noticing that $q(\rho_{j_n}^n) = 0$, we have 
    \begin{equation*}
        m_{j_n,j_n}^n + m_{j_n,j_n-1}^n + m_{j_n,j_n+1}^n \ge 1 \ .
    \end{equation*}
    As a result, combining with the fact that $\rho_{j_n}^n\ge \rho_{j_n\pm1}^n$, we have  
    \begin{align}\nonumber
        m_{j_n,j_n}^n \rho_{j_n}^n + m_{j_n,j_n+1}^n \rho_{j_n+1}^n + m_{j_n,j_n-1}^n \rho_{j_n-1}^n
        &\ge \left(m_{j_n,j_n}^n  + m_{j_n,j_n+1}^n  + m_{j_n,j_n-1}^n \right)\rho_{j_n}^n \\
        & \ge \rho_{j_n}^n \ge \overline{\rho}\ . \label{proof:rho_bound_1}
    \end{align}
    On the other hand, the scheme \eqref{scheme:limit_implicit} implies that
    \begin{align*}
        m_{j_n,j_n}^n \rho_{j_n}^n + m_{j_n,j_n+1}^n \rho_{j_n+1}^n + m_{j_n,j_n-1}^n \rho_{j_n-1}^n
        & = \rho_{j_n}^{n-1}\left[ 1 + \dt \, r_0 \left(1 - \frac{\rho_{j_n}^{n-1}}{\rho_{\rm max}}\right)_+ \right]\ .
    \end{align*}
    If $\rho_{j_n}^{n-1} <\rho_{\rm max}$, we have
    \begin{equation*}
        \rho_{j_n}^{n-1}\left[ 1 + \dt \, r_0 \left(1 - \frac{\rho_{j_n}^{n-1}}{\rho_{\rm max}}\right)_+ \right] \le \rho_{\rm max}(1+r_0 \dt) < \bar{\rho}\ .
    \end{equation*}
    If $\rho_{j_n}^{n-1} \ge\rho_{\rm max}$, we have 
    \begin{equation*}
        \rho_{j_n}^{n-1}\left[ 1 + \dt \, r_0 \left(1 - \frac{\rho_{j_n}^{n-1}}{\rho_{\rm max}}\right)_+ \right] = \rho_{j_n}^{n-1} < \bar{\rho}\ ,
    \end{equation*}
    where the last inequality is due to the fact that $n$ is the smallest integer such that $\rho_{j_n}^n \ge n$. 
    We conclude that 
    \begin{equation*}
        m_{j_n,j_n}^n \rho_{j_n}^n + m_{j_n,j_n+1}^n \rho_{j_n+1}^n + m_{j_n,j_n-1}^n \rho_{j_n-1}^n < \bar{\rho}\ ,
    \end{equation*}
    which contradicts \eqref{proof:rho_bound_1}. 
    In this way, we proved that there is no $n$ such that $\rho_{j_n}^n \ge\bar{\rho}$.
    In other words, we must have $\rho_j^n\le\bar{\rho}$ for all $j$ and $n\ge0$.
\end{proof}

\section{Numerical experiments}\label{numerical results}

In this section we present several numerical examples.
In particular, we {numerically verify} the convergence of the kinetic model proposed in \eqref{eq:kinetic_approx}, which we denote as $\rho_\textnormal{kinetic}^\e$, to the volume-exclusion Keller-Segel model \eqref{eq:macro_proliferation}, denoted as $\rho_\textnormal{macro}$, as $\varepsilon\to 0$ in one and two dimensions. 
\subsection{Energy dissipation and convergence tests in 1D}
In Appendix \ref{subsec:energy} we proved that, under some assumptions \cite{calvez2006volume,carrillo2001entropy,de2018energy}, the volume-exclusion Keller-Segel model \eqref{eq:macro_proliferation} is energy dissipative, 
{where the energy is defined by the functional
\begin{equation}\label{def:energy}
    \mathcal{E}(t)=\int\Phi(\rho)\diff x-\frac{1}{2}\int\rho c\diff x\ ,    
\end{equation}
\noindent with $ \Phi $ satisfing Eq. \eqref{functionPhi}. {Via numerical integration, we can accurately approximate $\Phi(\rho)$. The energy $\mathcal{E}(t)$ in \eqref{def:energy} can then be numerically approximated via quadrature rules.} 
We will verify numerically that the energy   along the solutions $\rho_\textnormal{macro}$ of the macro model (see \ref{subsec:energy}) indeed decreases in time. 
Moreover, we study how the functional \eqref{def:energy} evolves along the numerical solutions of the kinetic model. 
For clarity, we will denote by $\mathcal{E}_\e(t)$ the value of the functional \eqref{def:energy} computed on the solution of the kinetic system $\rho_\textnormal{kinetic}^\e$ for a given $\e$ at a given time $t$. } 
The convergence of density profiles as $\varepsilon\to0$ will be numerically tested as well. 
We will compare $\rho_\textnormal{macro}(t,x)$ and $\rho_\textnormal{kinetic}^{\varepsilon}(t,x)$ at specific time points and show the convergence rate by checking $\frac{\|\rho_\textnormal{macro}-\rho_\textnormal{kinetic}^{\varepsilon}\|_2}{\| \rho_\textnormal{macro}\|_2}$ in the limit $\varepsilon\to0$, where $||\cdot||_2$ is the $L_2$ norm.

For simplicity, we consider the 1D problem  within the domain $(x, v) \in (-20,20)^2$. 
We use a uniform mesh with {$\Delta x = 0.1, \Delta v = 0.2$}. 
The periodic boundary condition is applied in the $x$-direction and the zero boundary condition is applied in the $v$-direction. 
In the simulations, we choose  $r_0=0.1$, $\rho_{\rm max} = 0.5$, $\bar{\rho} = 1$ and 
\begin{equation}\label{def:psi01}
    \psi_0(v) = \frac{1}{\sqrt{2\pi}}e^{-\frac{v^2}{2}}\ , \quad \psi_1(v, \nabla c) = \frac{v}{\sqrt{2\pi}}e^{-\frac{v^2}{2}} \nabla c\ .
\end{equation}
{It is easy to check that the choices of $\psi_0(v)$ and $\psi_1(v)$ satisfy the Hypothesis \ref{hypo1} and Hypothesis  \ref{hypo2}.}

By choosing the time step to be $\Delta t = 10^{-4}$ and the initial data to be
\begin{align*}
	&\rho_{0_\textnormal{macro}}(x) = \rho_{0_\textnormal{kinetic}}^\varepsilon(x) = 0.5 + u(x)\ , \quad c_0(x) = c_0^\varepsilon = 0.5\ , \quad g_0^\varepsilon(x, v) = 0\ ,
\end{align*}
where $u(x)$ is a uniformly distributed random function ranging in $(-0.1, 0.1)$,
we compute the solution until $t=40$. 
In Figure \ref{fig:density_eps_A_20} we start with a comparison between $\rho_\textnormal{kinetic}^{\varepsilon}$, for different values of $\varepsilon$ {($\e = 0.2$, purple curves, $\e = 0.1$, yellow curves and $\e = 0.05$, red curves) and $\rho_\textnormal{macro}$ (blue curves), for different simulation times $t=5,\ 20,\ 32,\ 40$ (from upper left to bottom right panels, respectively)}. When the aggregates are forming ($t=5$) or merging together ($t=32$), the discrepancy between the kinetic and the macroscopic solutions are  larger, specially for large values of $\e$ (purple line). As time progresses ($t=40$) this difference becomes smaller and we observe a very good agreement between the solutions of the kinetic and the macroscopic models for small values of the scaling parameter $\e$.

\begin{figure}[tbhp]
\centerline{\includegraphics[scale = 0.3]{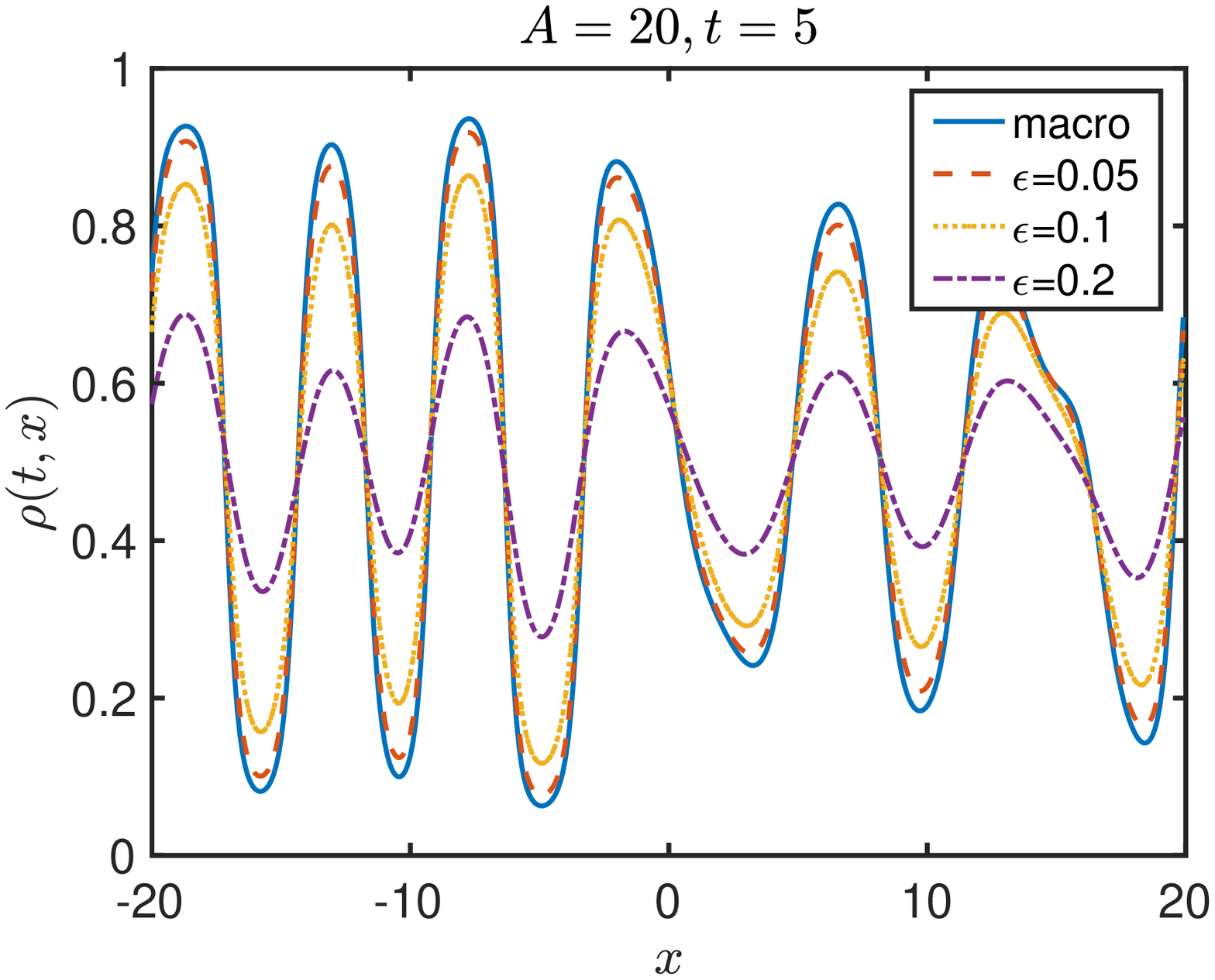}
\includegraphics[scale = 0.3]{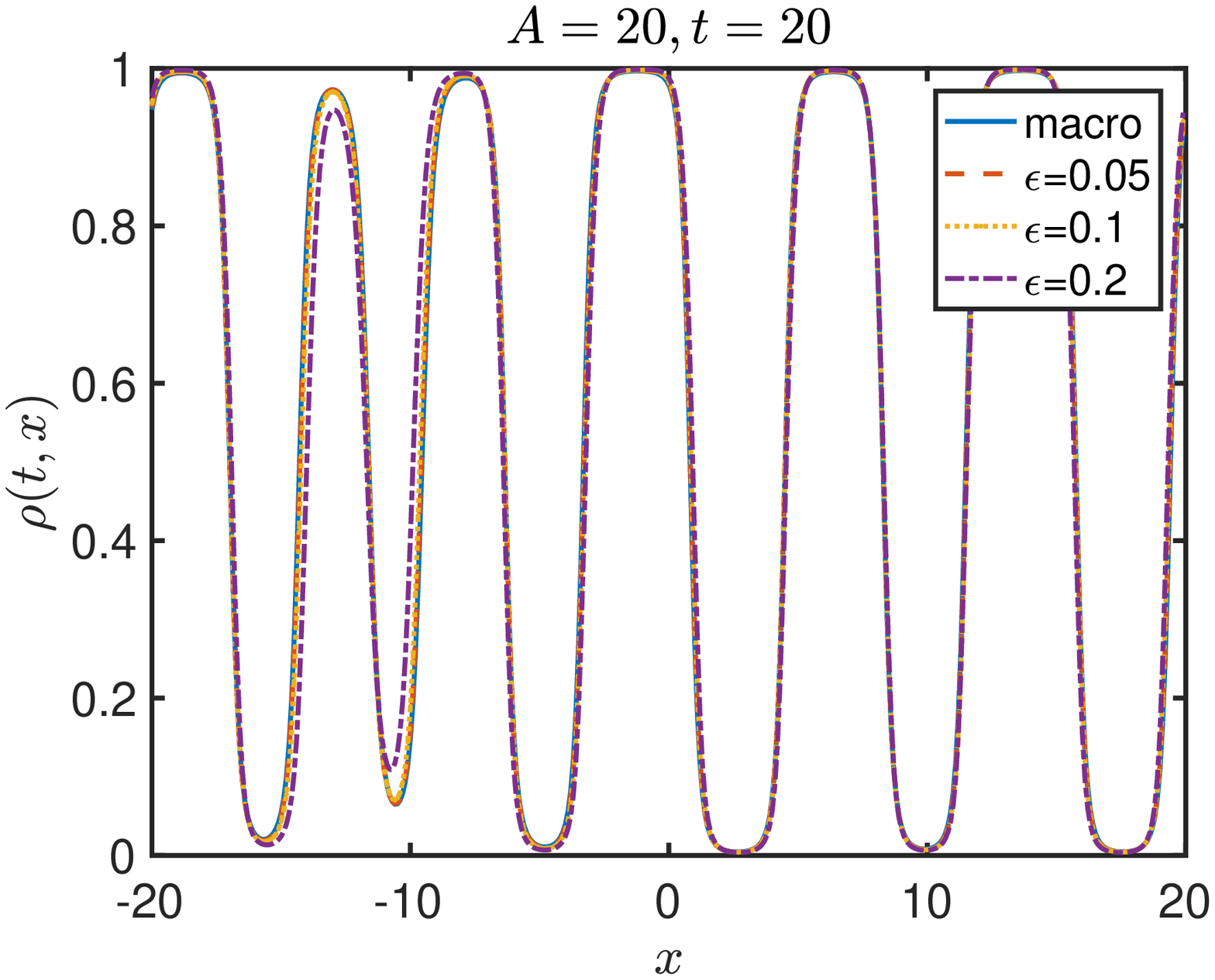}
}
\centerline{\includegraphics[scale = 0.3]{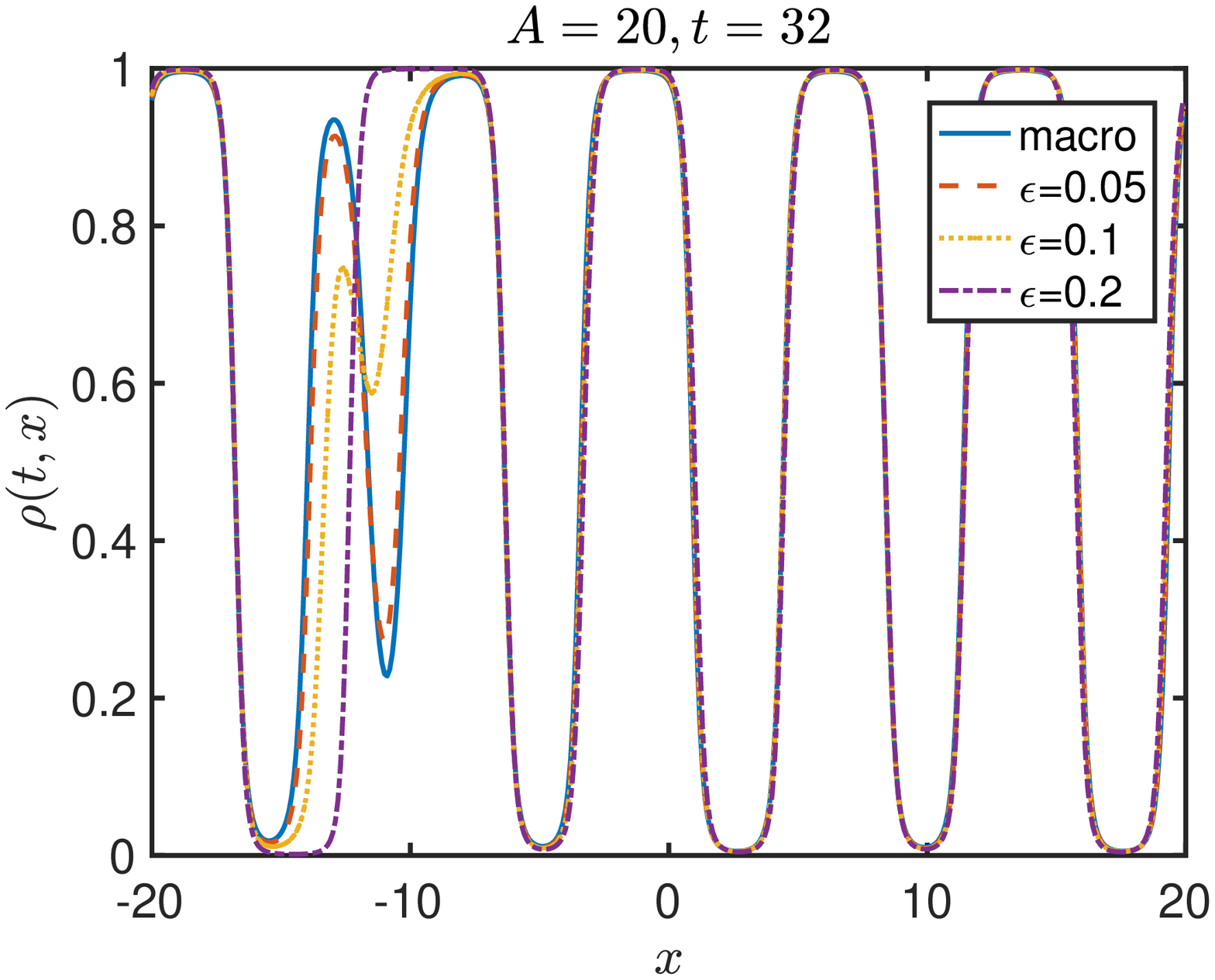}
\includegraphics[scale = 0.3]{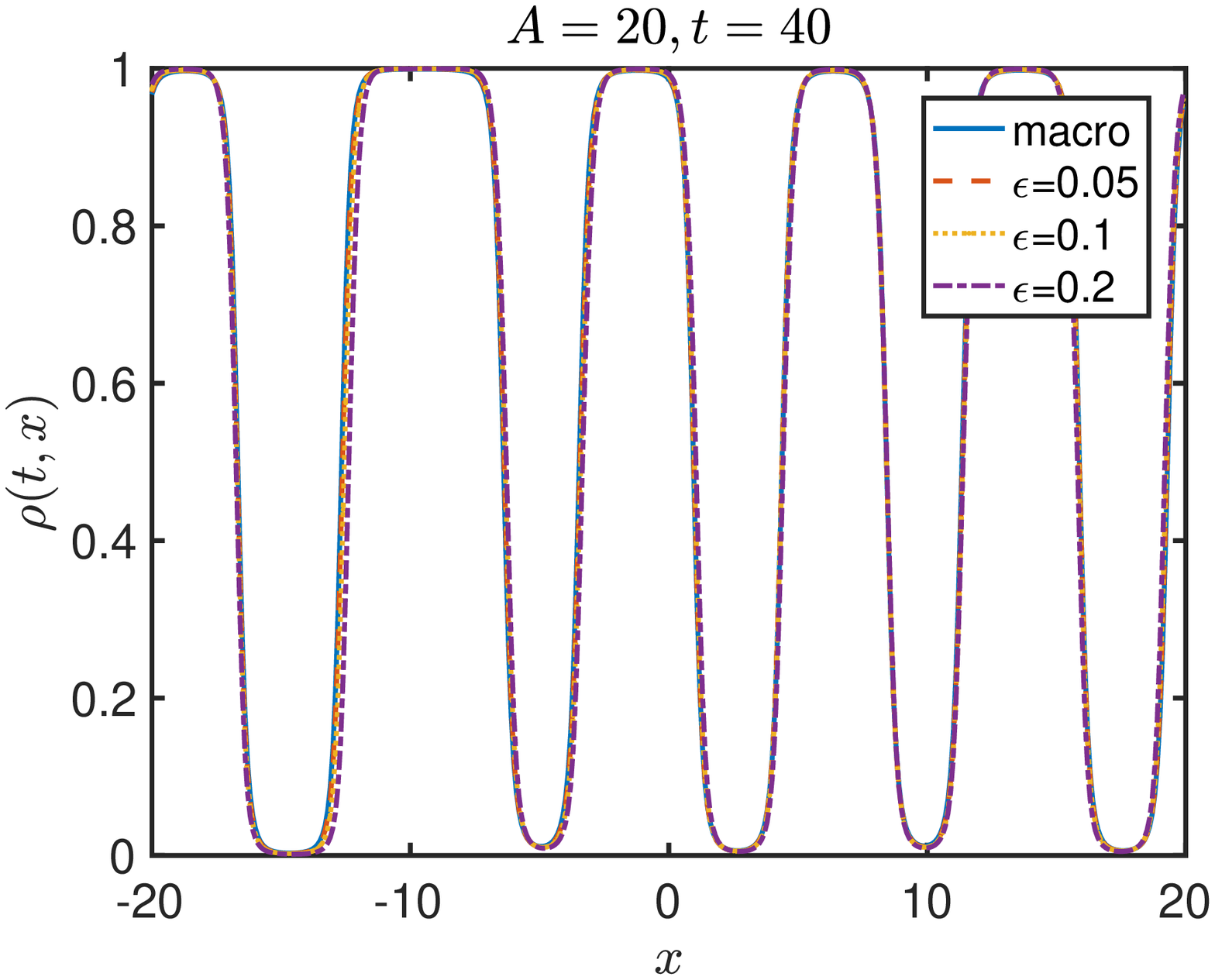}
}
\caption{Comparison of $\rho_\textnormal{macro}$ and $\rho_\textnormal{kinetic}^{\varepsilon}$ for $A=20$ and different  $\e$'s.}
\label{fig:density_eps_A_20}
\end{figure}

{In Figure \ref{fig:plot_energy_w_sol}, we show the evolution of the energy quantities $\mathcal{E}(t)$ given by \eqref{def:energy} (blue curve) and $\mathcal{E}_\varepsilon(t)$ as functions of time, for different values of $\e$: $\varepsilon=0.2$ (purple curve), $\e = 0.1$ (yellow curve) and $\e = 0.05$ (red curve), up to $t=50$. This figure shows that the energies of the kinetic and macroscopic models are in very good agreement}. The inset figures show the evolution in time of the macroscopic density (continuous blue line) and the kinetic density for different values of $\e$ (lines with the same style as for the energy). It is clear from these figures that the larger discrepancies between the kinetic and macro energies are indeed related with changes in the density profiles, for example when two aggregates merge together {(see the inset plots at $t=5$ and $t=32$)}. {Even in this critical case of aggregation formation we observe that the kinetic solution for $\e=0.05$ agrees with the macroscopic solution.}

A similar behaviour is observed in Figure \ref{fig:plot_energy_time} (left) where we plot the {relative} $L_2$-error between the kinetic and the macroscopic solutions {as a function of time} and for different values of $\e$ {($\e = 0.2$, black curve, $\e = 0.1$, blue curve and $\e = 0.05$, red curve)}. In agreement with the behaviour observed in Figure \ref{fig:plot_energy_w_sol} this error is {larger} at times $t=5$ and $t=32$, approximately, which corresponds {to times where aggregates are merging}.

In Figure \ref{fig:plot_energy_time} (right) we {show} the rate of convergence of the {relative} $L_2$-error between the kinetic, $\rho_\textnormal{kinetic}^\e$, and macroscopic, $\rho_\textnormal{macro}$, solutions for different values of $\e$, at different times. {We observe that the error between both solutions decreases as $\e$ decreases, and the convergence order is around 1.5 in {$\ell^2$} norm. Altogether, these first results suggest that the macroscopic and kinetic models are in good agreement for small values of $\e$, and that the kinetic model converges towards the macroscopic model as $\e \rightarrow 0$ in the 1D case. {In Figure \ref{fig:plot_energy_time} (left) the kinetic model seems to converge faster to the ``aggregated-state'' compared to the macroscopic dynamics} i.e. for large values of $\e$ (black curve) we see an early merging of aggregates, compared to the blue and red curves. These changes in speed could
be due to the diffusion scaling, in which the macroscopic model is obtained in a regime where there are many velocity jumps but small net displacements in one order of time. Therefore, in the macroscopic setting, each particle interacts with many more particles than
in the kinetic model, which could result in a delay in the aggregation process. In the next section, we take a step further and analyse the evolution of the pattern sizes in time as function of the chemotactic sensitivity  $A$.}

\begin{figure}[tbhp]
\centerline{\includegraphics[scale = 0.3]{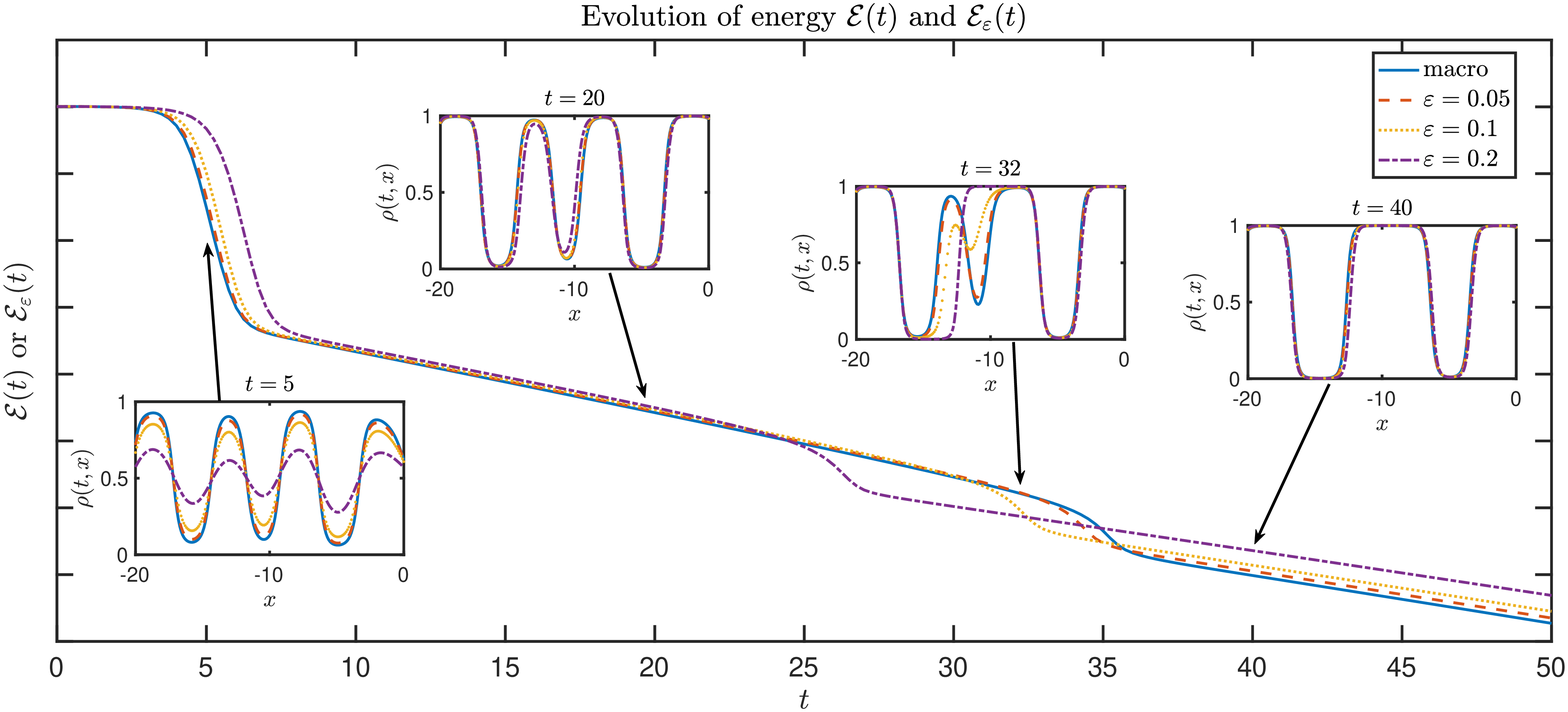}
}
\caption{Evolution of $\mathcal{E}(t)$ and $\mathcal{E}_\e(t)$ along with the comparison between the kinetic solutions $\rho_\textnormal{kinetic}^\varepsilon$ and the macroscopic solutions $\rho_\textnormal{macro}(t,x)$ at $t=5,\ 20,\ 32,\ 40$ for $A=20$.}
\label{fig:plot_energy_w_sol}
\end{figure}

\begin{figure}[tbhp]
\centerline{
\includegraphics[scale = 0.3]{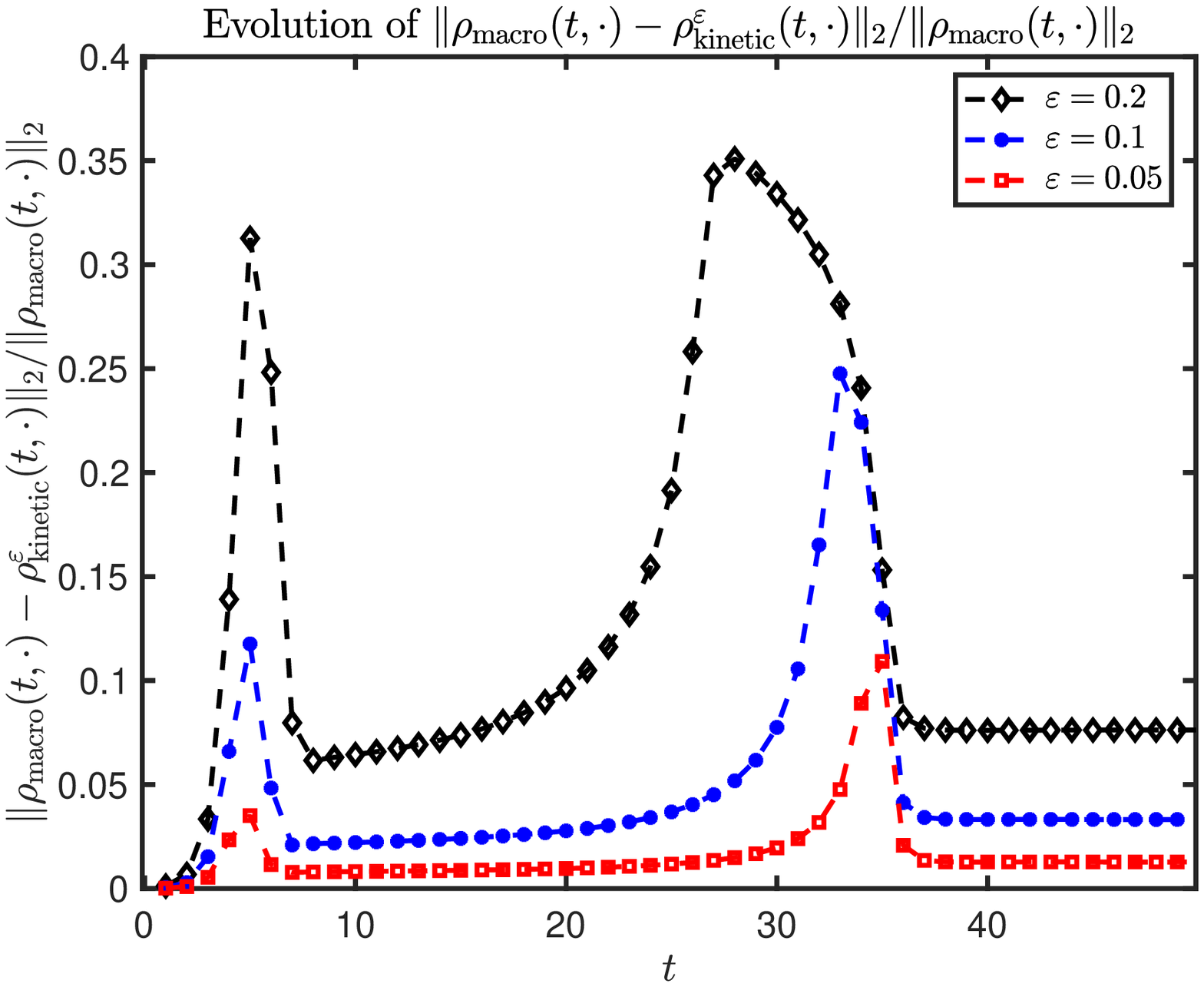}
\includegraphics[scale = 0.3]{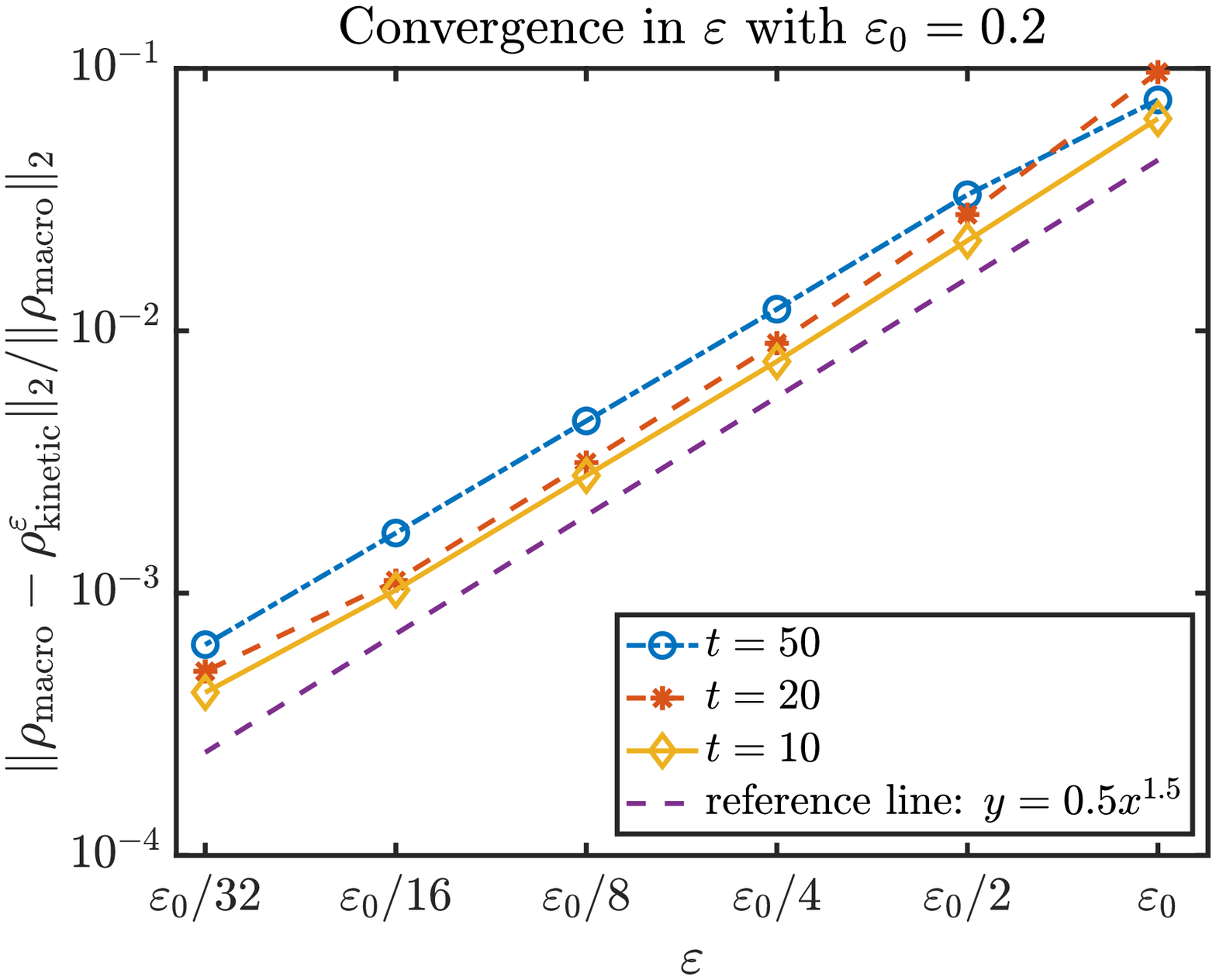}
}
\caption{Left: Evolution of the {relative} $L_2$-error $\frac{\|\rho_{\rm{macro}}(t, \cdot) - \rho_{\rm{kinetic}}(t,\cdot)\|_2}{\|\rho_{\rm{macro}}(t, \cdot)\|}$ over time with $A=20$. Right: Convergence of the relative $L_2$-error in $\e$  at $t=10,\  20, \ 50$. The numerical setting is the same as in Figure \ref{fig:density_eps_A_20}.}
\label{fig:plot_energy_time}
\end{figure}

\subsection{Pattern formation from a perturbed 1D initial data}
With a strong chemotaxis effect, cells will aggregate to form patterns in regions where the chemoattractant is highly concentrated. 
For the volume-exclusion Keller-Segel model \eqref{eq:macro_proliferation}, a relation between the aggregate size from a perturbed initial data and the strength of chemotaxis effect $A$ {was proven in \cite{almeida2020treatment} via linear stability analysis}. 
In this section, we numerically verify this relation for both the kinetic \eqref{eq:kinetic_approx} and the macroscopic model \eqref{eq:macro_proliferation}. 
{Again, we only consider here the 1D case with periodic boundary conditions in space. 
More specifically, we consider the domain $(x,v)\in (-20,20)^2$ with a uniform mesh {$\dx = 0.1,\  \Delta v =  0.2$}. We choose the time step $\Delta t = 10^{-3}$ and starting from a randomly perturbed initial data
\begin{equation*}
    \rho_{0_\textnormal{macro}}(t,x) = \rho_{0_\textnormal{kinetic}}^{\varepsilon}(t,x) = 0.5 + u(x)\ ,
\end{equation*}
we let the simulation run until $t=20$. 
To avoid effects due to the randomness of the initial data, we will compute the pattern size for 10 solutions, each evolved from some random initial data, and simply average. }

To numerically compute the pattern sizes, we consider the Fourier transform of the density function $\rho(t,x)$ (macro and kinetic) and extract the frequency that corresponds to the maximal Fourier mode. 
Specifically, we consider
\begin{equation*}
    k_{\rm max} = \rm{argmax}_{\lambda} (|\hat{\rho}(\lambda)|)\ , 
\end{equation*}
where $\hat{\rho}(\lambda) = \mathcal{F}(\rho)(t,x)$ is the Fourier transform of the density function $\rho(t,x)$. 
Then, $1/k_{\rm max}$ can be used to describe the pattern size. 

\begin{figure}[tbhp]
\centerline{
\includegraphics[scale = 0.3]{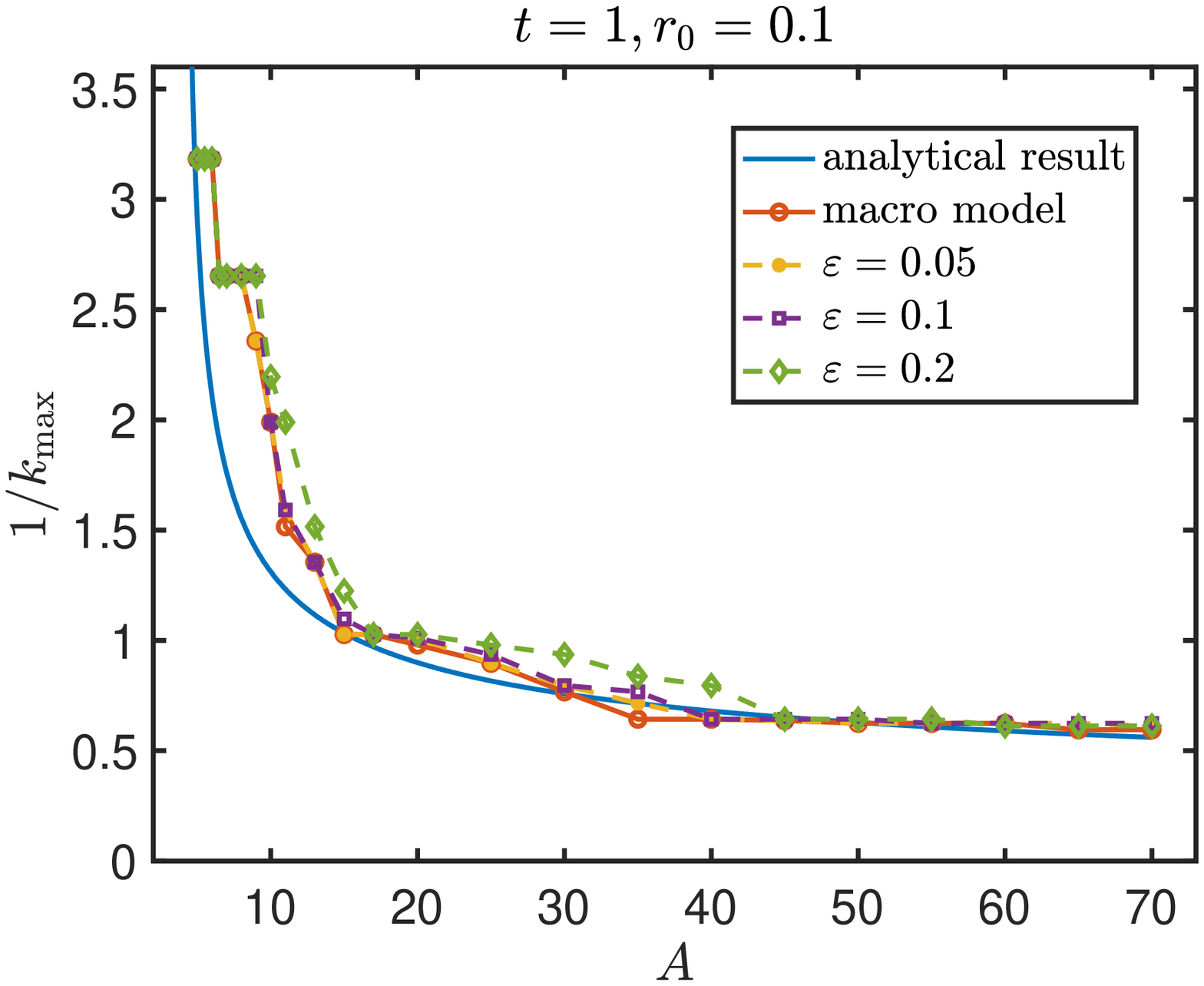}
\includegraphics[scale = 0.3]{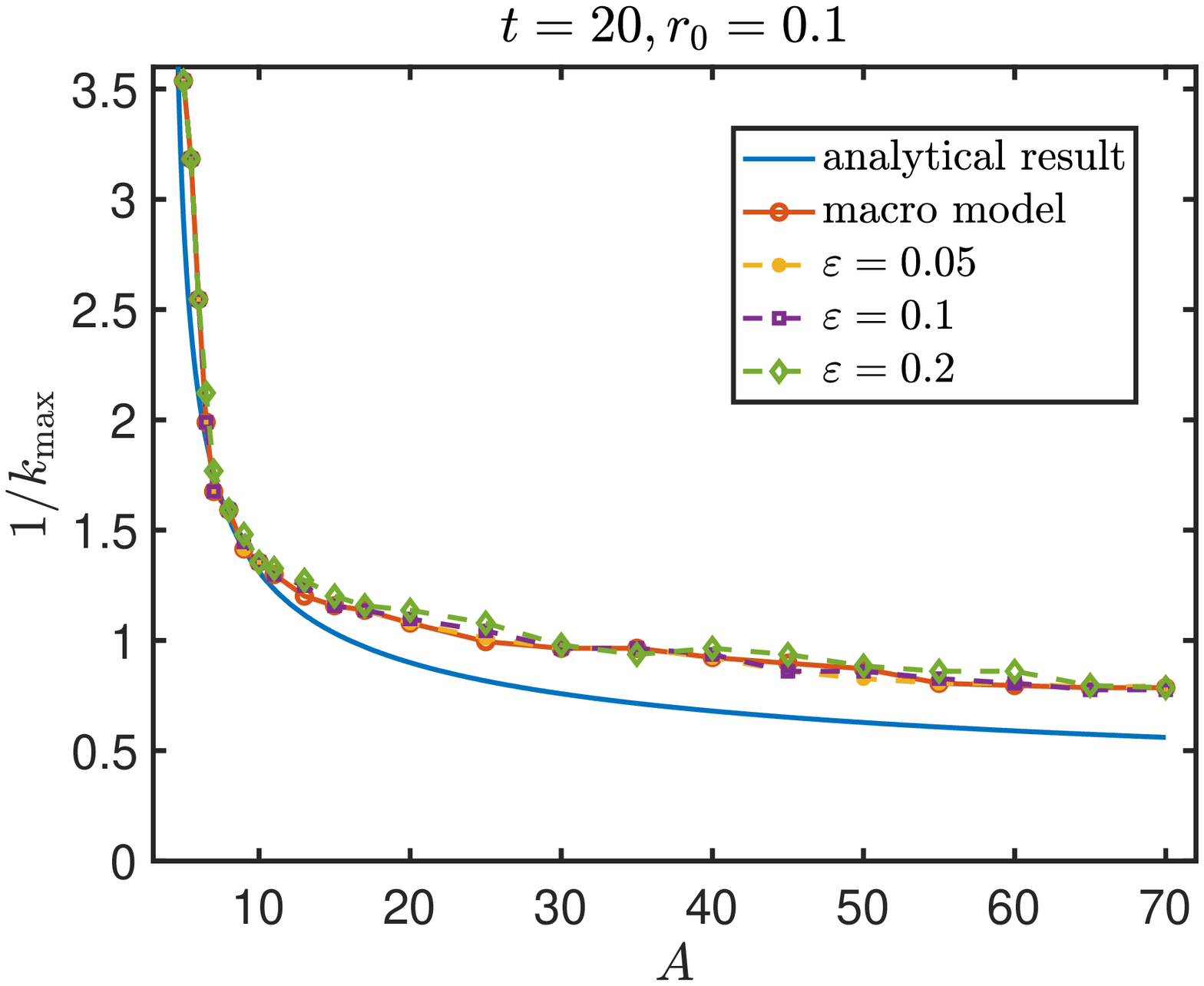}
}
\caption{Comparison of $\frac{1}{k_{\textnormal{max}}}$ between the analytical, the numerical results of the macro and the kinetic model. The numerical result is computed by averaging 10 solutions, each evolved from a random initial data around $0.5$.}
\label{fig:k2_inv_compare}
\end{figure}

In Figure \ref{fig:k2_inv_compare} we show the overall pattern sizes $1/k_{\rm max}$ as a function of the chemotaxis sensitivity $A$ {at times $t=1$ (left panel) and $t=20$ (right panel). For each time, we plot the analytical prediction of $\frac{1}{k_{\max}}$ (blue line, see \cite{almeida2020treatment}), the numerical result for the macroscopic model (red line) and the results for the kinetic model with various values of $\e$ ($\e = 0.05$ in yellow, $\e = 0.1$ in purple and $\e = 0.2$ in green lines). 
As one can observe, we obtain a very good agreement between the predicted pattern sizes and the ones computed numerically for both the macroscopic and kinetic models. As predicted by the stability analysis performed in \cite{almeida2020treatment}, the pattern sizes decrease as the chemotactic sensitivity $A$ increases, and we recover the critical value $A^* \approx 6.9$ bellow which there are no patterns, i.e for which the perturbations are damped and the solution comes back to a homogeneous distribution. }

\subsection{2D numerical examples}
The numerical schemes for both the kinetic model \eqref{eq:scheme} and the macroscopic model \eqref{scheme:limit} can be generalized to multi-dimensional problems, where the tensor-product grid is adopted {(see Appendix \ref{sec:appendix-FD-2D} for a detailed description of the 2D numerical scheme for the kinetic model)}. 

In this section we perform 2D simulations for both the kinetic model \eqref{eq:macro_proliferation} and the volume-exclusion Keller-Segel model \eqref{eq:kinetic_approx}. 
We consider the computation domain $\Omega_\mathbf{x} = \{\mathbf{x} = (x_1,x_2)\in\mathbb{R}^2: -20\le x_1, x_2 \le 20\}$ with a uniform mesh $\Delta x_1 = \Delta x_2 = 0.1$ and periodic boundary conditions.  
For the kinetic model \eqref{eq:kinetic_approx} we need to further define the domain of the velocity $\Omega_\mathbf{v} = \{\mathbf{v} = (v_1,v_2)\in\mathbb{R}^2: -10\le v_1, v_2 \le 10\}$ with a uniform mesh $\Delta v_1 = \Delta v_2 = 0.2$ and zero boundary conditions. 
We choose $r_0=0.1$, $\rho_{\rm max} = 0.5$, $\bar{\rho} = 1$ and 
\begin{equation*}
    \psi_0(\mathbf{v}) = \frac{1}{2\pi} e^{-\frac{|\mathbf{v}|^2}{2}}, \quad \psi_1(\mathbf{v}\ , \nabla_\mathbf{x} c) = \frac{\mathbf{v}}{2\pi} e^{-\frac{|\mathbf{v}|^2}{2}}\cdot \nabla_\mathbf{x} c\ .
\end{equation*}
As in the 1D case, we can check that the choices of $\psi_0(\mathbf{v})$ and $\psi_1(\mathbf{v})$ satisfy the Hypotheses \ref{hypo1} and \ref{hypo2}.

We fix $\Delta t = 10^{-2}$ and choose the initial data to be
\begin{align*}
	&\rho_{0_\textnormal{macro}}(\mathbf{x}) =  \rho_{0_\textnormal{kinetic}}^\varepsilon(\mathbf{x}) =  0.5+u(x)\ , \quad c_0(\mathbf{x}) = c_0^\varepsilon(\mathbf{x}) = 0.5\ , 
\end{align*}
where $u(x)$ is a randomly chosen uniformly distributed function ranging in $(-0.1, 0.1)$. 
In Figure \ref{fig:2D_macro} we show the numerical results at $t=5$ (first row), $t=20$ (second row) and $t=50$ (third row) for the macroscopic model with different chemotaxis sensitivities $A=6$ (first column), $A=20$ (second column) and $A=50$ (third column). As one can observe, starting from an initial data perturbed around the homogeneous value $0.5$, we obtain the formation of labyrinthic patterns for a chemotactic sensitivity $A>6$ (middle and right columns), while the solution dampens to the homogeneous state for $A=6$ (left column), in agreement with the predictions of the stability analysis performed in \cite{almeida2020treatment} and the results in Figure \ref{fig:k2_inv_compare}. Moreover, we observe that larger values of the chemotactic sensitivity $A$ leads to sharper layers near the boundary of the patterns (compare middle and right columns) as expected. 

In Figure \ref{fig:2D_compare} we compare the solutions of the macroscopic 2D model (top row) with the solutions of the 2D kinetic model for $\varepsilon = 10^{-2}$ (bottom row), for two values of the chemotactic sensitivity $A = 20$ (first and third columns) and $A = 50$ (second and fourth columns), and for different values of the initial data $\rho_0$: $\rho_0 = 0.5$ (first two columns) and $\rho_0 = 0.1$ (last two columns). As observed in \cite{painter2002volume}, we recover the formation of different types of patterns as a function of the initial condition for both the kinetic and the macro solution, i.e. labyrinthic patterns in the case $\rho_0 = 0.5$ and round patterns for $\rho_0 = 0.1$ (compare the first two columns with the last two). Moreover, we observe that the pattern sizes decrease and become sharper when the chemotactic sensitivity $A$ increases also for the kinetic solution.

In order to quantify the differences between the kinetic and macroscopic 2D models, we show in Figure \ref{fig:2D_convergence} (left) the evolution in time of the relative $L^2$-error between the macroscopic and kinetic models for different values of $\varepsilon$: $\varepsilon = 10^{-2}$ (blue curve), $\varepsilon = 10^{-4}$ (red curve), $\varepsilon = 10^{-6}$ (yellow curve), $\varepsilon = 10^{-8}$ (purple curve), and Figure \ref{fig:2D_convergence} (right) shows this relative error as function of $\varepsilon$ for different time points: $t=1$ (blue curve), $t=3$ (red curve), $t=5$ (yellow curve) and $t=10$ (purple curve). We observe that the relative error between both models decreases as $\varepsilon$ decreases, and the reference line $y = x$ (green curve) shows that the rate of convergence of the kinetic model towards the macroscopic one is roughly $\mathcal{O}(\varepsilon)$.

\begin{figure}[tbhp]
\centerline{
\includegraphics[scale = 0.25]{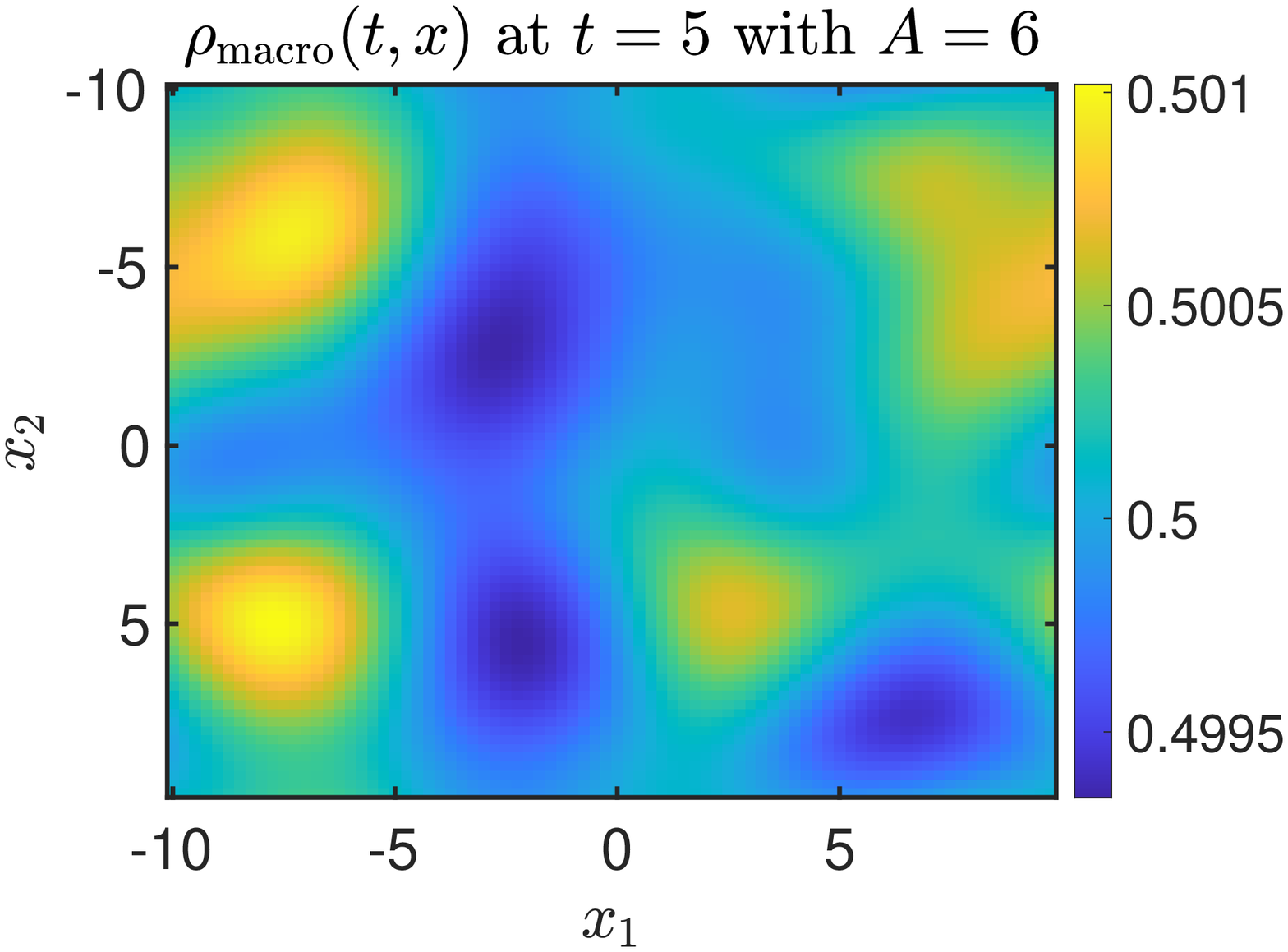}
\includegraphics[scale = 0.25]{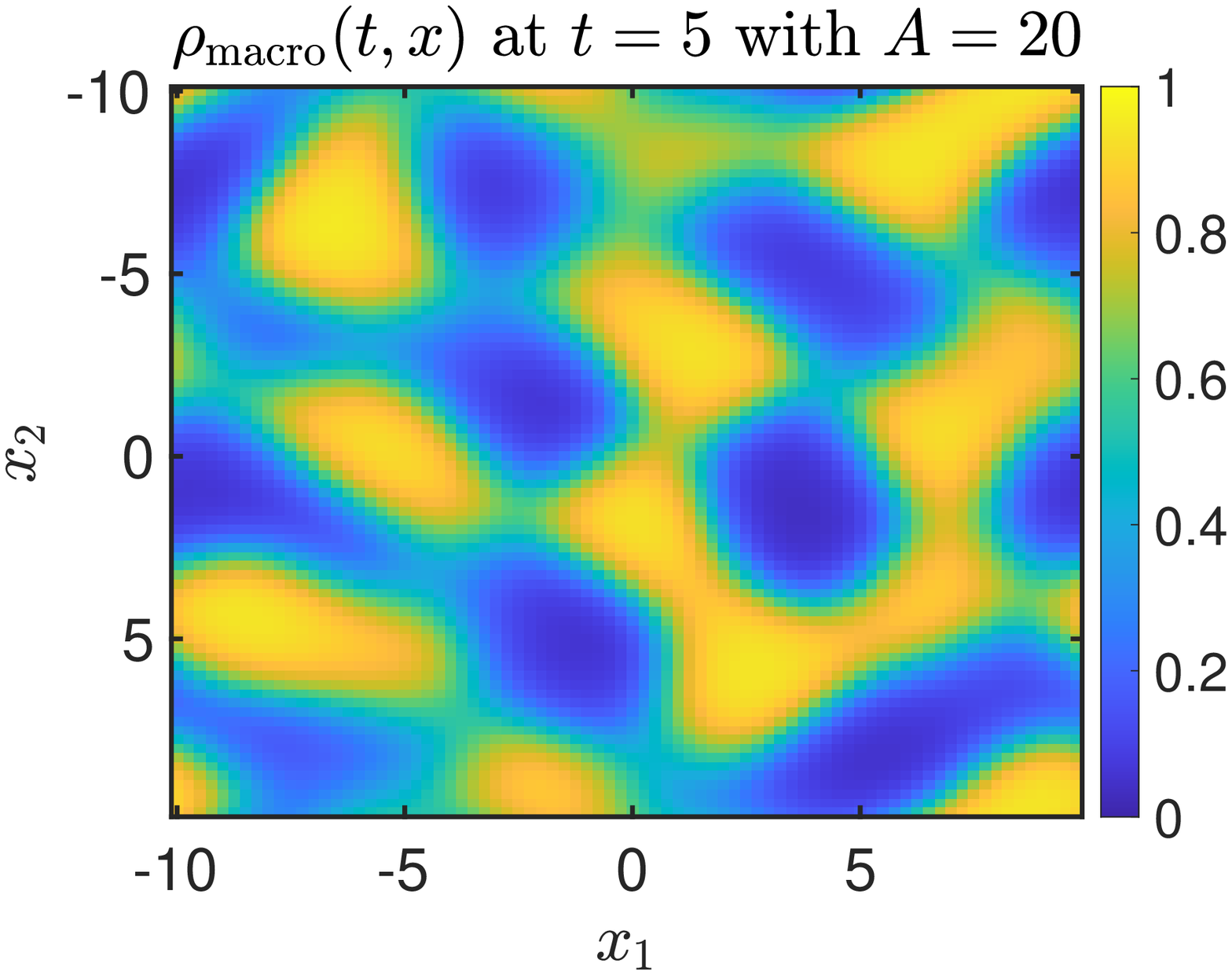}
\includegraphics[scale = 0.25]{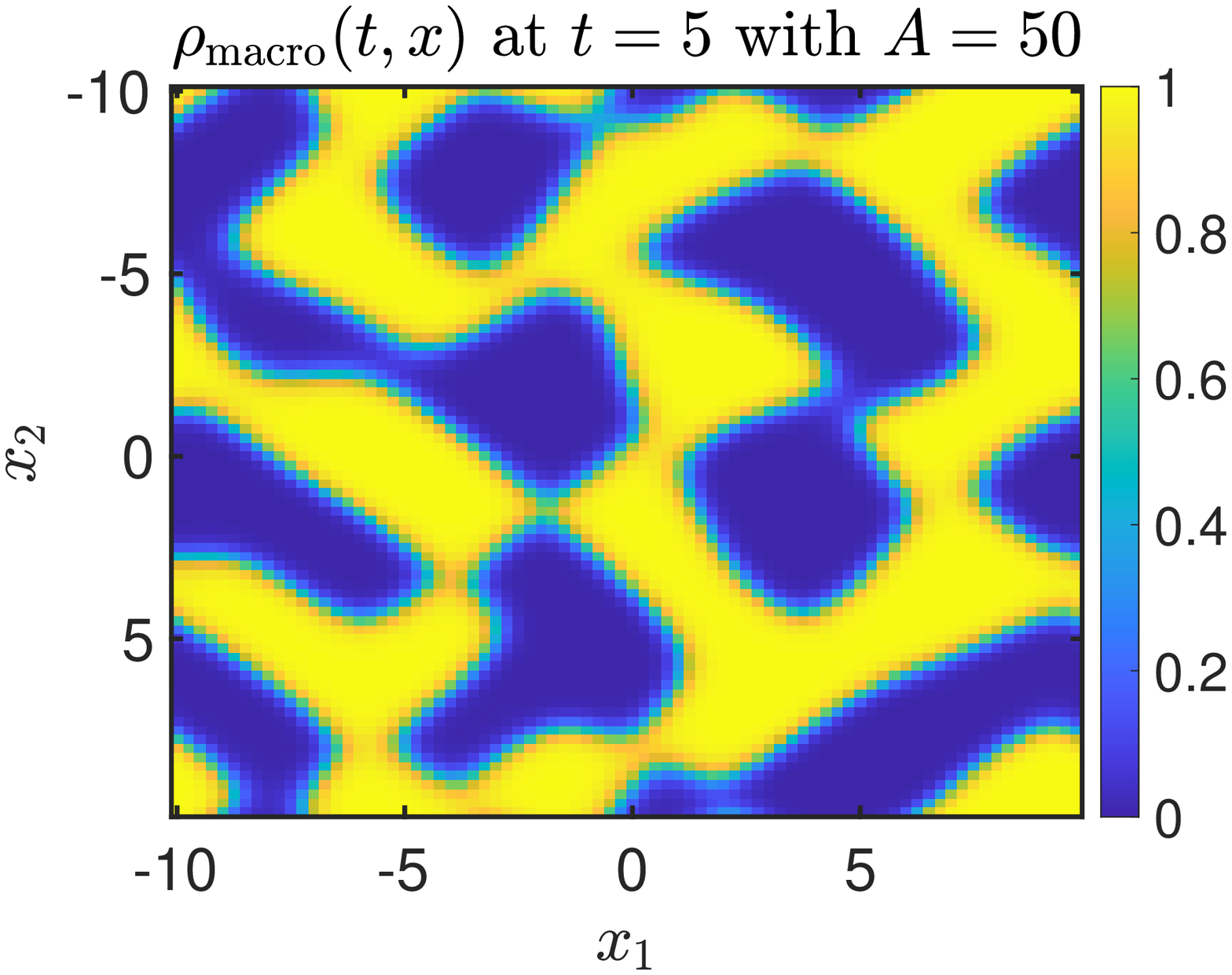}
}
\centerline{
\includegraphics[scale = 0.25]{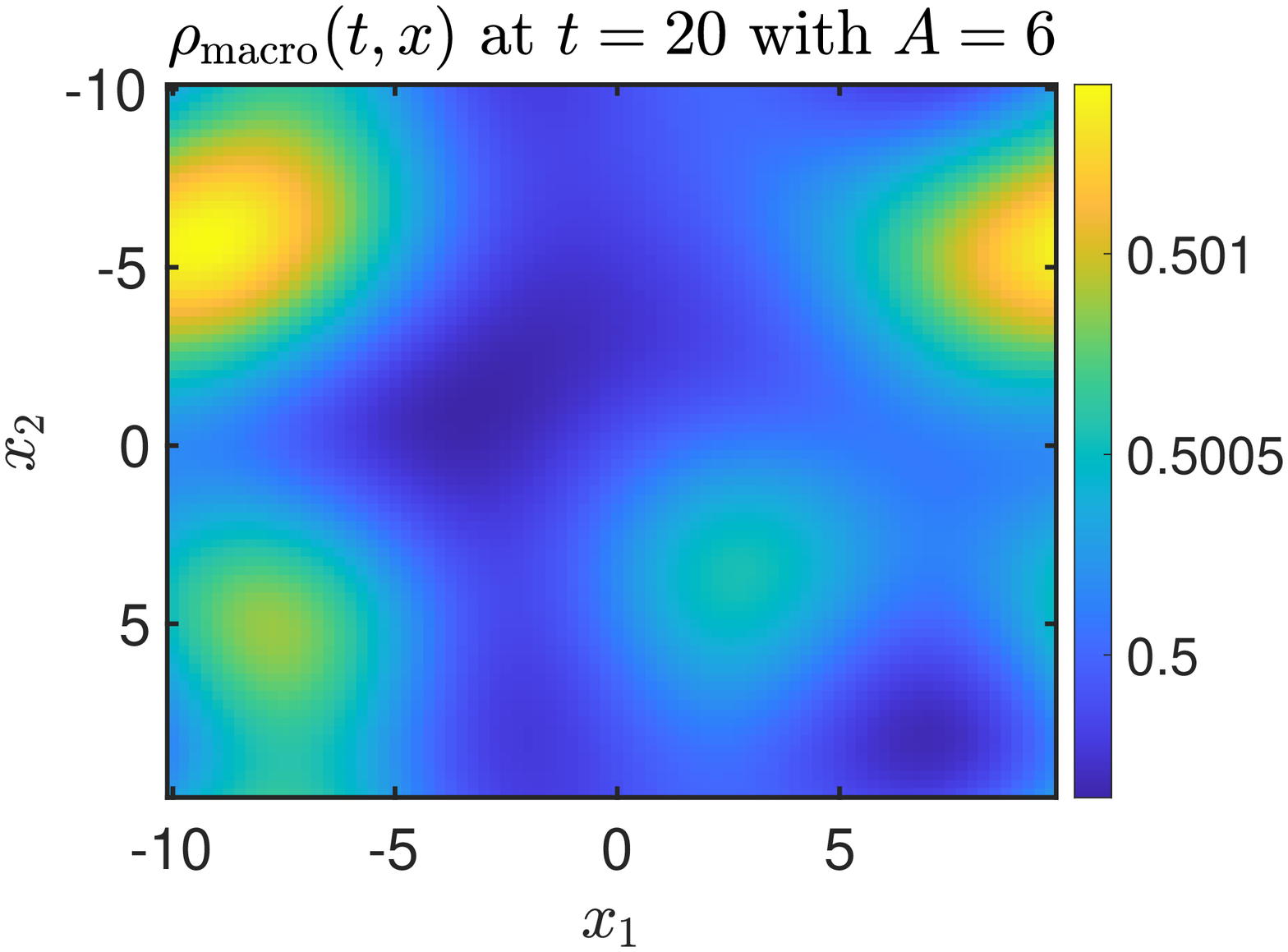}
\includegraphics[scale = 0.25]{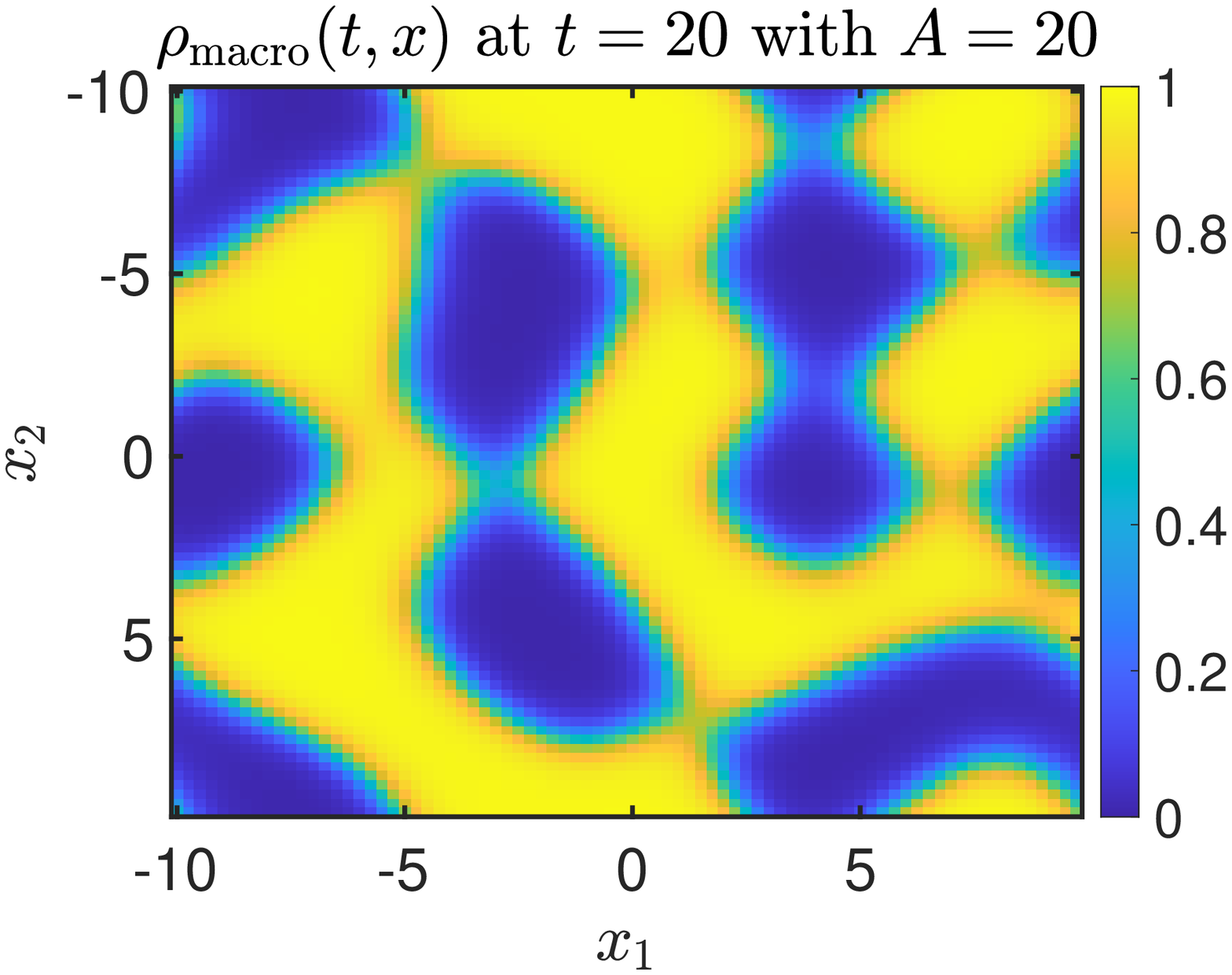}
\includegraphics[scale = 0.25]{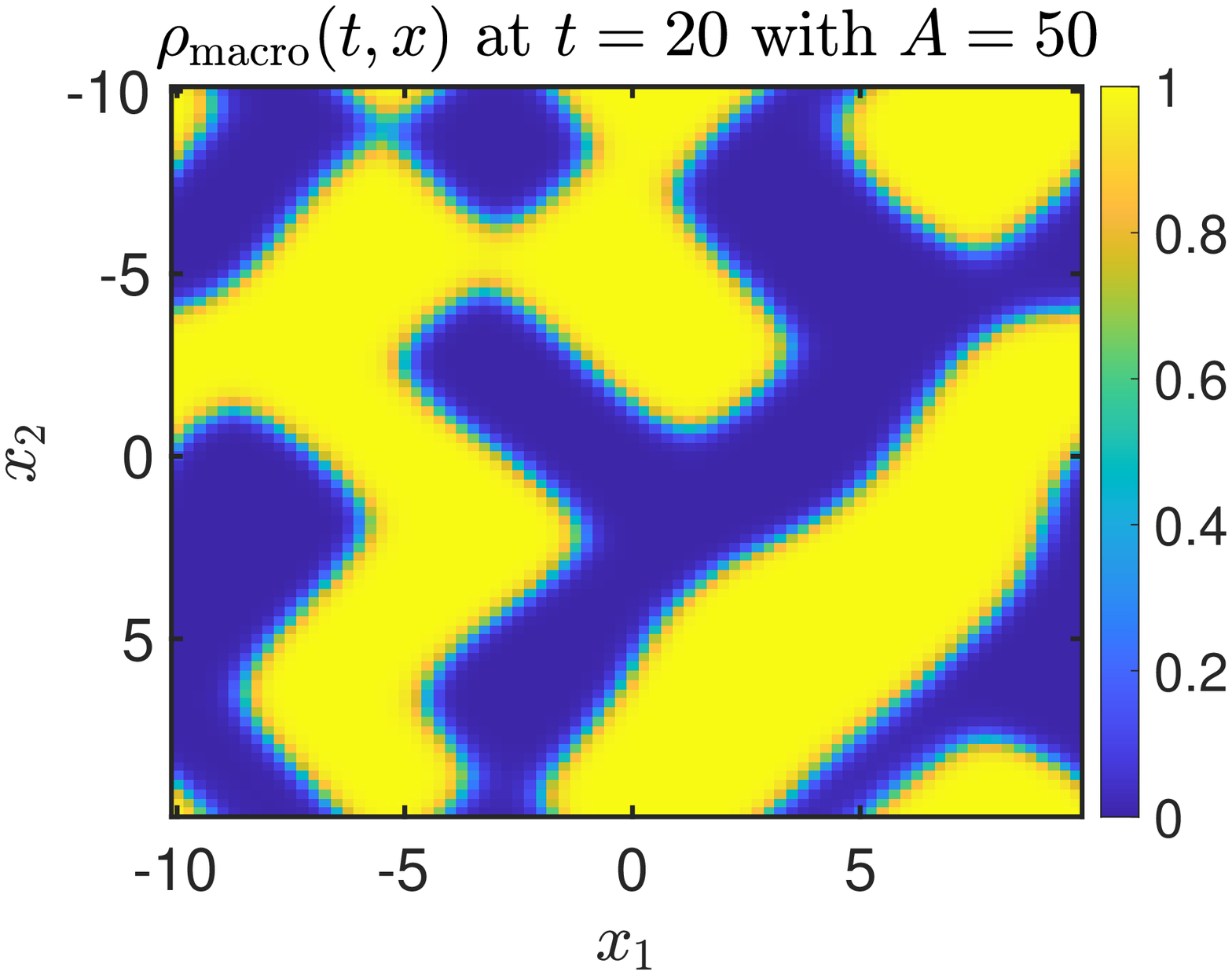}
}
\centerline{
\includegraphics[scale = 0.25]{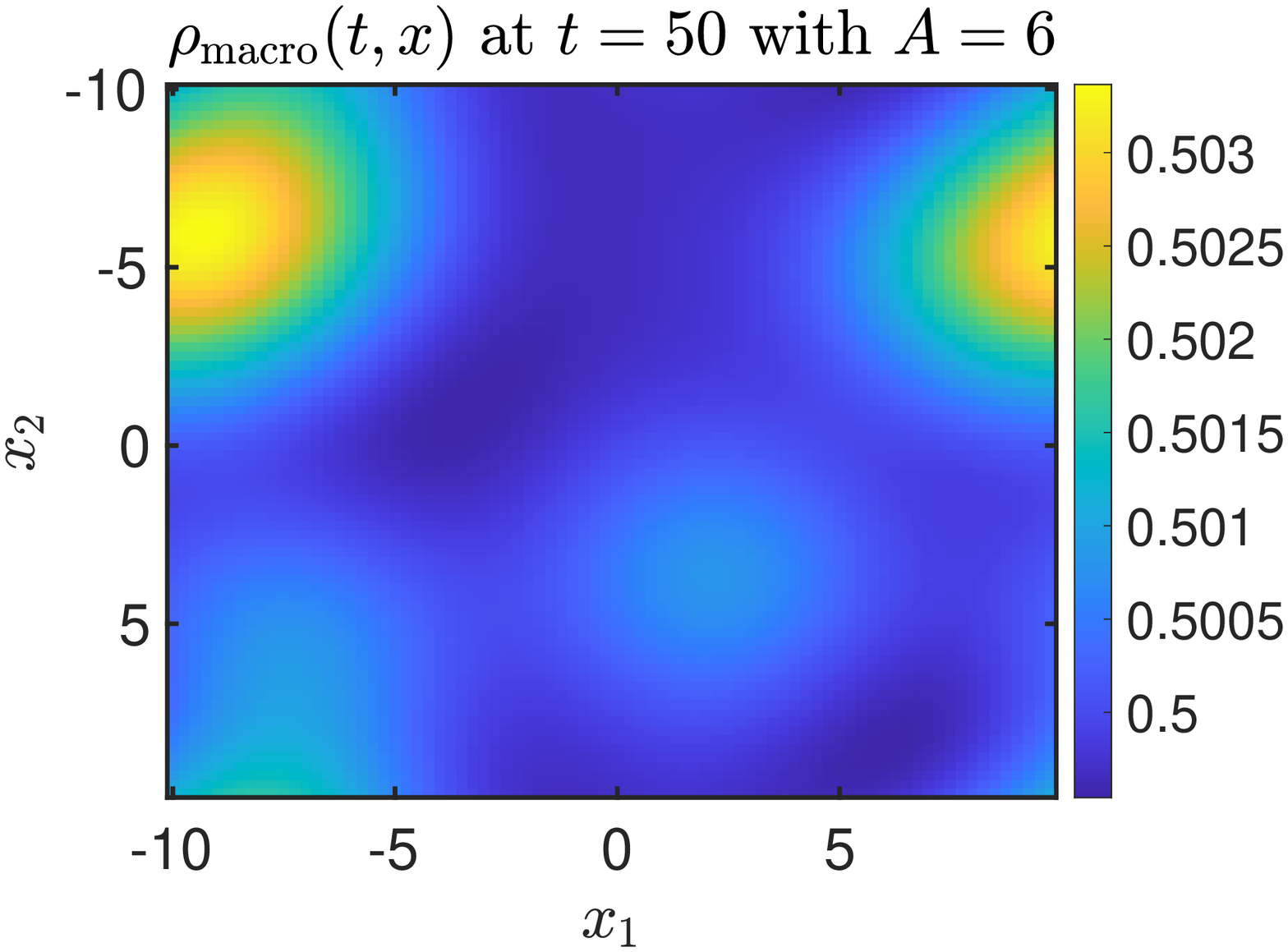}
\includegraphics[scale = 0.25]{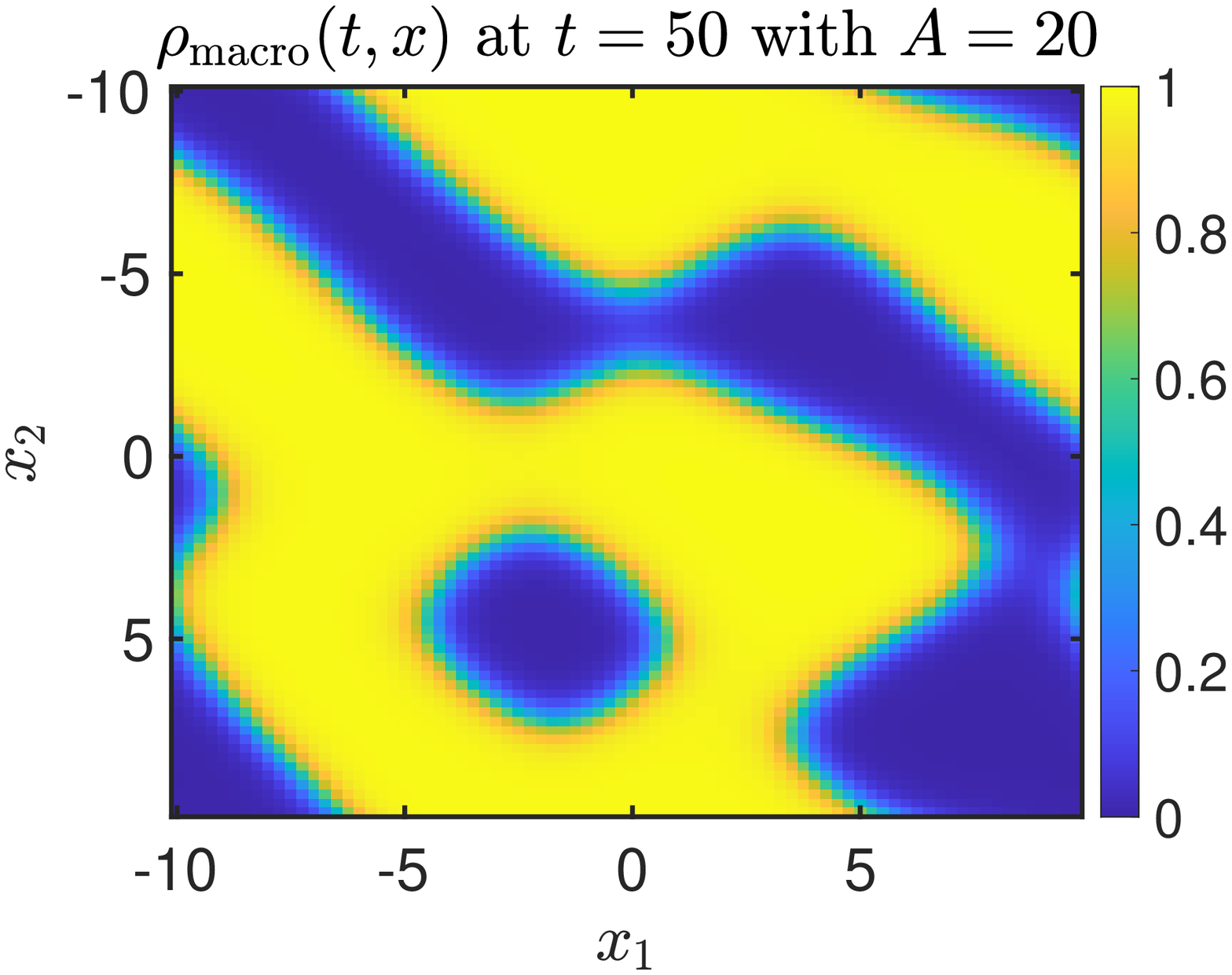}
\includegraphics[scale = 0.25]{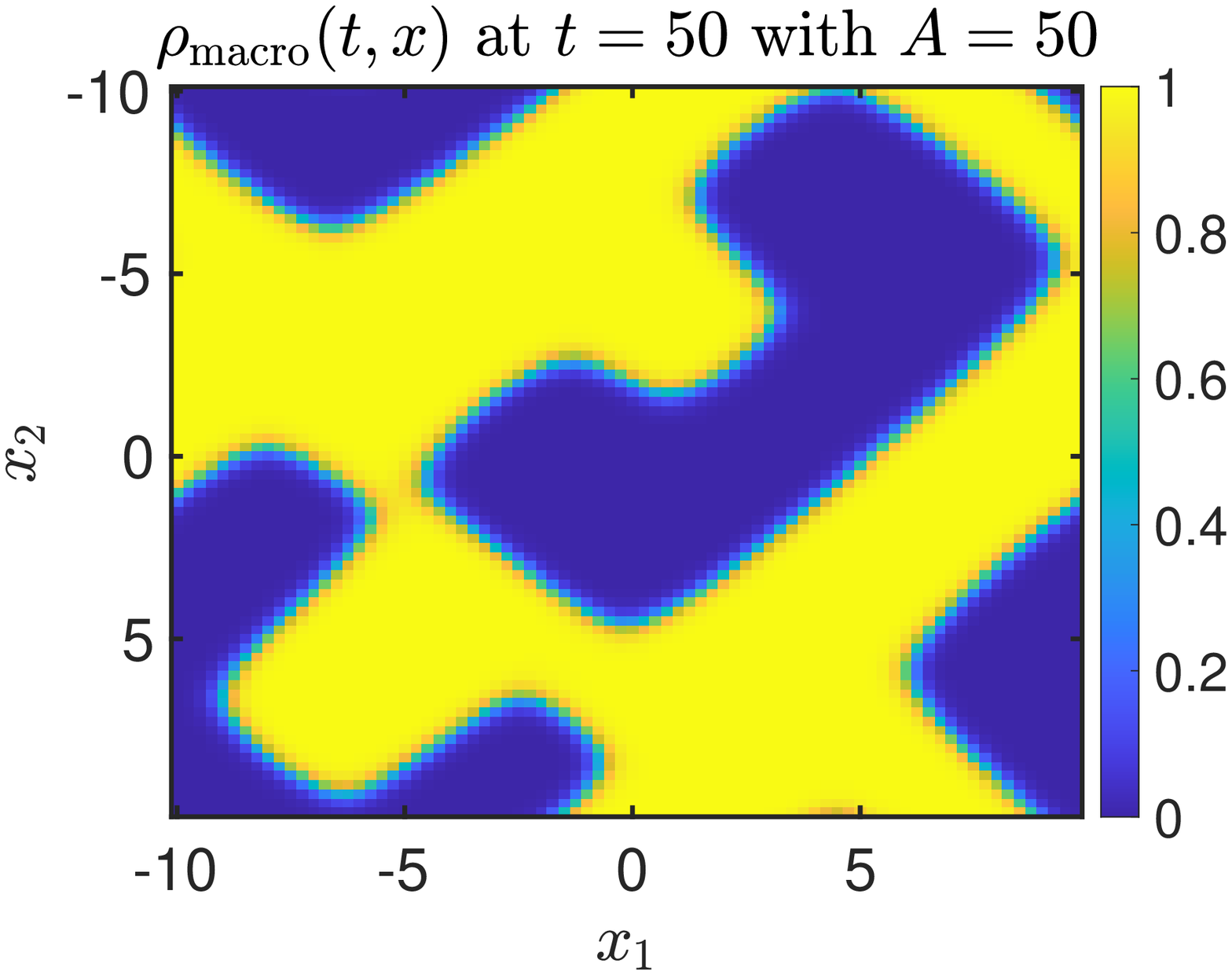}
}
\caption{Plots of density $\rho_\textnormal{macro}(t,\mathbf{x})$ at $t=5,\ 20,\ 50$ (from top to bottom) with $A=6,\ 20,\ 50$ (from left to right), respectively.}
\label{fig:2D_macro}
\end{figure}

\begin{figure}[tbhp]
\centerline{
\includegraphics[scale = 0.2]{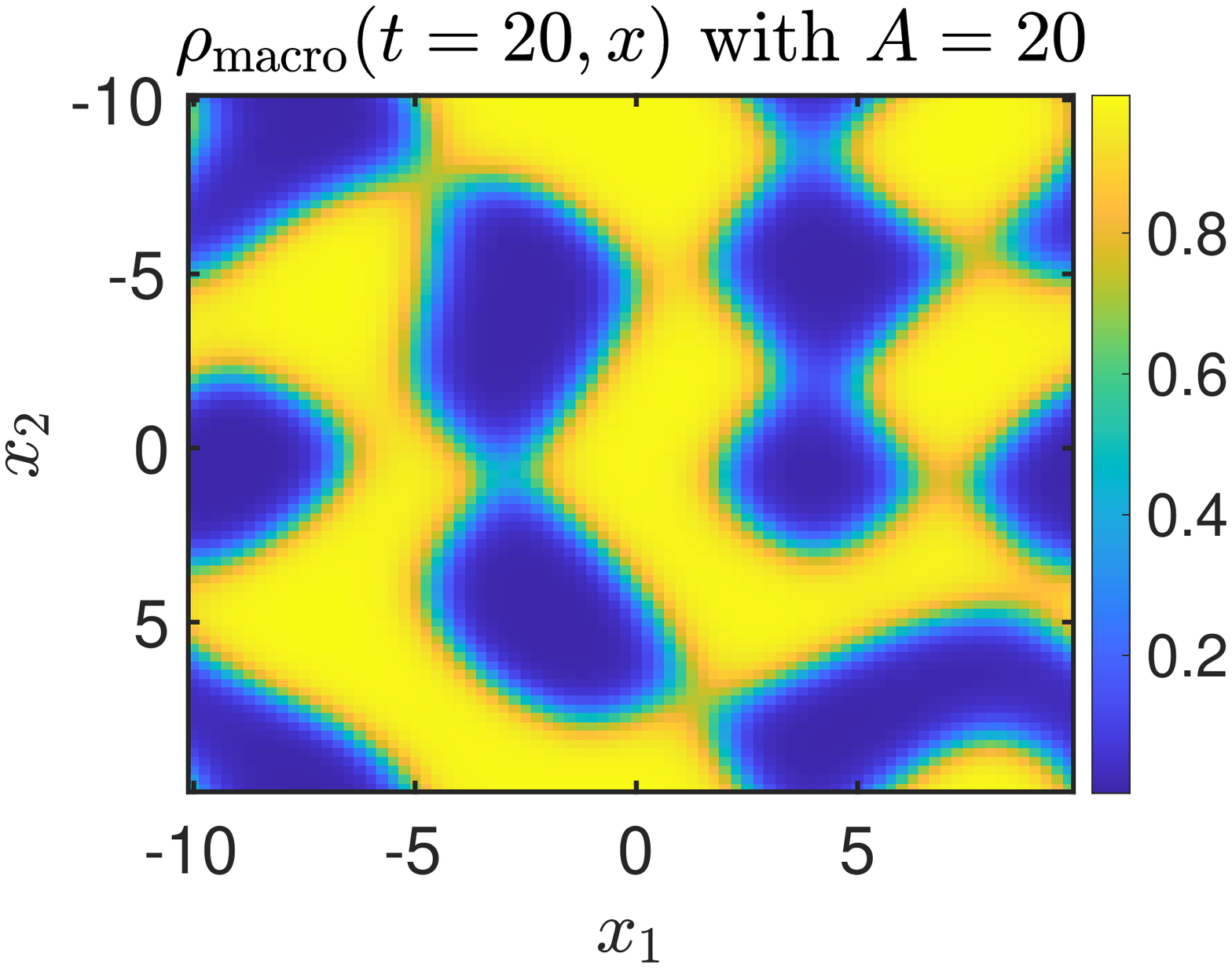}
\includegraphics[scale = 0.2]{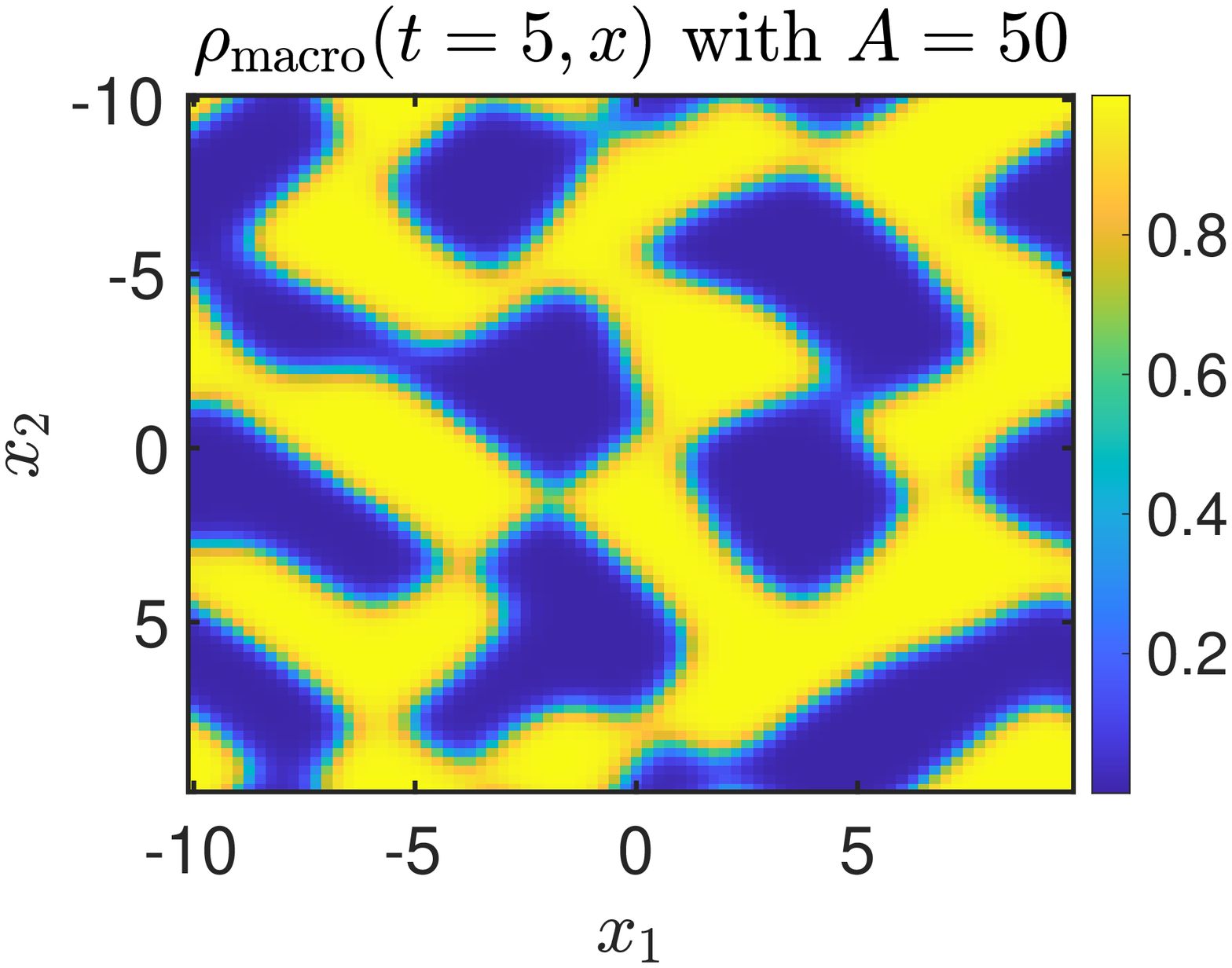}
\includegraphics[scale = 0.2]{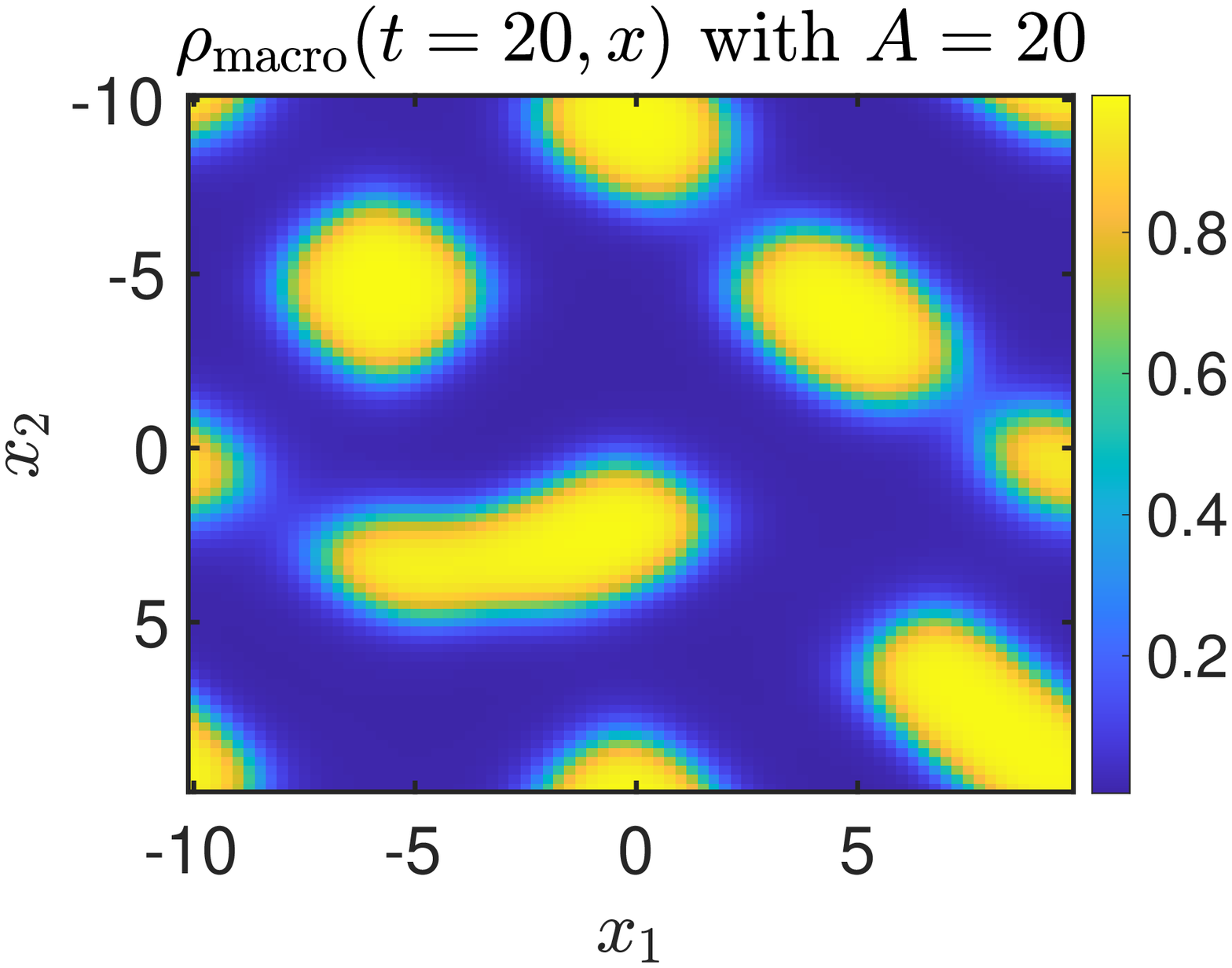}
\includegraphics[scale = 0.2]{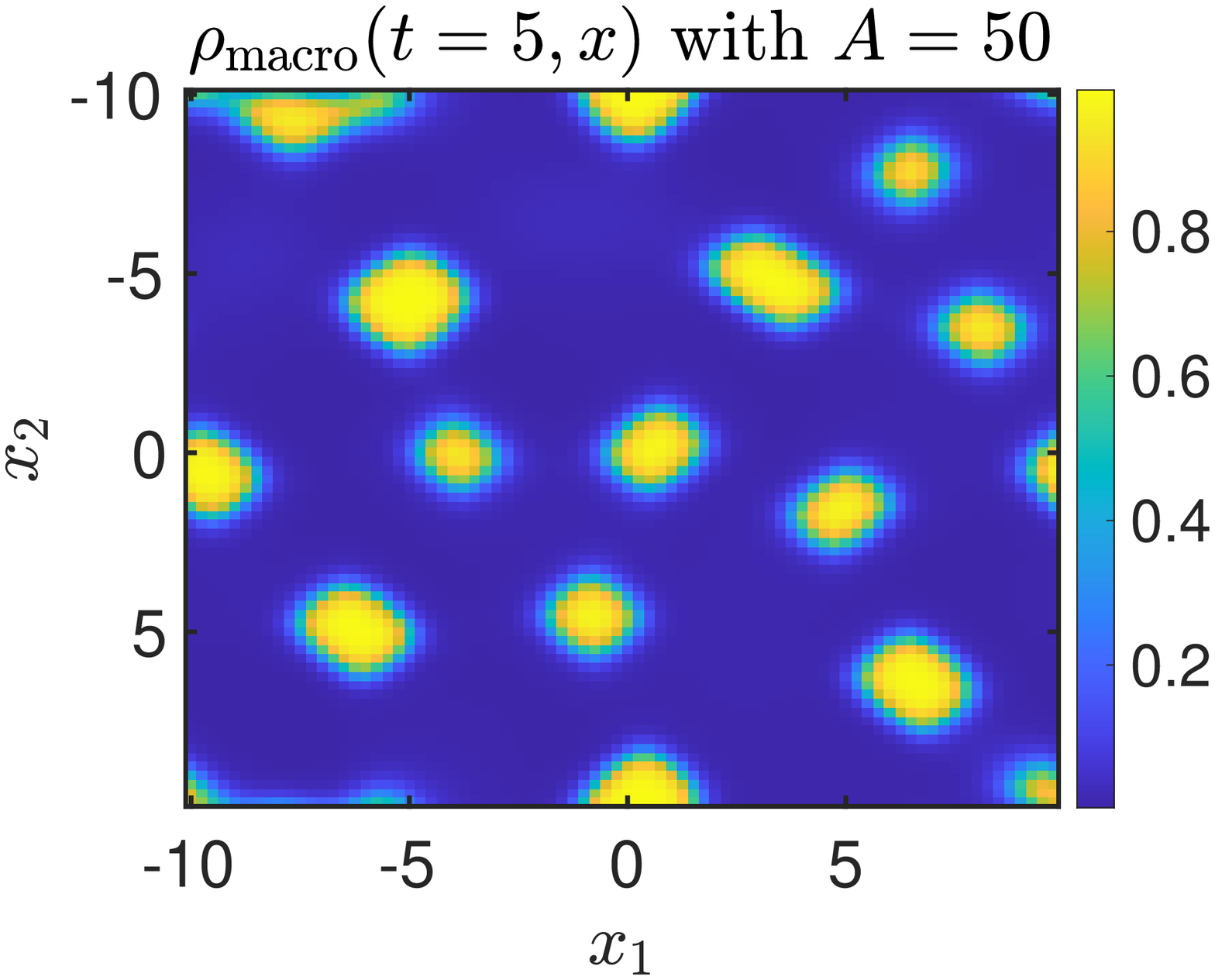}
}

\centerline{
\includegraphics[scale = 0.2]{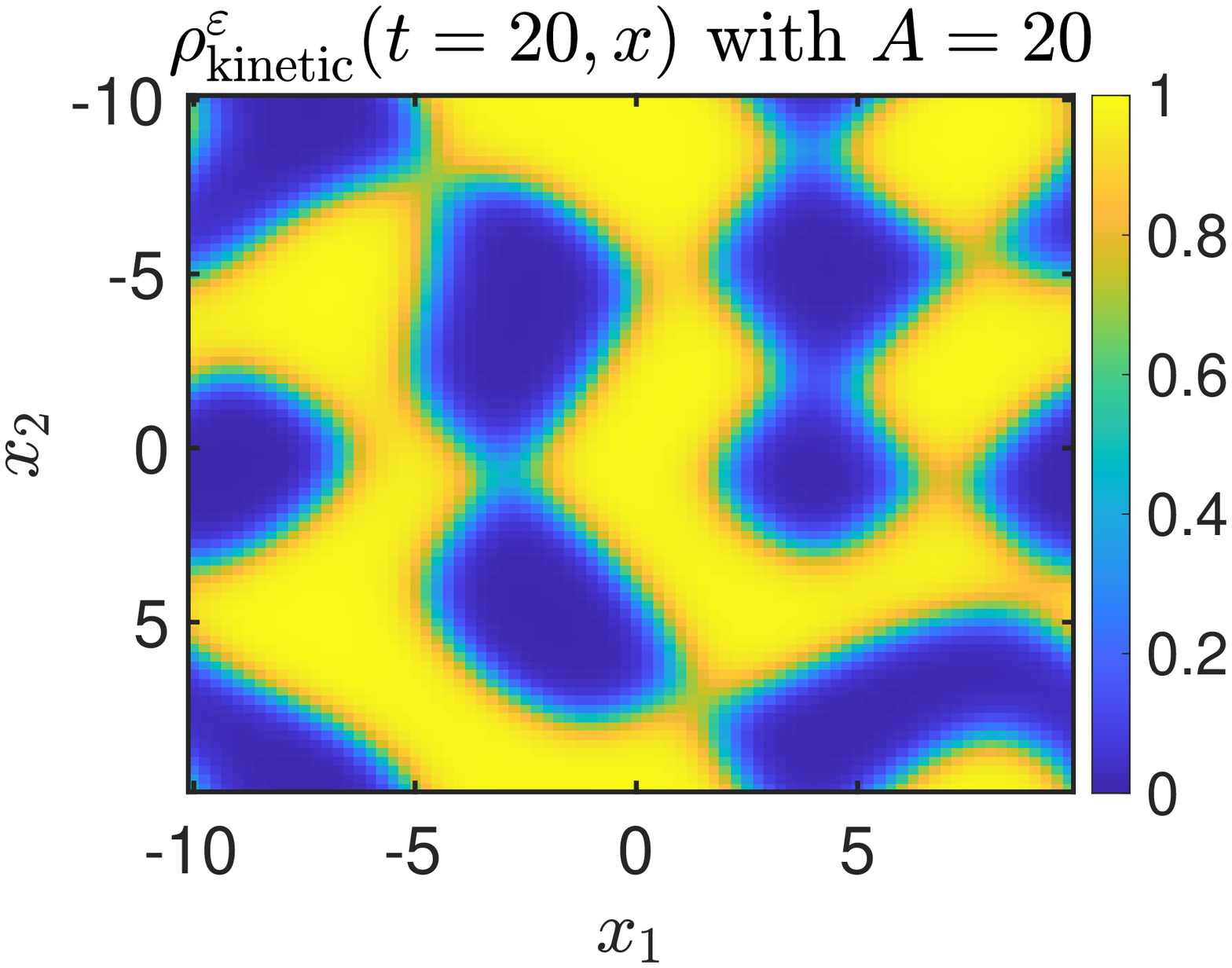}
\includegraphics[scale = 0.2]{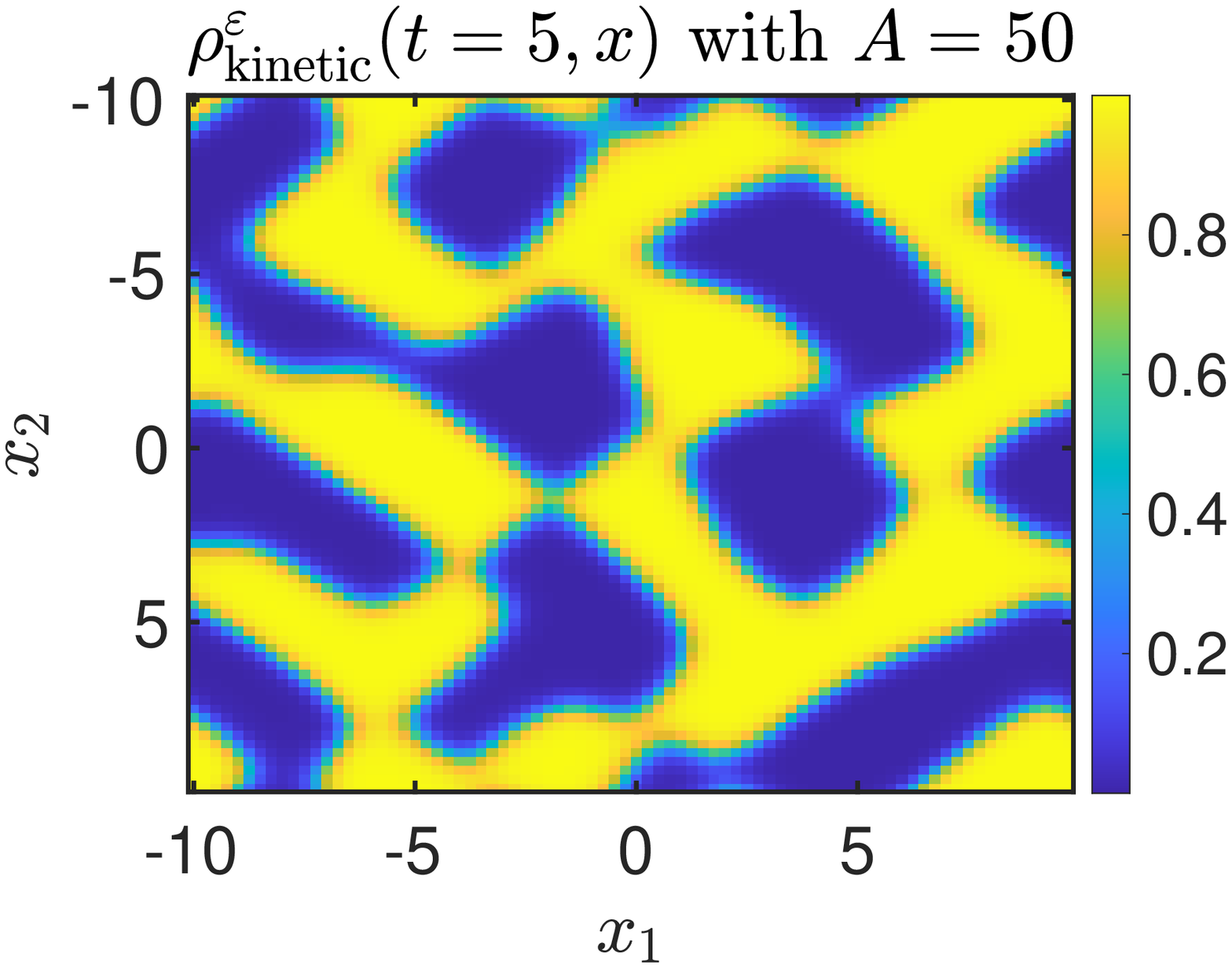}
\includegraphics[scale = 0.2]{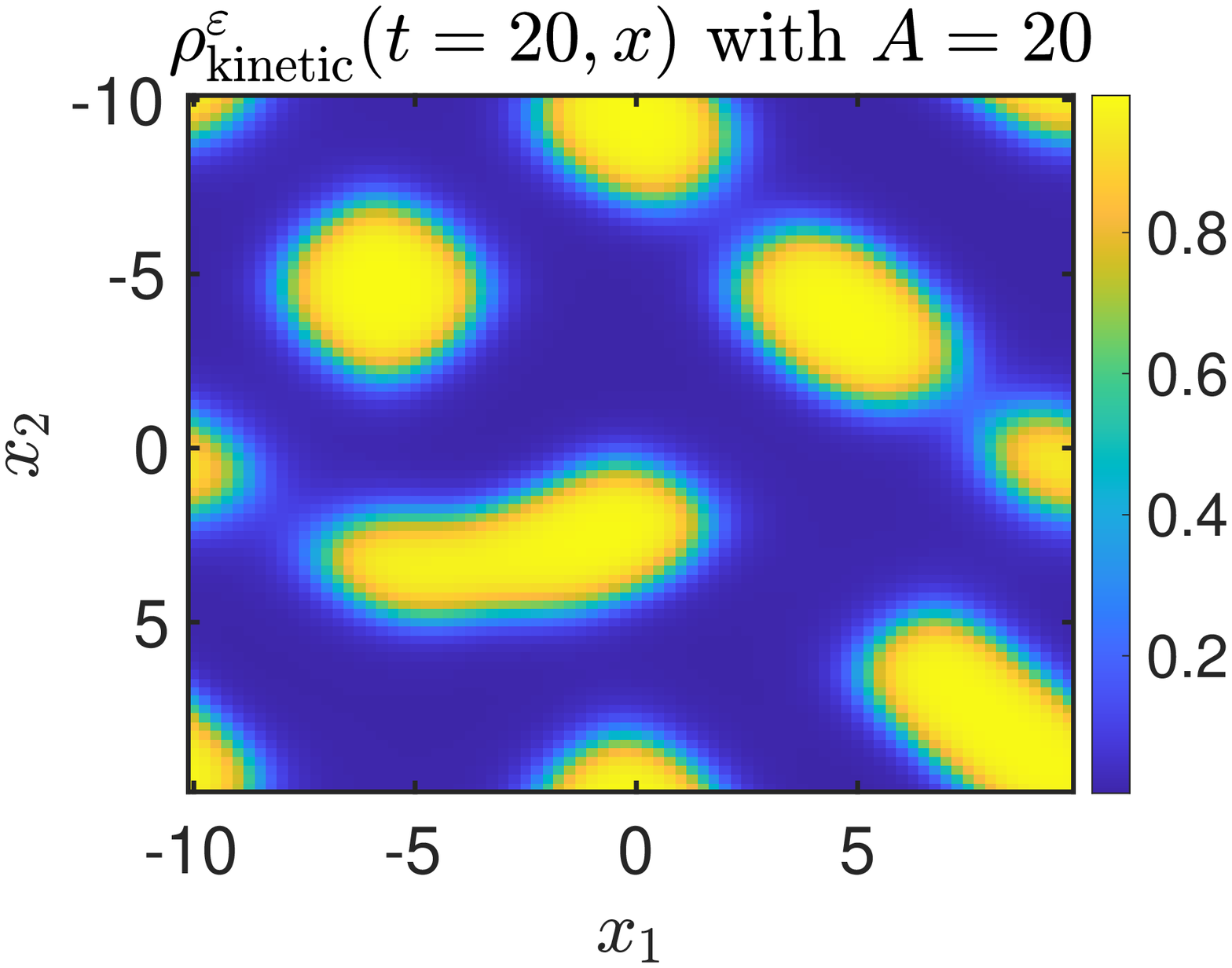}
\includegraphics[scale = 0.2]{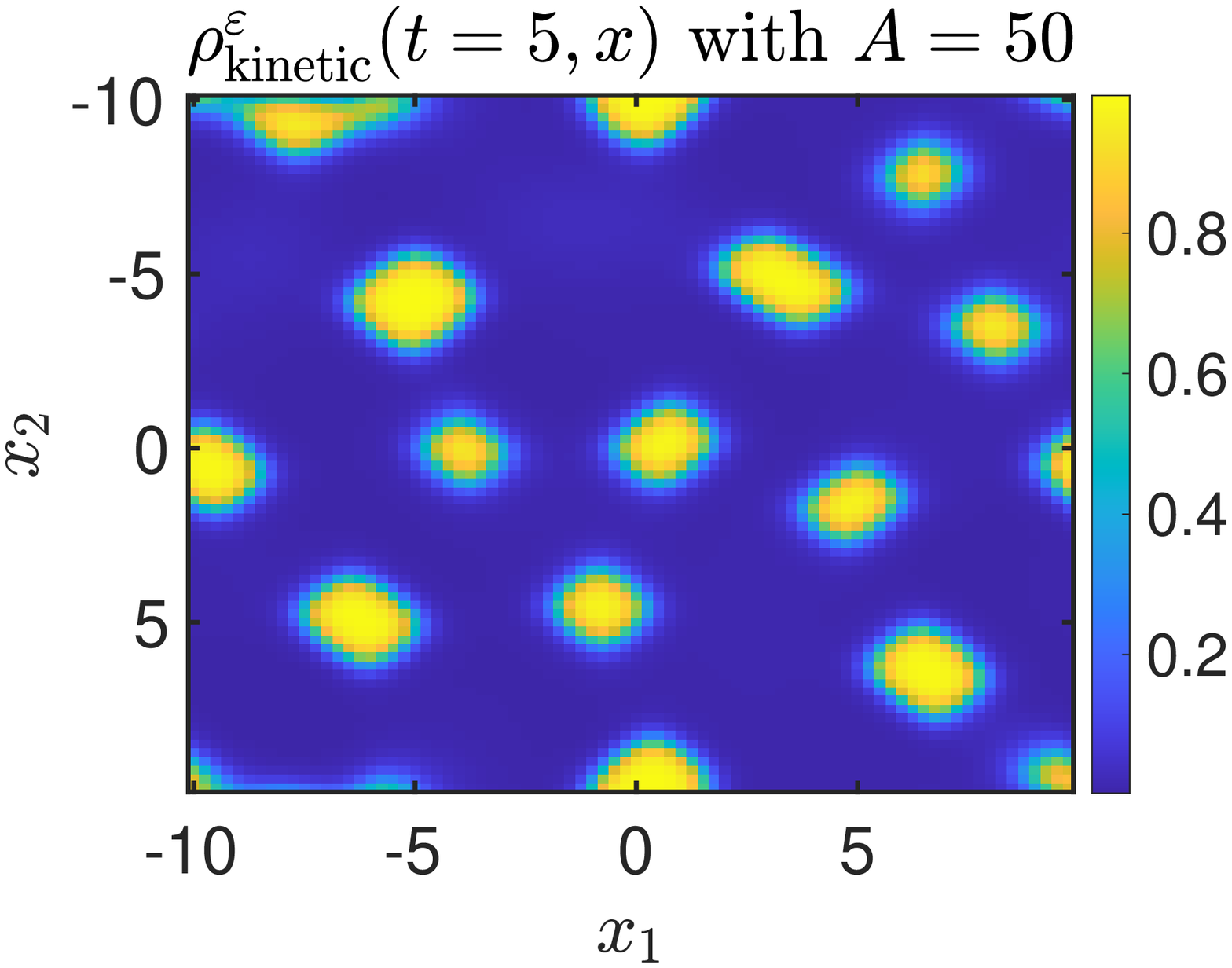}
}
\caption{Comparison between $\rho_\textnormal{macro}(t, \bx)$ (first row) and $\rho_\textnormal{kinetic}^{\varepsilon}(t, \bx)$ with $\varepsilon=10^{-2}$ (second row) for different initial conditions $\rho_0\approx 0.5$ (first two columns) and $\rho_0\approx0.1$ (last two columns) and different values of the chemotactic sensitivity $A=20$, (columns 1 and 3) and $A=50$ (columns 2 and 4). }
\label{fig:2D_compare}
\end{figure}

\begin{figure}[tbhp]
\centerline{
\includegraphics[scale = 0.3]{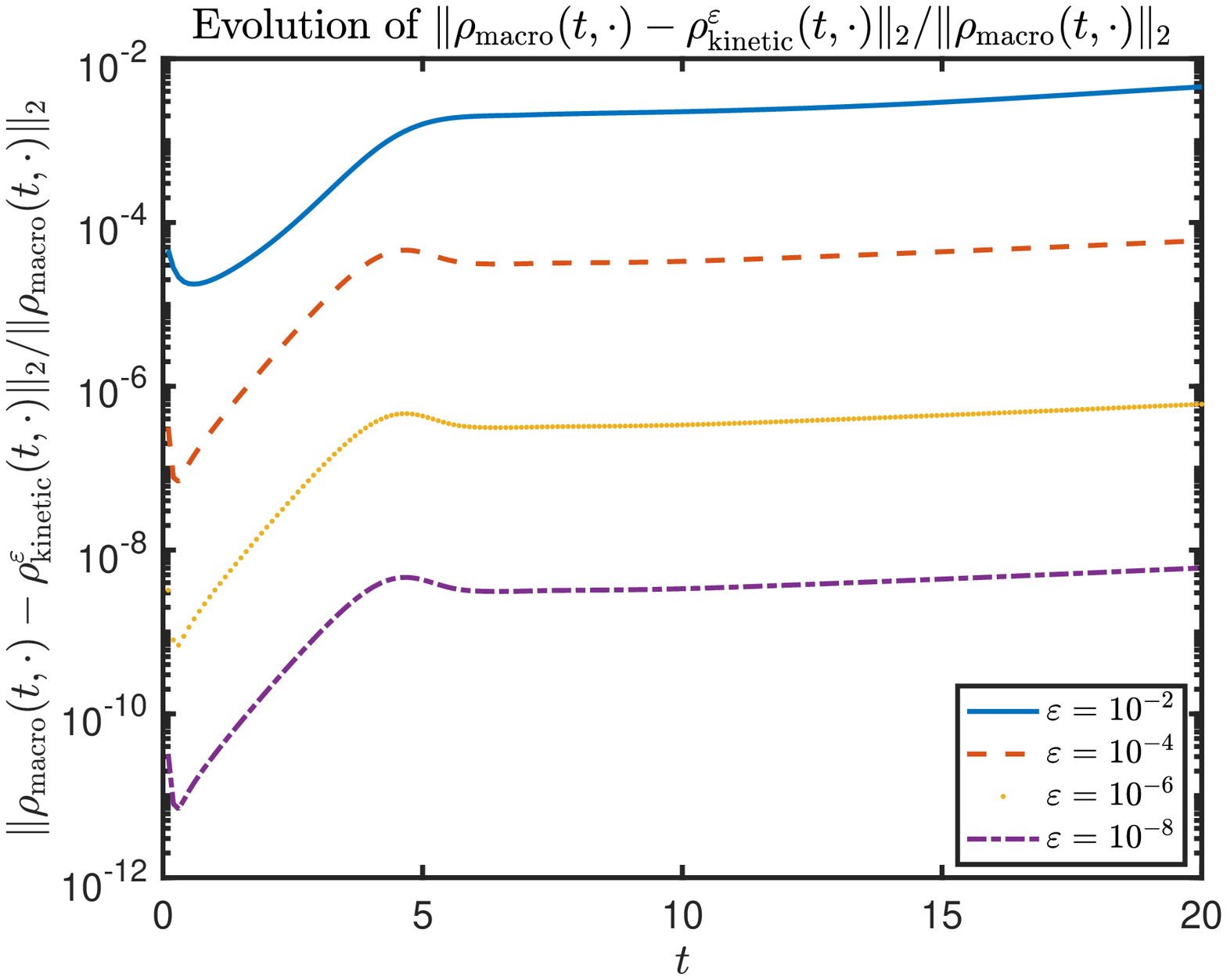}
\includegraphics[scale = 0.3]{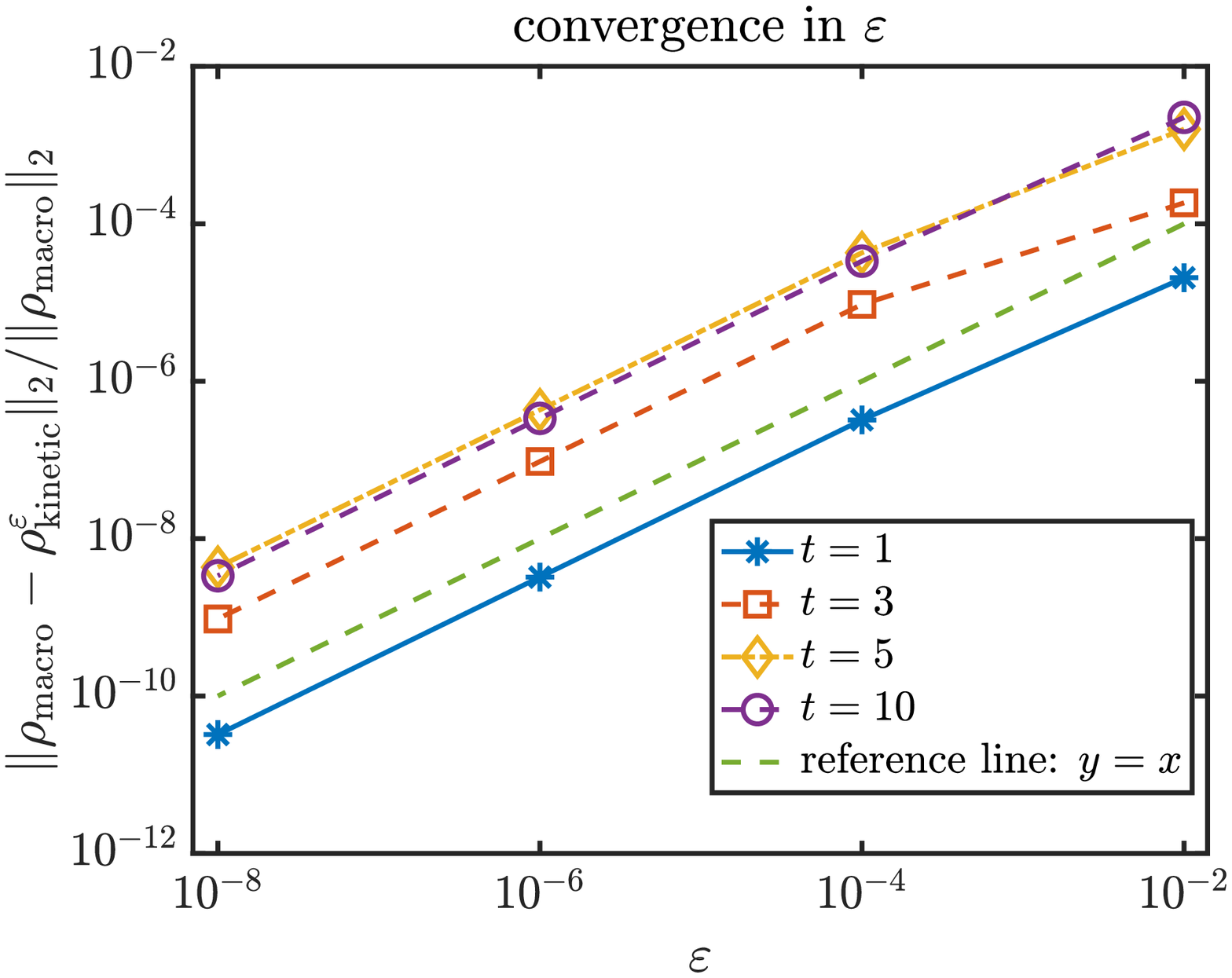}
}
\caption{{Left: Evolution of the relative $L_2$-error. Right: } Convergence of the relative $L_2$-error in $\e$ at $t=1,\ 3,\ 5,\ 10$ with $A=20$ and $\dt=10^{-2}$.}
\label{fig:2D_convergence}
\end{figure}

\section{Conclusions}

{In this paper we have derived a model for chemotaxis incorporating a density dependence in the chemotactic sensitivity function that takes into account the finite size of the cells and volume limitations. We showed, with formal arguments, that the macroscopic chemotactic system can be seen as the diffusion limit of a kinetic 'velocity-jump' model, provided that both the transport term and the turning operator are density dependent. This derivation provides a more direct interpretation of the diffusion tensor and chemotactic sensitivity in terms of more fundamental characteristics of the motion.

We further studied this macroscopic limit numerically using an asymptotic preserving finite difference scheme based on a micro-macro decomposition of the unknown in the sense of  \cite{LemouMieussens}, a projection technique to obtain
a coupled system of two evolution equations for the microscopic and macroscopic
components, and a suitable semi-implicit time discretization. The scheme was successfully extended to account for nonlinear terms by implicit-explicit discretization in an upwind manner, allowing for accurate approximations in the case of strong chemosensitivity. This scheme enabled us to explore numerically the different behaviours observed by the kinetic and macroscopic models in 1D and 2D, and we showed that both models are in good agreement as the diffusion scaling parameter becomes smaller. Moreover, the numerical simulations of the kinetic model
revealed the same pattern sizes as obtained with the macroscopic model and predicted theoretically, with very good precision as $\e$ goes to zero in the kinetic setting. It is noteworthy that both models also feature the same dynamics in time, with a slight delay in the macroscopic simulations compared to the kinetic dynamics. This could be due to the fact that the macroscopic model is obtained in a regime where there are many velocity jumps but small net displacements in one order of time. Therefore, in the macroscopic setting, each particle interacts with many more particles than in the kinetic model, which could result in a delay in the aggregation process. 

From the modelling perspective, it would be natural to extend the derivation to consider different turning kernels, to take into account cell-cell adhesion or nonlocal movement, for instance. The idea to construct the scheme could be generalized to include these cases, but we stress the fact that the detailed discretization is problem-dependent. 
Moreover, the rigorous derivation of volume-filling chemotactic equations from stochastic processes of interacting populations could be considered by adapting ideas from \cite{stevens2000derivation} for instance. 
}

\section*{Acknowledgement}
The authors wish to thank L. Almeida and K. J. Painter for helpful discussions and guidance.
DP was supported by Sorbonne Alliance University with an Emergence project MATHREGEN, grant number S29-05Z101. GER was partially supported by the Fondation Sciences Mathématiques de Paris (FSMP) and the Advanced Grant Nonlocal-CPD (Nonlocal PDEs for Complex Particle Dynamics: Phase Transitions, Patterns and Synchronization) of the European Research Council Executive Agency (ERC). XR was partially supported by the project MoGlimaging, Plan Cancer THE Call, from INSERM, France.

\appendix

\section{Energy dissipation in the macroscopic model} \label{subsec:energy}
{Under proper assumptions, the macroscopic volume-exclusion Keller-Segel model \eqref{eq: macroscopic equation1} can be proven to be energy dissipate, which will be a key feature to be preserved in numerical methods. 
Following a gradient flow approach to energy in the sense of \cite{calvez2006volume,carrillo2001entropy,de2018energy}, we start by defining $H(\rho) = \frac{D_0}{\beta}\ln\left(\frac{\rho}{q(\rho)}\right)$. }
Then the volume-exclusion Keller-Segel model \eqref{eq: macroscopic equation1} can be reformulated as 
\begin{equation}
    \partial_t\rho+\nabla\cdot(\chi(\rho)\rho(-H'(\rho)\nabla\rho+\nabla c))= 0\ , \label{eq: new version of macro}
\end{equation}
where $\chi(\rho)=\beta q(\rho), \, H'(\rho)=D_0\frac{q(\rho)-\rho q'(\rho)}{\beta q(\rho)\rho}$.
The energy functional of the model can be given by 
\begin{equation}\label{def:energy_macro}
    \mathcal{E}(t)=\int\Phi(\rho)\diff x-\frac{1}{2}\int\rho c\diff x\ ,
\end{equation}
where 
\begin{equation}\label{functionPhi}
    \Phi'(\rho)=H(\rho)\ .
\end{equation}

\begin{proposition}\label{prop:energy}
Suppose that $\rho(t,\mathbf{x})$ and $c(t,\mathbf{x})$ solves the macroscopic volume-exclusion Keller-Segel equation \eqref{eq: new version of macro} coupled with the equation
\begin{equation}\label{elliptic_c}
   \Delta c + \rho - c = 0\ , 
\end{equation}
we have
\begin{equation}
    \frac{\diff}{\diff t}\mathcal{E}(t)=-\int \rho\chi(\rho)|\nabla(H-c)|^2\diff x\leq 0\ ,
\end{equation}
where the energy functional $\mathcal{E}(t)$ is given by \eqref{def:energy_macro}.
\end{proposition}

\begin{proof}
 Multiplying \eqref{eq: new version of macro} by $H- c$, we get
\begin{align}
    (H-c)\partial_t\rho & =(H-c)\nabla\cdot(\chi(\rho)\rho\nabla(H-c))\nonumber\\ & = \frac{1}{2}\nabla\cdot\left( \chi(\rho)\rho\nabla(H-c)^2\right)-\chi(\rho)\rho|\nabla(H-c)|^2 \ ,\nonumber
\end{align}
where we used the relation $H'(\rho)\nabla\rho-\nabla c=\nabla(H-c)$.
Then
\begin{align*}
    \frac{\diff}{\diff t}\mathcal{E}(t) & = \frac{\diff }{\diff t}\int \left(\Phi-\frac{1}{2}\rho c \right)\diff x
    = \frac{\diff }{\diff t} \left[\int \left(\Phi-\frac{1}{2} c^2 -\frac{1}{2} |\nabla c|^2 \right)\diff x\right]\\
    & = \int\left[ H\partial_t\rho - c \left(\partial_t c - \Delta(\partial_t c)\right) \right] \diff x  = \int (H - c)\partial_t\rho \diff x \\
    &=-\int \rho\chi(\rho)|\nabla(H-c)|^2\diff x\leq 0\ .
\end{align*}
\end{proof}
\begin{remark}\label{rem 6}
Proposition \ref{prop:energy} can be generalized to the equation 
\begin{equation*}
    \partial_t\rho-\nabla\cdot(\chi(\rho)\rho(-H'(\rho)\nabla\rho+\nabla c))=r_0\rho\left(1 - \frac{\rho}{\rho_{\rm max}}\right)_+,
\end{equation*}
where a proliferation term is included satisfying $\rho_{\rm max} + \frac{\rho_{\rm max}}{\bar{\rho}} \le 1$. 
In fact, it can be checked that, when $\rho<\rho_{\rm max}$, we have $\rho < q(\rho)$ and thus $H<0$. 
The conclusion is then obvious. 
\end{remark}

\section{A finite difference scheme for the 2D kinetic model} \label{sec:appendix-FD-2D}
The finite difference scheme \eqref{eq:scheme} can be generalized to multi-dimensional problems where a tensor product grid is applied. 
Here we consider the 2D kinetic model with the special choice $\vec{\psi}_1 = \vec{\phi}(v_1,v_2) \cdot \nabla c$, where $\vec{\phi}(v_1,v_2) = (\phi_1(v_1,v_2), \phi_2(v_1,v_2))^T $. 

Denoting $\rho_{j_1,j_2}^n$, $c_{j_1,j_2}^n$,  $g_{j_1+\frac12,j_2,k_1,k_2}^{(1), n}$ and $g_{j_1,j_2+\frac12,k_1,k_2}^{(2), n}$ to be the numerical approximations of $\rho(t_n,x_{j_1}, x_{j_2})$, $c(t_n,x_{j_1}, x_{j_2})$, $g(t_n, x_{j_1+\frac12}, x_{j_2}, v_{k_1}, v_{k_2})$ and $g(t_n, x_{j_1}, x_{j_2+\frac12}, v_{k_1}, v_{k_2})$, respectively. 
The approximations of $\rho(t, \mathbf{x})$ at half grid points such as $(x_{j_1}, x_{j_2+\frac12})$ can be then easily approximated by the average $\bar{\rho}_{j_1, j_2+\frac12} := (\rho_{j_1, j_2} + \rho_{j_1, j_2+1})/2$. 
It is worth noticing that we used different notations for approximating $g(t_n, x_{j_1+\frac12}, x_{j_2}, *, *)$ and $g(t_n, x_{j_1}, x_{j_2+\frac12}, *, *)$ since different upwind discretizations will be used depending on whether the half grid is in $x_1$-direction or $x_2$-direction. 
An illustration of the grids in $\mathbf{x}$-space for computing $\rho(t,\mathbf{x})$ and $g(t,\mathbf{x},\mathbf{v})$ in 1D and 2D can be found in Figure \ref{fig:grid_illstration}.

\begin{figure}[tbhp]
\centerline{
\includegraphics[scale = 0.4]{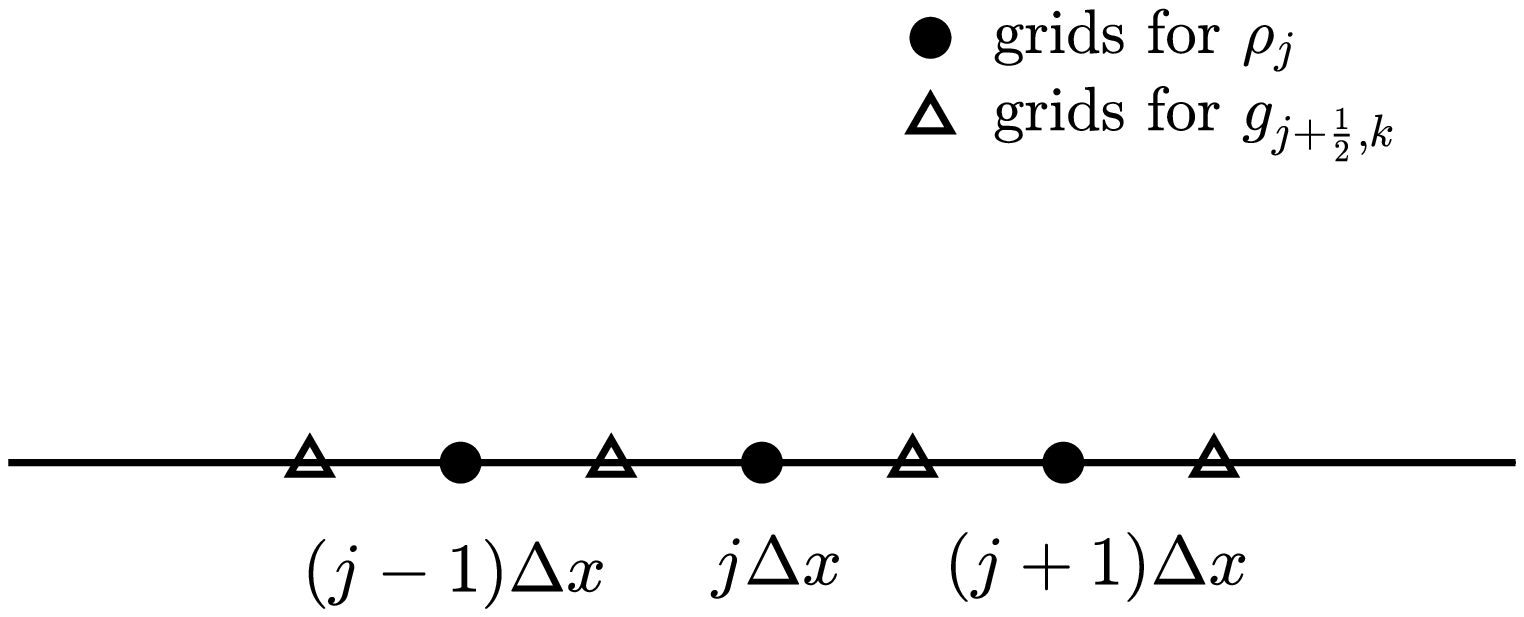}
\includegraphics[scale = 0.4]{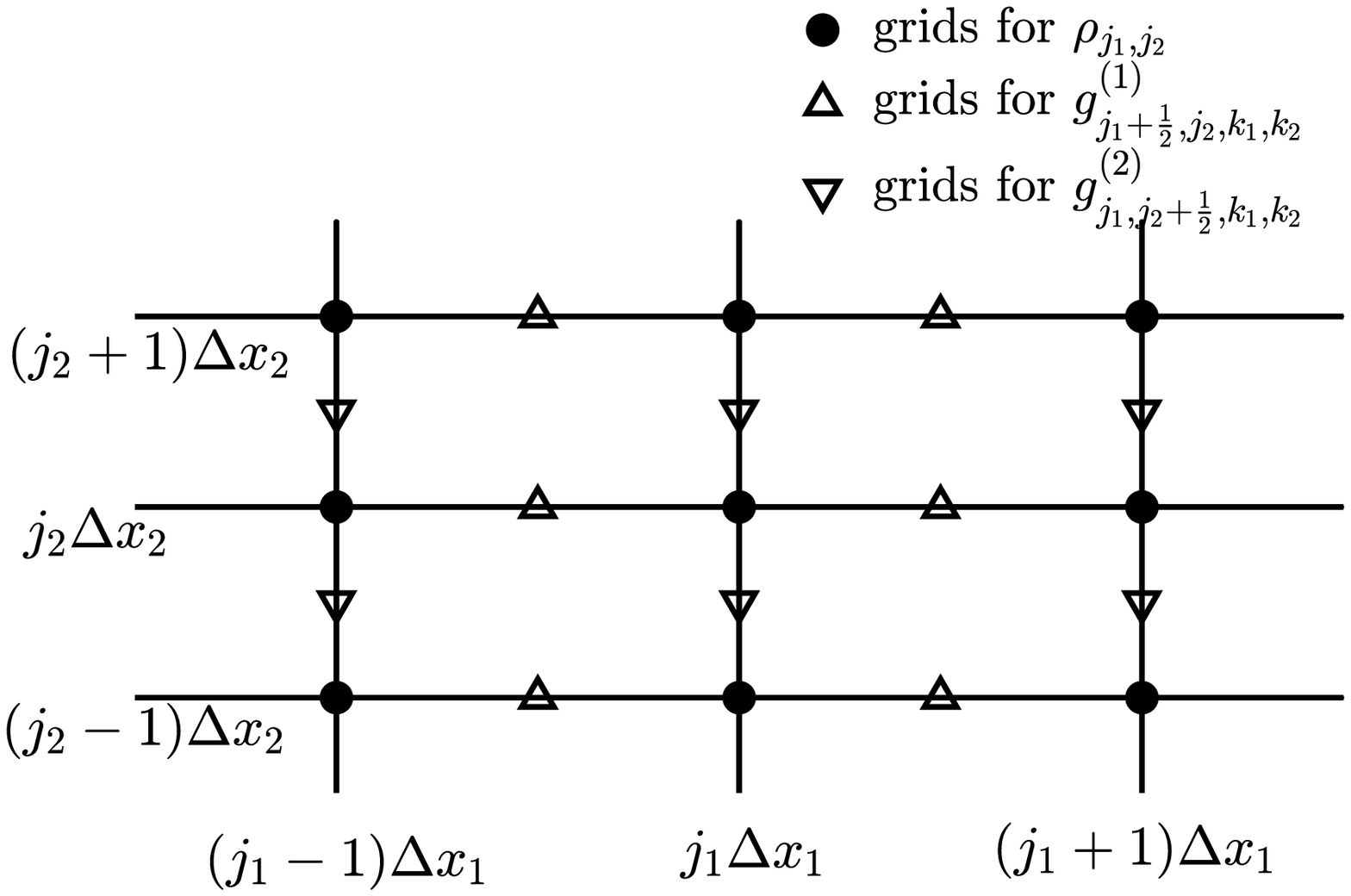}
}
\caption{Illustration of the grids in $\mathbf{x}$-space for computing $\rho(t,\mathbf{x})$ and $g(t,\mathbf{x},\mathbf{v})$ in 1D and 2D.}
\label{fig:grid_illstration}
\end{figure}

With the notations defined, the 2D kinetic model \eqref{eq:kinetic_approx} can be discretized as 
\begin{align}
\label{eq:scheme_2D}
    \begin{cases}
    & \delta_t^+ \rho_{j_1,j_2}^n +  
    \sum_{k_1,k_2} \left[v_{k_1}\delta_{x_1}(q( {\rho}_{*,j_2}^n) g_{*,j_2,k_1,k_2}^{(1), n+1})_{j_1} +  v_{k_2}\delta_{x_2}(q( {\rho}_{j_1,*}^n) g_{j_1,*,k_1,k_2}^{(2), n+1})_{j_2}\right]\Delta v_1 \Delta v_2\\
     & \quad + D_h \left[ \delta_{x_1}( {\rho}_{*,j_2}^n q'( {\rho}_{*,j_2}^n)\delta_{x_1} \rho_{*,j_2}^{n+1})_{j_1} + \delta_{x_2}( {\rho}_{j_1,*}^n q'( {\rho}_{j_1,*}^n)\delta_{x_2} \rho_{j_1,*}^{n+1})_{j_2} \right]
    = r_0 \rho_{j_1,j_2}^{n}\left(1 - \frac{\rho_{j_1,j_2}^n}{\rho_{\rm max}}\right)_+,\\
    &\delta_t^+ g_{j_1+\frac12,j_2,k_1, k_2}^{(1),n} + \frac{1}{\varepsilon}(I-\Pi_h) K_{j_1+\frac12,j_2,k_1, k_2}^{(1),n} 
    = \frac{1}{\varepsilon^2}S_{j_1+\frac12,j_2,k_1,k_2}^{(1), n, n+1} + r_0 g_{j_1+\frac12,j_2,k_1,k_2}^{(1),n}\left(1 - \frac{\rho_{j_1+\frac12,j_2}^n}{\rho_{\rm max}}\right)_+,\\
        &\delta_t^+ g_{j_1,j_2+\frac12,k_1, k_2}^{(2),n} + \frac{1}{\varepsilon}(I-\Pi_h) K_{j_1,j_2+\frac12,k_1, k_2}^{(2),n} 
    = \frac{1}{\varepsilon^2}S_{j_1,j_2+\frac12,k_1,k_2}^{(2), n, n+1} + r_0 g_{j_1,j_2+\frac12,k_1,k_2}^{(2),n}\left(1 - \frac{\rho_{j_1,j_2+\frac12}^n}{\rho_{\rm max}}\right)_+,\\
    & (\delta_{x_1}^2 + \delta_{x_2}^2) c_{j_1,j_2}^{n+1} + \rho_{j_1,j_2}^{n+1} - c_{j_1,j_2}^{n+1} = 0,
\end{cases}
\end{align}
where  $\Pi_h$ is the discrete projection operator defined as 
\begin{equation*}
    \Pi_h \eta_{j_1,j_2,k_1,k_2}^n = \sum_{k_1,k_2}\eta_{j_1,j_2,k_1,k_2}^n \psi_{0}(v_{k_1},v_{k_2})\Delta v_1 \Delta v_2
\end{equation*} 
for some general function $\eta(t,\mathbf{x},\mathbf{v})$ and 
\begin{align*}
    & K_{j_1+\frac12,j_2,k_1, k_2}^{(1),n}  = v_{k_1}^+\delta_{x_1}(q( {\rho}_{*,j_2}^n) g_{*,j_2,k_1,k_2}^{(1),n})_{j_1} - v_{k_1}^-\delta_{x_1}(q( {\rho}_{*,j_2}^n) g_{*,j_2,k_1,k_2}^{(1),n})_{j_1+1}\\
    & \qquad + v_{k_2}^+\delta_{x_2}(q( {\rho}_{j_1+\frac12,*}^n) g_{j_1+\frac12, *,k_1,k_2}^{(1),n})_{j_2-\frac12} - v_{k_2}^-\delta_{x_2}(q( {\rho}_{j_1+\frac12,*}^n) g_{j_1+\frac12,*,k_1,k_2}^{(1),n})_{j_2+\frac12} \\
    &\qquad + \psi_{0}(v_{k_1}, v_{k_2})\delta_{x_1}\left\{ {\rho}_{*,j_2}^n q'( {\rho}_{*,j_2}^n)\left[v_{k_1}^2\delta_{x_1}\rho_{*,j_2}^n + v_{k_1}v_{k_2}(\delta_{x_2}\rho_{*,**}^n)_{j_2} \right]\right\}_{j_1+\frac12}\ , \\
    &\qquad + \psi_{0}(v_{k_1}, v_{k_2})\delta_{x_2}\left\{ {\rho}_{j_1+\frac12,*}^n q'( {\rho}_{j_1+\frac12,*}^n)\left[v_{k_1}v_{k_2}(\delta_{x_1}\rho_{**,*}^n)_{j_1+\frac12} + v_{k_2}^2 \delta_{x_2}\rho_{j_1+\frac12,*}^n\right]\right\}_{j_2}\ , \\
    & K_{j_1,j_2+\frac12,k_1, k_2}^{(2),n}  = v_{k_1}^+\delta_{x_1}(q( {\rho}_{*,j_2+\frac12}^n) g_{*,j_2+\frac12,k_1,k_2}^{(1),n})_{j_1-\frac12} - v_{k_1}^-\delta_{x_1}(q( {\rho}_{*,j_2+\frac12}^n) g_{*,j_2+\frac12,k_1,k_2}^{(1),n})_{j_1+\frac12}\\
    & \qquad + v_{k_2}^+\delta_{x_2}(q( {\rho}_{j_1,*}^n) g_{j_1, *,k_1,k_2}^{(1),n})_{j_2} - v_{k_2}^-\delta_{x_2}(q( {\rho}_{j_1,*}^n) g_{j_1,*,k_1,k_2}^{(1),n})_{j_2+1} \\
    &\qquad + \psi_{0}(v_{k_1}, v_{k_2})\delta_{x_1}\left\{ {\rho}_{*,j_2+\frac12}^n q'( {\rho}_{*,j_2+\frac12}^n)\left[v_{k_1}^2\delta_{x_1}\rho_{*,j_2+\frac12}^n + v_{k_1}v_{k_2}(\delta_{x_2}\rho_{*,**}^n)_{j_2+\frac12} \right]\right\}_{j_1}\ , \\
    &\qquad + \psi_{0}(v_{k_1}, v_{k_2})\delta_{x_2}\left\{ {\rho}_{j_1,*}^n q'( {\rho}_{j_1,*}^n)\left[v_{k_1}v_{k_2}(\delta_{x_1}\rho_{**,*}^n)_{j_1} + v_{k_2}^2 \delta_{x_2}\rho_{j_1,*}^n \right]\right\}_{j_2+\frac12}\ , \\
    & S_{j_1+\frac12, j_2, k_1, k_2}^{(1), n, n+1} = -\psi_{0}(v_{k_1}, v_{k_2})q( {\rho}_{j_1+\frac12,j_2}^n)  \left[v_{k_1} \delta_{x_1} \rho_{j_1+\frac12, j_2}^{n+1} + v_{k_2} \delta_{x_2} \rho_{j_1+\frac12, j_2}^{n+1} \right] \\
    & \qquad  + \phi_{1}(v_{k_1}, v_{k_2})\delta_{x_1} c_{j_1+\frac12,j_2}^n \Phi^{(1),n+1,n}_{j_1+\frac12,j_2}+\phi_{2}(v_{k_1}, v_{k_2}) \delta_{x_2} c_{j_1+\frac12,j_2}^n
    \Phi^{(2),n+1,n}_{j_1+\frac12,j_2}
    -q( {\rho}_{j_1+\frac12,j_2}^n) g_{j_1+\frac12,j_2,k_1,k_2}^{(1), n+1}\ ,\\
    & S_{j_1, j_2+\frac12, k_1, k_2}^{(2), n, n+1} = -\psi_{0}(v_{k_1}, v_{k_2})q( {\rho}_{j_1,j_2+\frac12}^n)  \left[v_{k_1} \delta_{x_1} \rho_{j_1, j_2+\frac12}^{n+1} + v_{k_2} \delta_{x_2} \rho_{j_1, j_2+\frac12}^{n+1} \right] \\
    & \qquad  + \phi_{1}(v_{k_1}, v_{k_2})\delta_{x_1} c_{j_1,j_2+\frac12}^n\Phi^{(1),n+1,n}_{j_1,j_2+\frac12} +\phi_{2}(v_{k_1}, v_{k_2}) \delta_{x_2} c_{j_1,j_2+\frac12}^n 
    \Phi^{(2),n+1,n}_{j_1,j_2+\frac12}
    -q( {\rho}_{j_1,j_2+\frac12}^n) g_{j_1,j_2+\frac12,k_1,k_2}^{(2), n+1}\ ,
\end{align*}
where, as in the 1D case, both $\Phi^{(1), n_1,n}$ and $\Phi^{(1), n_1,n}$ are upwind approximations of $q(\rho)\rho$ at $t=t_n$ and defined as 
\begin{align*}
    & \Phi^{(1), n_1,n}_{j_1, j_2} = 
	\begin{cases}
		\rho_{j_1 - \frac12,j_2}^{n_1} q(\rho_{j_1+\frac12,j_2}^n)\ , \quad & \text{ if } \delta_{x_1} c_{j_1,j_2}^n \ge 0\ , \\
		\rho_{j_1+\frac12,j_2}^{n_1} q(\rho_{j_1-\frac12,j_2}^n)\ , \quad & \text{ if } \delta_{x_1} c_{j_1,j_2}^n <0\ , 
	\end{cases}\\
	& \Phi^{(2), n_1,n}_{j_1, j_2} = 
	\begin{cases}
		\rho_{j_1,j_2 - \frac12}^{n_1} q(\rho_{j_1,j_2+\frac12}^n)\ , \quad & \text{ if } \delta_{x_2} c_{j_1,j_2}^n \ge 0\ , \\
		\rho_{j_1,j_2+\frac12}^{n_1} q(\rho_{j_1,j_2-\frac12}^n)\ , \quad & \text{ if } \delta_{x_2} c_{j_1,j_2}^n <0\ ,
	\end{cases}
\end{align*}

As for the 1D case, we can formally prove the asymptotic preserving property of the 2D scheme \eqref{eq:scheme_2D} in a similar way. 
In fact, when $\varepsilon\to 0$, we expect that  
\begin{equation*}
	S_{j_1+\frac12, j_2, k_1, k_2}^{(1), n, n+1} = 0\ , \qquad S_{j_1, j_2+\frac12, k_1, k_2}^{(2), n, n+1} = 0,
\end{equation*}
from where we get
\begin{align*}
  & q( {\rho}_{j_1+\frac12,j_2}^n) g_{j_1+\frac12,j_2,k_1,k_2}^{(1), n+1} = -\psi_{0}(v_{k_1}, v_{k_2})q( {\rho}_{j_1+\frac12,j_2}^n)  \left[v_{k_1} \delta_{x_1} \rho_{j_1+\frac12, j_2}^{n+1} + v_{k_2} \delta_{x_2} \rho_{j_1+\frac12, j_2}^{n+1} \right] \nonumber \\
        & \qquad  + \phi_{1}(v_{k_1}, v_{k_2})\delta_{x_1} c_{j_1+\frac12,j_2}^n \Phi^{(1),n+1,n}_{j_1+\frac12,j_2}+\phi_{2}(v_{k_1}, v_{k_2}) \delta_{x_2} c_{j_1+\frac12,j_2}^n
        \Phi^{(2),n+1,n}_{j_1+\frac12,j_2},\\
	&q( {\rho}_{j_1,j_2+\frac12}^n) g_{j_1,j_2+\frac12,k_1,k_2}^{(2), n+1} = -\psi_{0}(v_{k_1}, v_{k_2})q( {\rho}_{j_1,j_2+\frac12}^n)  \left[v_{k_1} \delta_{x_1} \rho_{j_1, j_2+\frac12}^{n+1} + v_{k_2} \delta_{x_2} \rho_{j_1, j_2+\frac12}^{n+1} \right] \nonumber\\
    & \qquad  + \phi_{1}(v_{k_1}, v_{k_2})\delta_{x_1} c_{j_1,j_2+\frac12}^n\Phi^{(1),n+1,n}_{j_1,j_2+\frac12} +\phi_{2}(v_{k_1}, v_{k_2}) \delta_{x_2} c_{j_1,j_2+\frac12}^n 
    \Phi^{(2),n+1,n}_{j_1,j_2+\frac12}\ .
\end{align*}
Substituting into the first equation in \eqref{eq:scheme_2D} and using the fact that
\begin{align*}
	&\sum_{k_1,k_2} v_{k_i}v_{k_j} \psi_0(v_{k_1},v_{k_2}) \Delta v_1 \Delta v_2 = D_h \delta_{i,j},  
	&\sum_{k_1,k_2} v_{k_i} \phi_j(v_{k_1},v_{k_2}) \Delta v_1 \Delta v_2 = A\delta_{i,j},  \quad i,j = 1,2,                   
\end{align*}
we recover the finite difference scheme for the macro model 
\begin{align*}
	&\delta_t^+ \rho_{j_1,j_2}^n   
    	- D_h \delta_{x_1} \left[ (q(\rho_{*,j_2}^n) - {\rho}_{*,j_2}^n q'( {\rho}_{*,j_2}^n)) \delta_{x_1} \rho_{*,j_2}^{n+1}\right]_{j_1} 
	- D_h \delta_{x_2} \left[ (q(\rho_{j_1,*}^n) -   {\rho}_{j_1,*}^n q'( {\rho}_{j_1,*}^n))\delta_{x_2} \rho_{j_1, *}^{n+1} \right]_{j_2} 	\nonumber \\
	 & \quad + A\delta_{x_1}\left[\Phi_{*,j_2}^{(1),n+1,n}\delta_{x_1} c_{*,j_2}^n\right]_{j_1} 
	 + A\delta_{x_2}\left[\Phi_{j_1,*}^{(1),n+1,n}\delta_{x_2} c_{j_1,*}^n\right]_{j_2}
    = r_0 \rho_{j_1,j_2}^{n}\left(1 - \frac{\rho_{j_1,j_2}^n}{\rho_{\rm max}}\right)_+\ .
\end{align*}

\bibliographystyle{abbrv}
\bibliography{AP_method}

\end{document}